\documentclass[11pt]{amsart}
\usepackage{amssymb}
\usepackage{latexsym}
\usepackage{graphicx}
\usepackage{subfigure}
\usepackage{tikz}
\usepackage[linktocpage=true]{hyperref}
\usepackage{ulem}
\usepackage[left=3.8cm, right=3.8cm, top=3cm, bottom=3cm]{geometry}
\usepackage{color}

\hypersetup{ colorlinks   = true, 
urlcolor  = blue, 
linkcolor    = blue, 
citecolor   = blue 
}

\def\a{{\alpha}}

\def\beq{\begin{equation}}
\def\eeq{\end{equation}}

\def\tF{\widetilde{F}}

\def\a{\alpha}

\newcommand{\Z}{{\mathbb Z}}
\newcommand{\R}{{\mathbb R}}
\newcommand{\Q}{{\mathbb Q}}
\newcommand{\C}{{\mathbb C}}

\newcommand{\T}{{\mathbb T}}

\newcommand{\N}{{\mathbb N}}
\newcommand{\PP}{{\mathbb P}}
\newcommand{\HH}{{\mathbb H}}

\newcommand{\CC}{{\mathcal C}}

\newcommand{\CB}{{\mathcal B}}
\newcommand{\CH}{{\mathcal H}}

\newcommand{\CL}{{\mathcal L}}
\newcommand{\CM}{{\mathcal M}}
\newcommand{\CN}{{\mathcal N}}

\newcommand{\CP}{{\mathcal P}}
\newcommand{\CS}{{\mathcal S}}

\newcommand{\CU}{{\mathcal U}}
\newcommand{\CK}{{\mathcal K}}

\newcommand{\CJ}{{\mathcal J}}

\newtheorem{theorem}{Theorem}[section]
\newtheorem{remark}{Remark}[section]
\newtheorem{lemma}{Lemma}[section]
\newtheorem{defi}{Definition}[section]
\newtheorem{prop}{Proposition}[section]
\newtheorem{corollary}{Corollary}[section]
\newtheorem{claim}{Claim}[section]

\newcommand{\la}{\langle}
\newcommand{\ra}{\rangle}

\begin{document}

\title[]{Asymptotics of spectral gaps of quasi-periodic Schr\"odinger operators}

\author{Martin Leguil}
\address{Department of Mathematics, University of Toronto, 40 St George
St. Toronto, ON M5S 2E4, Canada}
\email{martin.leguil@utoronto.ca}

\author{Jiangong You}
\address{
Chern Institute of Mathematics and LPMC, Nankai University, Tianjin 300071, China} \email{jyou@nankai.edu.cn}

\author{Zhiyan Zhao}
\address{Laboratoire J.A. Dieudonn\'{e}, Universit\'e C\^ote d'Azur, 06108 Cedex 02 Nice, France}
\email{zhiyan.zhao@unice.fr}

\author{Qi Zhou}
\address{Department of Mathematics, Nanjing University, Nanjing 210093, China}
\email{qizhou@nju.edu.cn}

\begin{abstract}
For non-critical almost Mathieu operators with Diophantine frequency, we establish   exponential asymptotics on the size of spectral gaps,  and show that  the spectrum is homogeneous. We also prove the homogeneity of the spectrum  for Sch\"odinger operators with (measure-theoretically) typical quasi-periodic  analytic potentials and fixed strong Diophantine frequency. As applications, we show the discrete version of Deift's conjecture \cite{Deift, Deift17} for subcritical analytic quasi-periodic initial data and solve a series of open problems of Damanik-Goldstein et al \cite{BDGL, DGL1, dgsv, Go} and Kotani \cite{Kot97}.
\end{abstract}

\maketitle



\section{Introduction and main results}
\noindent

We consider one-dimensional discrete Schr\"odinger operators on $\ell^2(\Z)$:\begin{equation}\label{schro}
(H_{V, \alpha, \theta} u)_n= u_{n+1}+u_{n-1} +  V(\theta + n\alpha) u_n,\quad \forall \  n\in\Z,
\end{equation}
where $\theta\in \T^d:=(\R/\Z)^d$ is the \textit{phase}, $V\colon \T^d \to \R$ is the \textit{potential}, and $\alpha\in\T^d$ is the \textit{frequency}. It is well known that the spectrum of $H_{V, \alpha, \theta}$,
 denoted by $\Sigma_{V, \alpha}$, is a compact subset of $\R$, independent of $\theta$ if $(1,\alpha)$ is rationally independent.
The {\it integrated density of states} (IDS) $N_{V,\alpha}\colon\R\to [0,1]$ of $H_{V,\alpha,\theta}$ is defined as
$$
N_{V, \alpha}(E):=\int_{\T} \mu_{V, \alpha,\theta}(-\infty,E] \, d\theta,
$$
where $\mu_{V,\alpha,\theta}$ is the spectral measure of $H_{V,\alpha,\theta}$.  Any bounded connected component of $\R \backslash\Sigma_{V, \alpha}$ is called a \textit{spectral gap}.
By the Gap-Labelling Theorem \cite{JM}, for any spectral gap $G$, there exists a unique $k\in \Z^d$ such that $N_{V,\alpha}|_G \equiv \langle k, \alpha\rangle {\rm\ mod} \ \Z $. Thus, the gaps in the spectrum of the operator $H_{V,\alpha,\theta}$ can be labelled by integer vectors: we denote by $G_k(V)=(E_k^-,E_k^{+})$ the gap with label $k\neq 0$. When $E^-_k=E^+_k$, we say the gap is {\it collapsed}.
We also set $\underline{E}:=\inf\Sigma_{V,\alpha}$, $\overline{E}:=\sup\Sigma_{V,\alpha}$,  and we let $G_0(V):=(-\infty, \underline{E}) \cup (\overline{E},\infty)$.

\subsection{Estimates on spectral gaps}
In this paper, we will focus on gap estimates for quasi-periodic operators as in $(\ref{schro})$. Before formulating our main results, let us first comment on the importance of gap estimates. From the  perspective of physics, $(\ref{schro})$ is a model for quantum Hall effect, and thus has attracted constant interest. In particular, after Von Klitzing's discovery of quantum Hall effect \cite{KDP},
Thouless and his coauthors \cite{TKNN}, assuming that all gaps are open for almost Mathieu operators, gave a theoretic explanation  of the quantization of the Hall conductance by Laughlin's argument, i.e.,  the Hall conductance  is quantized whenever the Fermi energy lies in an energy gap (Thouless was awarded the 2016 Nobel Prize partly due to this work).   From the mathematical point of view, gap estimates are a core problem  in the spectral theory of quasi-periodic Schr\"odinger operators. The question of lower bound estimates on spectral gaps is deeper than the well-known ``Dry Ten Martini Problem", while upper bound   estimates provide an efficient way for proving the  homogeneity of the spectrum, which is a key subject in the study of inverse spectral theory. As we will see, it is also related to Deift's conjecture \cite{Deift,Deift17} on the dynamics of  solutions to KdV equation with almost periodic initial data.

We start with the most important example of \eqref{schro}, namely  almost Mathieu operators (AMO),  which are defined as
$$
(H_{\lambda, \alpha, \theta} u)_n= u_{n+1}+u_{n-1} +  2\lambda \cos 2\pi (\theta + n\alpha) u_n,\quad \forall \  n\in\Z,
$$
with $\lambda \in \R$ and $\alpha \in \R \backslash \Q$. For simplicity, we denote by $\Sigma_{\lambda,\alpha}$ the spectrum of $H_{\lambda, \alpha, \theta}$ and
by $G_k(\lambda)=(E^-_k, E^+_k)$ the  gap with label $k$. Our first result is to establish exponential asymptotics for the spectral gaps of the AMO with Diophantine frequency.
Recall that $\alpha \in\R^d$ is {\it Diophantine} if there exist $\gamma>0$ and $\tau>d-1$ such that $\alpha \in {\rm DC}_d(\gamma,\tau)$, where 
\begin{equation}\label{dio}
{\rm DC}_d(\gamma,\tau):=\left\{x \in\R^d:  \inf_{j \in \Z}\left| \la n,x  \ra - j \right|
> \frac{\gamma}{|n|^{\tau}},\quad \forall \  n\in\Z^d\backslash\{0\} \right\}.
\end{equation}
Let ${\rm DC}_d:=\bigcup_{\gamma>0,
\, \tau>d-1} {\rm DC}_d(\gamma,\tau)$. In particular, when $d=1$, we simplify the above notations as ${\rm DC}(\gamma,\tau)$ and ${\rm DC}$.
 Our precise result is the following:
\begin{theorem}\label{thm_bounds_Mathieu}
For $\alpha\in{\rm DC}$, and for any $0<\xi<1$,
there exist $C=C(\lambda, \alpha,\xi)>0$, $\tilde{C}=\tilde{C}(\lambda,\alpha)$, and a numerical constant $\tilde{\xi}>1$, such that for all  $k \in \Z\backslash\{0\}$,
\begin{equation*}
\begin{array}{rccclll}
\tilde{C}\lambda^{\tilde{\xi}|k|} &\leq& |G_k(\lambda)| &\leq& C \lambda^{\xi |k|},&\quad&\text{if} \;\ 0<\lambda<1,\\
\tilde{C}\lambda^{-\tilde{\xi}|k|} &\leq& |G_k(\lambda)| &\leq& C \lambda^{-\xi |k|},&\quad&\text{if} \;\ 1<\lambda<\infty,
\end{array}
\end{equation*}
where $|G_k(\lambda)|$ denotes the length of $G_k(\lambda)$.
\end{theorem}

Let us review some recent works in connection with the question of gap estimates for quasi-periodic Schr\"{o}dinger operators.
The study of lower bounds dates back to a long-standing
conjecture, referred to in the literature as the ``Ten Martini Problem" \cite{Sim}, i.e.,  whether the spectrum of the almost Mathieu operator
$H_{\lambda, \alpha, \theta}$ is a Cantor set, in  the case  where $\lambda\neq0$ and $\alpha$ is irrational.
This problem was finally solved by Avila-Jitomirskaya \cite{AvilaJito1}: readers are invited to consult the history and references therein.
The so-called ``Dry Ten Martini Problem" is a further elaboration of the ``Ten Martini Problem" asking whether for any $\lambda\ne 0$ and  irrational $\alpha$,
all  possible spectral gaps of $H_{\lambda, \alpha, \theta}$  predicted by the Gap-Labelling theorem
 are non-collapsed. In \cite{AYZ}, Avila-You-Zhou solved this problem for any non-critical coupling constant $\lambda\neq 1$ (consult \cite{AYZ} for earlier advances on this problem).
Note that the ``Dry Ten Martini Problem" only concerns the openness of the spectral gaps, without asking any quantitative estimates on their size. After we claimed the result about  the Dry Ten Martini Problem in the conference ``Almost Periodic and Other Ergodic Problems" in 2015, Goldstein \cite{Go} asked us whether  any quantitative lower bound on the size of the gaps could be obtained.
Theorem \ref{thm_bounds_Mathieu} gives an answer to his question.

Let us move on to  upper bounds. The first result in this direction is due to Moser-P\"oschel. In  \cite{Moser-Poschel}, given an analytic potential $V\colon \T^d\to \R$, $d\geq 2$, and $\varpi \in {\rm DC}_d$, they consider the continuous quasi-periodic Schr\"{o}dinger operator on $L^2(\R)$:
$$
({\CL}_{V,\varpi} y)(t) = -y''(t)+V(\varpi t)y(t).
$$
Thanks to KAM techniques, Moser-P\"oschel  proved that if $V$ is small enough, then   $| G_k(V)|$ is exponentially small with respect to $|k|$ provided that $|k|$ is sufficiently large and  $\la k,\varpi\ra$ is not too close to the other $\la m,\varpi\ra$.\footnote{More precisely,  $\la k,\varpi\ra \in  \mathcal{R}(k)$, where  $$
\mathcal{R}(k):=\left\{ \la k,\varpi \ra\in\R : \inf_{j \in \Z}\left| \la m-k,\varpi  \ra - j \right|  \geq \frac{\gamma}{|m|^{\tau}},\quad \forall \ m \in \Z^d\setminus\{k\} \right\}.
$$}  Later, Amor \cite{HA} proved that in the same setting, the spectral gaps have sub-exponential decay for any $k\in \Z^d\backslash\{0\}$.
Although Amor \cite{HA} presented the result for discrete Schr\"{o}dinger operators, her method applies to the continuous case as well. Damanik-Goldstein \cite{DG} gave a stronger result:
$|G_k(V)| \leq \varepsilon e^{-\frac{r_0}{2}|k|}$ if $V\in C_{r_0}^{\omega}(\T^{d},\R)$ (i.e., the collection of bounded analytic functions on the strip $\{z\in\C:|\Im x|< r_0\}$) and $\varepsilon:=\sup_{|\Im x|< r_0}|V(x)|$ is sufficiently small.
We obtain the following upper bound:

\begin{theorem}\label{sharpdecay}
Let $\alpha  \in {\rm DC}_d$ and $V\in C_{r_0}^{\omega}(\T^{d},\R)$. For any $r\in(0,r_0)$, there exists $\varepsilon_0= \varepsilon_0(V, \alpha, r_0, r)>0$  such that if $\sup_{|\Im x|<r_0} |V(x)| < \varepsilon_0$, then for the discrete operator $H_{V,\alpha,\theta}$, we have
$$ |G_k(V)| \leq \varepsilon_0^{\frac23} e^{- r |k|}, \quad \forall \  k\in \Z^d\backslash\{0\}.$$
\end{theorem}

 We remark that the exponential decay rate of $|G_k(V)| $ can be arbitrarily close to the initial length $r_0$ of the strip.\footnote{After this work was completed, D. Damanik and M. Goldstein mentioned to us that in their setting it is also possible to get the sharp decay based on their method.}   We emphasize that our proof is  based on KAM, which works for both the discrete and the continuous case (cf. Theorem \ref{sharpcon} and  Corollary \ref{cor-local} (1)), while the proof in  \cite{DG}  is based on localization arguments, which cannot be directly employed in  the discrete case. Finally,  the above results are perturbative, in the sense that the smallness of $V$ depends on the Diophantine constants $\gamma$ and $\tau$.  This is optimal for the multi-frequency case in view of a counterexample due to Bourgain \cite{B1}.
However, our method does lead to  non-perturbative results (i.e., the smallness of $V$ does not depend on  $\gamma,\, \tau$) in the  one-frequency case. One may consult our Theorem \ref{thm_bounds_Mathieu}, Corollary \ref{cor-local}(2) for example.

%
%

\subsection{Homogeneous spectrum}
The exponential  decay of the spectral gaps can be used to prove the homogeneity of the spectrum. The  concept of homogeneous set was introduced by Carleson \cite{Car83}, and is defined as follows:

\begin{defi}\label{defi_homogeneous}
Given $\mu >0$, a closed set ${\CS}\subset\R$ is called $\mu-$homogeneous if for any $0<\epsilon\leq {\rm diam}{\CS}$ and any $E\in {\CS}$,
we have
$$
|\mathcal{S}\cap(E-\epsilon,E+\epsilon)| > \mu\epsilon.
$$
\end{defi}

Homogeneity of the spectrum plays an essential role in the inverse spectral theory of almost periodic potentials (as in the fundamental work of Sodin-Yuditskii \cite{SY95, SY97}). Assuming finite total gap length, homogeneity of the spectrum and the condition of being reflectionless (see Subsection \ref{subsec_Schrodinger} for the precise definition), Sodin-Yuditskii \cite{SY95} proved that the corresponding potential is almost periodic, and Gesztesy-Yuditskii \cite{GY} proved that the corresponding spectral measure is purely absolutely continuous.

Let us recall recent results on the homogeneity of the spectrum.
Building on the localization estimates developed in \cite{DG}, Damanik-Goldstein-Lukic \cite{DGL1} proved that  the  spectrum is homogeneous for continuous Schr\"{o}dinger operators $\mathcal{L}_{V,\varpi}$ with Diophantine  $\varpi$ and  sufficient small analytic $V$.
For the discrete operator $H_{V,\alpha}$ in the positive Lyapunov exponent regime, Damanik-Goldstein-Schlag-Voda \cite{dgsv} proved that the spectrum is homogeneous for any $\alpha\in {\rm SDC}$\footnote{
We say $\alpha\in\R$ is strong Diophantine if there exist $\gamma,\tau>0$ such that
\begin{equation}\label{strongdiophantinecondition}
\inf_{j\in\Z}|n\alpha- j| \geq \frac{\gamma}{|n| (\log |n|)^\tau}, \quad \forall \ n \in \Z\backslash\{0\}.
\end{equation}
For fixed $\gamma, \tau$, let ${\rm SDC}(\gamma,\tau)$ be the set of numbers satisfying \eqref{strongdiophantinecondition}, and let $\mathrm{SDC}:= \bigcup\limits_{\gamma,\tau>0}\ {\rm SDC}(\gamma,\tau)$.} and for some $\alpha\in {\rm DC}_d$ \cite{GSV}.
Inspired by the above results, it is natural to expect  a global description of the homogeneity of the spectrum for quasi-periodic Schr\"odinger operators (see Remark (2) after Theorem 1 in \cite{dgsv}). In this paper, we will prove the following:

\begin{theorem}\label{theorem_homo_spec}
Let $\alpha\in {\rm SDC}$. For a (measure-theoretically) typical analytic potential $V\in C^{\omega}( \T,\R)$, the spectrum $\Sigma_{V,\alpha}$ is $\mu-$homogeneous for some $\mu\in(0,1)$.
\end{theorem}


Let us first explain the  meaning of \textquotedblleft measure-theoretically typical\textquotedblright.  In infinite-dimensional settings, it is common to replace the notion of {\it almost every} by {\it prevalence}: we  fix some probability measure $\mu$ of compact support (describing a set of admissible perturbations $w$), and declare a property to be {\it measure-theoretically typical} if it is satisfied for almost every perturbation $v+w$ of every starting condition $w$. In finite-dimensional vector spaces, prevalence implies full Lebesgue measure.

For any $E \in \R$, we define a \textit{Schr\"odinger cocycle}  $(\alpha,S_E^V)$,  where
$S^{V}_{E}(\theta):=
\begin{pmatrix}
E-V(\theta) & -1\\
1 & 0
\end{pmatrix}$. The energy $E \in \Sigma_{V,\alpha}$ is called {\it supercritical} (resp.  {\it subcritical}), if the associate Lyapunov exponent satisfies $L(\alpha,S_E^V)>0$ (resp. $L(\alpha,S_{E}^V(\cdot+{\rm i}\epsilon))=0$ for any $|\epsilon|<\delta$, with $\delta>0$).
By Avila's global theory of one-frequency quasi-periodic Schr\"odinger operators \cite{Aglobal}, for  a (measure-theoretically) typical analytic potential $V\in C^{\omega}( \T,\R)$, any  $E \in \Sigma_{V,\alpha}$
is either subcritical or supercritical.
More precisely,  Avila \cite{Aglobal} proved that  for a (measure-theoretically) typical $V\in C^\omega(\T,\R)$,
there exist some integer $n\geq 1$ and a collection of points $a_1<b_1<\dots<a_n<b_n$ in the spectrum $\Sigma_{V,\alpha}$ such that $\Sigma_{V,\alpha}\subset \bigcup_{j=1}^n[a_j, b_j]$, where energies alternate between supercritical and subcritical along the sequence $(\Sigma_{V,\alpha}\cap [a_j, b_j])_j$. We denote by $I_i:=[a_j, b_j]$ the intervals such that the energies in $\Sigma_{V,\alpha}\cap [a_j, b_j]$ are subcritical, and let $\Sigma^{\mathrm{sub}}_{V,\alpha}:=\bigcup_{i}(\Sigma_{V,\alpha} \cap I_i)$ be the set of subcritical energies. Since Theorem \ref{theorem_homo_spec} in the supercritical regime has been proved in  \cite{dgsv}, we only need to prove the result for energies $E$ in the subcritical part of the spectrum.   Let $p_n/q_n$ be the  best approximants of $\alpha$ and
$\beta(\alpha):=\limsup\limits_{n\rightarrow\infty}\frac{\ln q_{n+1}}{q_n}$. Our precise result is the following:

\begin{theorem}\label{theo calibration gaps bands}
Let $\alpha\in\R\backslash \Q$ satisfy $\beta(\alpha)=0$.  For typical potentials $V\in C^{\omega}(\T,\R)$, the following assertions hold.
\begin{enumerate}
\item There exist constants $C, \vartheta>0$ depending on $V,\alpha$, such that $$|G_k(V)| \leq C e^{-\vartheta |k|}, \quad \forall \  k\in \Z\backslash \{0\} \ with  \ \overline{G_k(V)} \cap \Sigma_{V,\alpha}^{\rm sub}\neq \emptyset. $$
\item  For any $\tilde{\epsilon}>0$,  there exists $D=D(V, \alpha,\tilde{\epsilon})>0$ such that
 $${\rm dist}(G_{k}(V),G_{k'}(V) ) \geq D e^{-\tilde{\epsilon} |k'-k|}, $$
  if $k\neq k'\in \Z$ satisfy
$\overline{G_k(V)}\cap I_i \neq \emptyset$ and $\overline{G_{k'}(V)}\cap I_i \neq \emptyset$ for some $ i$.
\item There exists $\mu_0\in (0,1)$ such that
$$|\Sigma_{V,\alpha}\cap(E-\epsilon,E+\epsilon)| > \mu_0 \epsilon,\quad \forall \  E\in \Sigma_{V,\alpha}^{\mathrm{sub}}, \quad \forall  \ 0<\epsilon\leq {\rm diam}\Sigma_{V, \alpha}.$$
 \end{enumerate}
\end{theorem}

 Theorem \ref{theo calibration gaps bands} answers an open question raised by Damanik-Goldstein-Schlag-Voda \cite{dgsv}  (Problem 1 of \cite{dgsv}, see also Question 3.1 of \cite{DGL1}), i.e., whether the spectrum $\Sigma_{\lambda V,\alpha}$ is homogeneous, assuming that $L(\alpha,S_E^{\lambda V})$ vanishes identically on $\Sigma_{\lambda V,\alpha}$ for $0<|\lambda|<\lambda_0$. Actually, Theorem \ref{theo calibration gaps bands} gives an even more precise description of the structure of the spectrum. 

Let us  make a short comment on Damanik-Goldstein-Schlag-Voda's open problem (Problem 1 of \cite{dgsv}). Under the assumption that $L(\alpha,S_E^{\lambda V})$ vanishes on the spectrum for $0<|\lambda |<\lambda_0$, they initially asked whether one could find a complete set of Bloch-Floquet eigenfunctions for $0<|\lambda|<\lambda_0$. In fact, this point is an easy consequence of Avila's Almost Reducibility Conjecture (subcriticality implies almost reducibility)\footnote{Consult Section \ref{ARCGlobal} for details.} \cite{A2, Aglobal}: one could follow for instance Theorem 4.2 of \cite{AYZ1} to give a proof. What they really asked was whether the spectrum is homogenous; we refer the reader to  Problem 1 of \cite{dgsv} for more explanations (one may also consult  Question 3.1 of \cite{DGL1}).

For the most important example of AMO, Damanik-Goldstein-Lukic (Question 3.2 of \cite{DGL1}) asked
for which values of $\lambda$ the spectrum of AMO is homogeneous.
Back in 1997, Kotani \cite{Kot97} already asked a similar question: whether the spectrum is homogeneous under the conditions $\lim\limits_{n\rightarrow \infty}q_n^2/q_{n+1}=0$ and $0<\lambda<1$. In this paper we answer their questions as follows:

\begin{theorem}\label{homo-amo}
Assume that $\beta(\alpha)=0$ and $\lambda\neq  1$. The following assertions hold:
\begin{enumerate}
\item  For any given $r\in (0, \frac{1}{12}|\ln\lambda|)$, there exists  $C=C(\lambda,\alpha,r)>0$ such that  $$|G_k(\lambda)| \leq C e^{-r |k|}, \quad \forall \  k\in \Z\backslash \{0\}.$$
\item  For any $\tilde{\epsilon}>0$, there exists a constant $D>0$ depending on $\tilde{\epsilon},\lambda,\alpha$, such that
\begin{equation*}
\begin{array}{rclll}
\mathrm{dist}(G_k(\lambda),G_{k'}(\lambda)) &\geq& D e^{-\tilde{\epsilon} |k'-k|}, &\ & \forall \ k\neq k'\in \Z\backslash\{0\},  \\
|E_k^-  - \underline{E}|,\;  |E_k^ + - \overline{E}| &\geq& D e^{-\tilde{\epsilon} |k|},&\ & \forall \ k\in  \Z\backslash\{0\}.
\end{array}
\end{equation*}
\item $\Sigma_{\lambda,\alpha}$  is $\mu-$homogeneous for some $\mu\in (0,1)$.
\end{enumerate}
\end{theorem}

If $\lambda=1$,  one knows that $|\Sigma_{\lambda,\alpha}|=0$ for every $\alpha\in \R \backslash \Q$ \cite{AK06, L}, hence the spectrum is not homogeneous.  Prior to us,  Damanik-Goldstein-Schlag-Voda \cite{dgsv} proved that if $\alpha \in \mathrm{SDC}$ and $\lambda\neq 1$, then $\Sigma_{\lambda,\alpha}$ is homogeneous. Compared to their result, not only we  weaken the condition $\alpha \in \mathrm{SDC}$ to $\beta(\alpha)=0$, but more importantly, we establish a  calibration between the gaps and the bands of the operator (see (1) and (2) in the above statements). Indeed, as pointed out by Damanik-Goldstein-Schlag-Voda \cite{dgsv}: ``This feature was not known for the almost Mathieu operator even in the regime of small coupling". In their work, they established the following weaker estimate: there exists $N_0(\alpha,\lambda)\geq 0$ such that if $N\geq N_0$ and if $G_{k}(\lambda),G_{k'}(\lambda)$ are two gaps with $|G_{k}(\lambda)|,|G_{k'}(\lambda)| > e^{-N^{1-}}$, then $\mathrm{dist}(G_{k}(\lambda),G_{k'}(\lambda))>e^{-(\log N)^{C_0}}$ for some constant $C_0=C_0(\lambda,\alpha)>0$.

The arithmetic property $\beta(\alpha)=0$ is essential to the homogeneity of the spectrum. After this work, Avila-Last-Shamis-Zhou \cite{ALSZ} proved that if $\beta>0$ and $e^{-\beta/2}<\lambda<e^{\beta/2}$, then $\Sigma_{\lambda,\alpha}$  is not homogeneous.

\subsection{Deift's conjecture}
As an application of  homogeneity of the spectrum, we can prove the  discrete version of Deift's conjecture for some almost periodic initial datum (not neccessarily small).
 Recall that Deift's  conjecture (Problem 1 of \cite{Deift,Deift17}) asks whether  for almost periodic initial datum, the solutions to the KdV equation are almost periodic in the time variable.

The  calibration estimates between the gaps and the bands of  Schr\"odinger operators (similar to items $(1)-(3)$ in Theorem \ref{theo calibration gaps bands})  played an important role in the proof of Deift's conjecture  for small analytic quasi-periodic data \cite{BDGL,DG2016,DGL2}.
Let us make a short review of the recent developments on this important conjecture. Tsugawa \cite{Ts} proved local existence and uniqueness of solutions to the KdV equation when the frequency is Diophantine and the Fourier coefficients of the potential decay at a sufficiently fast polynomial rate. Damanik-Goldstein  \cite{DG2016} then proved global existence and uniqueness for a Diophantine frequency and small quasi-periodic analytic initial datum. Recently, Binder-Damanik-Goldstein-Lukic \cite{BDGL} showed that in the same setting, the solution is in fact almost periodic in time, thus proving Deift's conjecture in this special case. In this paper, we consider the discrete version of Deift's conjecture, namely that for almost periodic initial data,  the Toda flow is almost periodic in the time variable.

The Toda flow is defined to be any solution of the Toda lattice equation
\begin{equation}\label{Toda}
 \left\{ \begin{array}{rcl}
 \displaystyle a'_n(t)&=&a_n(t)\left(b_{n+1}(t)-b_n(t)\right),\\[2mm]
  \displaystyle b'_n(t)&=&2(a^2_{n}(t)-a^2_{n-1}(t)),
 \end{array} \right. \quad n\in\Z.
 \end{equation}
In view of Theorem 12.6 in \cite{Teschl}, with initial condition $(a(0),b(0))\in \ell^{\infty}(\Z)\times\ell^{\infty}(\Z)$,
 there is a unique solution $(a,b)\in{\CC}^{\infty}(\R, \ell^{\infty}(\Z)\times\ell^{\infty}(\Z))$ to \eqref{Toda}. If we identify $(a(t), b(t))$ with a doubly infinite Jacobi matrix $J(t)$, i.e.,
 \begin{equation}\label{Jacobi}
 (J(t) u)_n=a_{n-1}(t)\, u_{n-1}+b_n(t) \,u_n+ a_{n}(t) \, u_{n+1},
 \end{equation}
then (\ref{Toda}) can be expressed equivalently as a Lax pair:
 \begin{equation}\label{Toda_Lax}
\frac{d}{dt} J(t)=P(t)J(t)-J(t)P(t)
\end{equation}
where $P(t)$ is an operator defined as
\[ ( P(t) u)_n:=-a_{n-1}(t)\, u_{n-1}+a_n(t)\, u_{n+1}.\]

Now, we take the almost periodic initial condition $(a_n(0), b_n(0))=(1,V(\theta+n\alpha))$, $n\in\Z$, with $V\in C^{\omega}(\T,\R)$ and $\alpha\in \R\backslash \Q$, i.e.,
$J(0)=H_{V,\alpha,\theta}$, and  consider the almost periodicity of the solution $(a(t), b(t))$.
In fact, Binder-Damanik-Goldstein-Lukic \cite{BDGL} asked whether one could generalize their result to Avila's subcritical regime: in particular, for the most important example of almost Mathieu operators, whether the result holds for $0<\lambda<1$. For partial advance on this problem, one can consult \cite{BDLV}.
 In this paper, we give an affirmative answer to their question as follows:

\begin{theorem}\label{thm_deift_conjecture}
Let $\alpha \in \R\backslash\Q$ with $\beta(\alpha)=0$. Given a potential $V\in C^{\omega}(\T, \R)$ which is subcritical\footnote{i.e., for any $E \in \Sigma_{V,\alpha}$, the cocycle $(\alpha,S_E^V)$ is subcritical.}, we consider the Toda flow (\ref{Toda}) with initial condition $(a_n,b_n)(0)=(1,V(\theta+n\alpha))$. We then have:
\begin{enumerate}
  \item For any $\theta\in\T$, (\ref{Toda}) admits a unique solution $(a(t), b(t))$ defined for all $t \in \R$.
  \item For every $t$, the Jacobi matrix $J(t)$ given by (\ref{Jacobi})
  is almost periodic with constant spectrum $\Sigma_{V,\alpha}$.
  \item The solution $(a(t),b(t))$ is almost periodic in $t$ in the following sense: there exists a continuous map ${\CM}\colon\T^{\Z} \to \ell^{\infty}(\Z)\times\ell^{\infty}(\Z)$, a point $\varphi\in \T^{\Z}$ and a direction $\varpi\in\R^{\Z}$, such that $(a(t),b(t))={\CM}(\varphi+\varpi t)$.
\end{enumerate}
In particular, the above conclusion holds for $V=2\lambda\cos2\pi(\cdot)$ with $0<\lambda < 1$.
\end{theorem}

\subsection{Ideas of the proofs}

While we answer a series of open problems  posed in \cite{BDGL, DGL1, dgsv, Go}, we used a  totally different approach compared to these papers. Our approach is from the perspective of dynamical systems, and  is based on quantitative almost reducibility. The philosophy is that nice quantitative almost reducibility should induce nice spectral applications.
This approach has been proved to be very fruitful \cite{A2,AvilaJito, AYZ1, AYZ}.

As for the upper and lower bounds on the size of  spectral gaps, we need to analyze the behavior of Schr\"odinger cocycles at the edge points of the spectral gaps. At the edge points, the cocycles are reducible to constant parabolic cocycles. The crucial points for us are  the exponential decay of  the off-diagonal element of the parabolic matrix and the control of the  growth of the conjugacy with respect to the  label $k$. Furthermore, in order to obtain uniformity of the decay rate  with respect to the label $k$, we need some strong almost reducibility result, namely that the cocycle is almost reducible in a fixed band.

Now, we distinguish between two cases in the proof. If the frequency is Diophantine,   we will develop a new KAM scheme to prove the almost reducibility with nice estimates (which works for multifrequencies, and for both continuous and discrete cocycles). Moreover, in order to get a sharp decay of the spectral gaps (Theorem \ref{sharpdecay}), we prove almost reducibility of the cocycle  in a fixed band, arbitrarily close to the initial band. We remark that although Chavaudret \cite{Chavaudret}  developed some kind of strong almost reducibility result, the estimates there are not sufficient to yield good spectral applications.
On the other hand, if the frequency $\alpha$ satisfies $\beta(\alpha)=0$, we need the almost localization argument (via Aubry duality) given by Avila \cite{A1} (initially developed by Avila-Jitomirskaya \cite{AvilaJito}): as we can see from the proof, Corona estimates are the key ingredient for these nice almost reducibility estimates.

However, each of these two approaches leads to local results. In order to deal with the global regime, we need Avila's global theory of analytic ${\rm SL}(2,\R)-$cocycles \cite{Aglobal}, especially his
 proof of the Almost Reducibility Conjecture \cite{A2, Aglobal}. Moreover, in order to have a uniform control on the conjugacies with respect to $E\in \Sigma_{V,\alpha}$ (which ultimately yields uniform decay rate with respect to the label $k$), we shall perform some compactness argument. Here, the key point still follows from Avila's global theory, namely openness of the almost reducibility property, and  compactness of the  subcritical spectrum. We should also point out that for AMO, the strategy is  different, since the proof is based on Avila-Jitomirskaya's almost localization technique \cite{AvilaJito} instead of Avila's Almost Reducibility Conjecture. In particular, it is the main reason why we can get sharp decay of the spectral gaps for noncritical AMO (Theorem \ref{thm_bounds_Mathieu}).   Avila's  Almost Reducibility Conjecture does not allow us to get this sharp result since the analytic strip has to be   shrinked greatly in his proof.

In Section \ref{Sec_bounds}, we will prove a criterion  (Theorem \ref{thm_upperbound}) to get quantitative upper and lower  bounds on the size of the gaps, building on quantitative reducibility results which even work for Liouvillean frequencies. Although the method developed by  Moser-P\"oschel \cite{Moser-Poschel} and  Amor \cite{HA} can be used to obtain some decay of the upper bounds for small potentials, yet, when dealing with large potentials,  their approach does not work since their estimates need explicit dependence on the parameters. In fact, when reducing  the global potential  to local regimes by Avila's global theory, the explicit dependence of the parameters is lost. However, our method is purely dynamical,  which means that we only need the information for the fixed cocycle, and it is the main reason why we can deal with all subcritical regimes.
We also emphasize that our estimates on the lower bounds of the gaps of AMO  crucially depend  on  a key proposition of \cite{AYZ} which was initially used by the authors to prove the  non-critical   ``Dry Ten Martini Problem" for Liouvillean frequencies.

Homogeneity of the spectrum in the subcritical regime is derived from the upper bounds on the size of spectral gaps, together with H\"older continuity of the IDS.
Theorem \ref{theorem_homo_spec} is proved by combining this with previous work of Damanik-Goldstein-Schlag-Voda \cite{dgsv} and Avila's global theory of one-frequency Schr\"odinger operators \cite{Aglobal}. As a consequence of homogeneity (Theorem \ref{theo calibration gaps bands}) and purely absolutely continuous spectrum of subcritical Schr\"odinger operators \cite{A2}, we are then able to prove the discrete version of Deift's conjecture for such initial data, building on an previous result of   Vinnikov-Yuditskii \cite{VY}.


\section{Preliminaries} \label{preliminaries}

For a function $f$ defined on a strip $\{|\Im z|<h\}$, we define $|f|_{h}:=\sup_{|\Im z|<h} |f(z)|$. Analogously, for $f$ defined on $\T$, we set $|f|_{\T}:=\sup_{x \in \T} |f(x)|$.   For any $f\colon\T^d\to \C$, we let $[f]:=\int_{\T^d} f(\theta) d\theta$.
  When $\theta \in \R$, we also set $\|\theta\|_{\T}:=\inf_{j \in \Z} |\theta-j|$.


\subsection{Continued Fraction Expansion}\label{sec:2.1}
Let $\alpha \in (0,1)\backslash \Q$, $ a_0:=0$ and
$\alpha_{0}:=\alpha.$  Inductively, for $k\geq 1$, we define
$$a_k:=\left\lfloor\alpha_{k-1}^{-1}\right\rfloor,\qquad \alpha_k:=\alpha_{k-1}^{-1}-a_k.$$
Let $p_0:=0$, $p_1:=1$, $q_0:=1$, $q_1:=a_1$. We define inductively
$p_k:=a_k p_{k-1}+p_{k-2}$, $q_k:=a_kq_{k-1}+q_{k-2}$.
Then $(q_n)_n$ is the sequence of denominators of the best rational
approximations of $\alpha$, since we have
$\|k\alpha\|_{\T} \geq \|q_{n-1}\alpha\|_{\T}$, $\forall \   1
\leq k < q_n$,
and
$$
{1 \over 2q_{n+1}}\leq \|q_n \alpha \|_{\T} \leq {1 \over q_{n+1}}.
$$
Let $\displaystyle \beta(\alpha):=\limsup_{n\rightarrow\infty}\frac{\ln q_{n+1}}{q_n}$. Equivalently, we have
\begin{equation}\label{equibeta}
\beta(\alpha)=\limsup_{k\rightarrow \infty} \frac{1}{|k|} \ln \frac{1}{ \|k\alpha\|_{\T}} .
\end{equation}

\subsection{Schr\"{o}dinger operators}\label{subsec_Schrodinger}

Given $V\in C^\omega(\T^d,\R)$ and $\alpha\in\R^d$, we define the Schr\"{o}dinger operator $H_{V,\alpha,\theta}$  as in (\ref{schro}).
For any $\psi\in\ell^2(\Z)$, we let $\mu_{V,\alpha,\theta}^\psi$ be
the spectral measure of $H_{V,\alpha,\theta}$ corresponding to $\psi$:
$$\la(H_{V,\alpha,\theta}-E)^{-1}\psi, \psi \ra = \int_\R\frac{1}{E-E'}d\mu_{V,\alpha,\theta}^\psi(E'), \quad \forall \  E\in \C\backslash\Sigma_{V,\alpha}.$$
We denote $\mu_{V,\alpha,\theta}:=\mu_{V,\alpha,\theta}^{e_{-1}}+\mu_{V,\alpha,\theta}^{e_{0}}$, where
$\{e_n\}_{n\in\Z}$ is the canonical basis of $\ell^2(\Z)$.

More generally,
we consider the self-adjoint  Jacobi matrices $J$:
$$(Ju)_n=a_{n-1}u_{n-1}+b_n u_n+a_n u_{n+1},\quad n\in\Z.$$
Let $\Sigma\subset\R$ be the spectrum of $J$. Given any $z\not\in\Sigma$,
the Green's function of $J$ is the integral kernel of $(J-z)^{-1}$:
 $$G_J(m, n; z) := \la e_n, (J -z)^{-1} e_m \ra.$$
 \begin{defi}
A Jacobi operator $J$ is said to be {\it reflectionless} on $\Sigma$ if
$\Re(G_J(0,0;  E + i0)) = 0$ for Lebesgue$-$a.e. $E \in \Sigma$.
 \end{defi}

Given any $z\in\HH:=\{z\in\C : \Im z>0\}$, the difference equation $Ju=zu$ has two solutions $u^{\pm}$ (defined up to normalization) with $u_0^{\pm}\neq 0$, which are in $\ell^2(\Z_{\pm})$ respectively.
Let $m_J^{\pm}:=\mp\frac{u^{\pm}_{\pm1}}{a_0 u^{\pm}_{0}}$.
Then $m_J^+$ and $m_J^-$ are Herglotz functions, i.e., they map $\HH$ holomorphically into itself. For almost every $E\in\R$, the non-tangential limits $\lim\limits_{\epsilon\rightarrow0^+}m_J^{\pm}(E+{\rm i}\epsilon)$ exist.
Note that we have
\begin{equation}\label{green_m+-}
G_J(0,0;  z)=\frac{1}{a_0^2(m_J^+(z)+m_J^-(z))},\quad z\in\HH.
\end{equation}

\subsection{Quasiperiodic cocycles}\label{subslyap}

%
%
%

Given $A \in C^\omega(\T^d,{\rm SL}(2,\C))$ and $\alpha\in\R^d$ rationally independent, we define the quasi-periodic \textit{cocycle} $(\alpha,A)$:
$$
(\alpha,A)\colon \left\{
\begin{array}{rcl}
\T^d \times \C^2 &\to& \T^d \times \C^2\\[1mm]
(x,v) &\mapsto& (x+\alpha,A(x)\cdot v)
\end{array}
\right.  .
$$
The iterates of $(\alpha,A)$ are of the form $(\alpha,A)^n=(n\alpha,  \mathcal{A}_n)$, where
$$
\mathcal{A}_n(x):=
\left\{\begin{array}{l l}
A(x+(n-1)\alpha) \cdots A(x+\alpha) A(x),  & n\geq 0\\[1mm]
A^{-1}(x+n\alpha) A^{-1}(x+(n+1)\alpha) \cdots A^{-1}(x-\alpha), & n <0
\end{array}\right.    .
$$
The {\it Lyapunov exponent} is defined by
$\displaystyle
L(\alpha,A):=\lim\limits_{n\to \infty} \frac{1}{n} \int_{\T^d} \ln |\mathcal{A}_n(x)| dx
$.

The cocycle $(\alpha,A)$ is {\it uniformly hyperbolic} if, for every $x \in \T^d$, there exists a continuous splitting $\C^2=E^s(x)\oplus E^u(x)$ such that for every $n \geq 0$,
$$
\begin{array}{rl}
|A_n(x) \, v| \leq C e^{-cn}|v|, &  v \in E^s(x),\\[1mm]
|A_n(x)^{-1}   v| \leq C e^{-cn}|v|, &  v \in E^u(x+n\alpha),
\end{array}
$$
for some constants $C,c>0$.
This splitting is invariant by the dynamics, i.e.,
$$A(x) E^{*}(x)=E^{*}(x+\alpha), \quad *=``s"\;\ {\rm or} \;\ ``u", \quad \forall \  x \in \T^d.$$

Assume that $A \in C (\T^d, {\rm SL}(2, \R))$ is homotopic to the identity. It induces the projective skew-product $F_A\colon \T^d \times \mathbb{S}^1 \to \T^d \times \mathbb{S}^1$ with
$$
F_A(x,w):=\left(x+\a,\, \frac{A(x) \cdot w}{|A(x) \cdot w|}\right),
$$
which is also homotopic to the identity.
Thus we can lift $F_A$ to a map $\tF_A\colon \T^d \times \R \to \T^d \times \R$ of the form $\tF_A(x,y)=(x+\alpha,y+\psi_x(y))$, where for every $x \in \T^d$, $\psi_x$ is $\Z$-periodic.
The map $\psi\colon\T^d \times \T  \to \R$ is called a {\it lift} of $A$. Let $\mu$ be any probability measure on $\T^d \times \R$ which is invariant by $\widetilde{F}_A$, and whose projection on the first coordinate is given by Lebesgue measure.
The number
$$
\rho(\alpha,A):=\int_{\T^d \times \R} \psi_x(y)\ d\mu(x,y) \ {\rm mod} \ \Z
$$
 depends  neither on the lift $\psi$ nor on the measure $\mu$, and is called the \textit{fibered rotation number} of $(\alpha,A)$ (see \cite{H,JM} for more details).

Given $\theta\in\T^d$, let $
R_\theta:=
\begin{pmatrix}
\cos2 \pi\theta & -\sin2\pi\theta\\
\sin2\pi\theta & \cos2\pi\theta
\end{pmatrix}$.
If $A\colon \T^d\to{\rm PSL}(2,\R)$ is homotopic to $\theta \mapsto R_{\frac{\la n, \theta\ra}{2}}$ for some $n\in\Z^d$,
then we call $n$ the {\it degree} of $A$ and denote it by $\deg A$.
The fibered rotation number is invariant under real conjugacies which are homotopic to the identity. More generally, if $(\alpha,A_1)$ is conjugated to $(\alpha, A_2)$, i.e., $B(\cdot+\alpha)^{-1}A_1(\cdot)B(\cdot)=A_2(\cdot)$, for some $B \colon \T^d\to{\rm PSL}(2,\R)$ with ${\rm deg} B=n$, then
\begin{equation}\label{rot-conj}
\rho(\alpha, A_1)= \rho(\alpha, A_2)+ \frac{\la n,\alpha \ra}2.
\end{equation}
Moreover, it follows immediately from the definition of rotation number that
\begin{lemma}\label{esti_rot_num}
If $A \colon \T^d\to{\rm SL}(2,\R)$ is homotopic to the identity, then
$$|\rho(\alpha, A)-\theta|<  |A-R_{\theta}|_{\T^d}.$$
\end{lemma}

%

A typical  example is given by the so-called \textit{Schr\"{o}dinger cocycles} $(\alpha,S_E^V)$, with
$$
S^V_{E}(\cdot):=
\begin{pmatrix}
E-V(\cdot) & -1\\
1 & 0
\end{pmatrix},   \quad E\in\R.
$$
Those cocycles were introduced because in connection with the eigenvalue equation $H_{V, \alpha, \theta}u=E u$: indeed, any formal solution $u=(u_n)_{n \in \Z}$ of $H_{V, \alpha, \theta}u=E u$ satisfies
$$
\begin{pmatrix}
u_{n+1}\\
u_n
\end{pmatrix}
= S_E^V(\theta+n\alpha) \begin{pmatrix}
u_{n}\\
u_{n-1}
\end{pmatrix},\quad \forall \  n \in \Z.
$$
The spectral properties of $H_{V,\alpha,\theta}$ and the dynamics of $(\alpha,S_E^V)$ are closely related by the well-known fact:
 $E\in \Sigma_{V,\alpha}$ if and only if $(\alpha,S^V_{E})$ is \textit{not} uniformly hyperbolic.

For any fixed $E \in \R$, the map $x \mapsto S_E^V(x)$ is homotopic to the identity, hence the rotation number $\rho(\alpha,S_E^V)$ is well defined. Moreover, $\rho(\alpha,S_E^V)\in [0,\frac12]$ relates to the  integrated density of states $N=N_{V,\alpha}$ as follows:
$$
N_{V,\alpha}(E)=1-2 \rho(\alpha,S_E^V).
$$
By \textit{Thouless formula}, we also have the following relation between the integrated density of states $N$ and the Lyapunov exponent $L$:
$$
L(\a, S^V_{E}) = \int \ln |E-E'| \, dN_{V,\alpha}(E').
$$

\subsection{Aubry duality and almost localization}\label{subs_duality}

Let 
$\theta\in \T$, $V\in C^\omega(\T,\R)$, and denote by $(\widehat v_l)_{l \in \Z}$ the Fourier coefficients of $V$.  The \textit{dual Schr\"odinger operator} $\widehat H_{V,\alpha,\theta}$ 
is defined on $\ell^2(\Z)$ by:
$$
\left(\widehat H_{V,\alpha,\theta} u\right)_j:=\sum\limits_{l\in \Z} \widehat v_l  u_{j-l} + 2 \cos 2\pi (\theta + j \alpha) u_j, \quad \forall \  j \in \Z.
$$

\textit{Aubry duality} involves an algebraic relation between the families of operators $\{H_{V,\alpha,\theta}\}_{\theta \in \T}$ and $\{\widehat H_{V,\alpha,\theta}\}_{\theta \in \T}$: given an eigenvector of $\widehat H_{V,\alpha,\theta}$ whose coefficients decay exponentially, one can construct an analytic \textit{Bloch wave} for the dual operator $H_{V,\alpha,\theta}$. However, if one wants to obtain information for all energies $E$, one cannot expect that all the eigenfunctions decay exponentially. The weaker notion of almost localization proved to be very useful in this context. 


\begin{defi}[Resonances]
Fix $\epsilon_0>0$ and $\theta\in \T$. An integer $k\in \Z$ is called an \textit{$\epsilon_0-$resonance} of $\theta$ if $\|2 \theta - k\alpha\|_{\T}\leq e^{-\epsilon_0 |k|}$ and $\|2 \theta - k\alpha\|_{\T}=\min_{|l|\leq |k|} \|2 \theta - l\alpha\|_{\T}$. We denote by $\{n_l\}_l$  the set  of $\epsilon_0-$resonances of $\theta$, ordered in such a way that $|n_1|\leq |n_2|\leq \dots$.
We say that $\theta$ is  $\epsilon_0-$\textit{resonant} if the set $\{n_l\}_l$ is infinite.
\end{defi}

\begin{defi}[Almost localization]
The family $\{\widehat H_{V,\alpha,\theta}\}_{\theta\in\T}$ is said to be \textit{almost localized} if there exist constants $C_0, C_1,\epsilon_0, \epsilon_1>0$ such that for all $\theta\in\T$, any generalized solution $u=(u_k)_{k \in \Z}$ to the eigenvalue problem $\widehat H_{V,\alpha,\theta} u = E u$ with $u_0=1$ and $|u_k| \leq 1+|k|$ satisfies
\begin{equation}\label{almost_localization}
|u_k| \leq C_1 e^{-\epsilon_1 |k|},\quad \forall \ C_0 |n_j| \leq |k| \leq C_0^{-1} |n_{j+1}|,
\end{equation}
where $\{n_l\}_l$  is the set  of  $\epsilon_0-$resonances of $\theta$.
\end{defi}

The basic fact for us is the following result of Avila and Jitomirskaya \cite{AvilaJito}:


\begin{theorem}\cite{AvilaJito}\label{almostredth}
Let $\alpha\in \R \backslash \Q $ satisfy $\beta(\alpha)=0$. There exists an absolute constant $c_0>0$ such that for any given $0<r_0<1$, $C_0>1$, there exist $\epsilon_0=\epsilon_0(r_0)>0$, $\epsilon_1=\epsilon_1(r_0,C_0)\in(0,r_0)$ and $C_1=C_1(\alpha, r_0,C_0)>0$ such that the following is true: given any $V \in C^\omega(\T,\R)$ satisfying $|V|_{r_0}\leq c_0 r_0^3$, the family $\{\widehat H_{V,\alpha,\theta}\}_{\theta\in \T}$ is almost localized with parameters $C_0,\, C_1,\, \epsilon_0,\, \epsilon_1$ as in (\ref{almost_localization}).
\end{theorem}

If we restrict ourselves to almost Mathieu operators, then we expect the decay rate of the eigenfunction to be $\ln\lambda$, which is the content of the following result.

\begin{theorem}\label{thm_almost_almost-1}
Let $\alpha \in \R\backslash \Q$ satisfy $\beta(\alpha)=0$. If $\lambda>1$, then 
 $\{H_{\lambda,\alpha,\theta}\}_{\theta}$ is almost localized. Moreover, for any $\delta\in (0,\ln \lambda)$, any $C_0>1$, there exists $\epsilon_0=\epsilon_0(\lambda,C_0,\delta)>0$ such that the following holds. Let  $H_{\lambda,\alpha,\theta} u =E u$ for some $E \in \Sigma_{\lambda,\alpha}$, with $|u_j|\leq 1$ for all $j \in \Z$.
 \begin{enumerate}
\item  If $\theta$ is not $\epsilon_0-$resonant, then
$|u_{j}|\leq e^{-(\ln \lambda -\delta)|j|}$ for $|j|$ large enough.
\item Else, let $\{n_l\}_l$ be the set of $\epsilon_0-$resonances of $\theta$. Given any $\eta>0$,
\begin{equation*}
|u_{j}|\leq e^{-(\ln \lambda -\delta)|j|},\quad \forall \ 2C_0 |n_l|+\eta|n_{l+1}| < |j| < (2C_0)^{-1} |n_{l+1}|,
\end{equation*}
provided that $|j|$ is large enough.
\end{enumerate}
\end{theorem}

If $\alpha \in \mathrm{DC}$, then the above result is shown in \cite{Jito1999} for a full measure set of $\theta\in\T$.  As for the case $\beta(\alpha)=0$, we could not find a reference in the literature. For completeness, we give a proof  in Appendix \ref{Appendix_A}.\\

As a direct corollary of Theorem \ref{almostredth}, one can see that the dual cocycle has subexponential growth on the strip $\{|\Im z|<\frac{\epsilon_1}{2\pi}\}$, which was first realized by Avila (one may consult footnote 5 of \cite{A1}). We sketch the proof here for completeness.

\begin{corollary}\label{subexp groowth}
Let $\alpha\in \R $ satisfy $\beta(\alpha)=0$ and $V \in C^\omega(\T,\R)$ satisfy $|V|_{r_0}\leq c_0 r_0^3$ for  the absolute constant $c_0$ in Theorem \ref{almostredth}.
For any $\delta>0$, there exists a constant $C_2= C_2(\alpha,\epsilon_1, \delta)>0$ such that for any $E \in \Sigma_{V,\alpha}$, the iterates of the cocycle $(\alpha,S_{E}^V)$ satisfy
\begin{equation}\label{subpolgrowth}
\sup_{ |\Im z|<\frac{\epsilon_1}{2\pi}}|\mathcal{A}_m (z)| \leq  C_2 e^{\delta |m|},\quad \forall \  m \in \mathbb{N}.
\end{equation}
\end{corollary}

\begin{proof} We first show that for any $E \in \Sigma_{V,\alpha}$,
\begin{equation}\label{eq vanish lyap exp}
L(\alpha, S_{E}^V(\cdot+{\rm i} y))=0,   \quad  \forall \  |y| <\frac{\epsilon_1}{2\pi}.
\end{equation}

As above, given any $E \in \Sigma_{V,\alpha}$,  there exist $\theta=\theta(E)\in \R$ and $\widehat u$ such that $\widehat H_{V,\alpha,\theta} \widehat u= E \widehat u$, with $\widehat u_0=1$ and $|\widehat u_j|\leq 1$ for all $j \in \Z$.
We claim that it is enough to show \eqref{eq vanish lyap exp} at energies $E \in \Sigma_{V,\alpha}$ such that $\theta=\theta(E)$ is not $\epsilon_0-$resonant. Indeed, by Theorem 4.2 in \cite{AvilaJito}, there exists $c>0$ depending on $V, \alpha$ such that if $\theta$ is $\epsilon_0-$resonant, then $\rho(E)$ is $c-$resonant.
Then, it follows from a standard Borel-Cantelli argument that the set of $E \in \Sigma_{V,\alpha}$ such that $\theta$ is $\epsilon_0-$resonant has zero Lebesgue measure (indeed, it has zero Hausdorff dimension, as shown by Avila \cite{A1}). On the other hand, by \cite{BJ,JKS}, we know that given any $|y|< \frac{\epsilon_1}{2\pi} $, the map $E \mapsto L(\alpha,S_{E}^V(\cdot+{\rm i} y))$ is continuous, and thus, the claim is proved.

Now, fix $E \in \Sigma_{V,\alpha}$ such that $\theta$ is not $\epsilon_0-$resonant and let $\widehat u$ be as above.
By \eqref{almost_localization}, $|\widehat u_j| \leq C_1 e^{-\epsilon_1 |j|}$ for all sufficiently large $|j|$.
Therefore, the function $u\colon z \mapsto \sum_{j \in \Z} \widehat u_j e^{2 \pi {\rm i} j z}$ is well-defined on $\{|\Im z|<\frac{\epsilon_1}{2\pi}\}$, and  ${\CU}\colon z \mapsto
\begin{pmatrix}
e^{2 \pi {\rm i}\theta} u(z)\\
u(z-\a)
\end{pmatrix}$ satisfies
$S_{E}^V(z) \, {\CU}(z) = e^{2\pi {\rm i}\theta} {\CU}(z+\alpha)$.
Set $Z:=(\mathcal{U},\, \frac{1}{\|\CU\|^2}R_{\frac14}\CU)$, where $R_{\frac14}$ denotes the rotation of angle $\frac{\pi}{2}$. Then $Z$ is defined on $\{|\Im z|<\frac{\epsilon_1}{2\pi}\}$ and conjugates $(\alpha,S_{E}^V)$ to $(\alpha,B)$, with $B(z)=\begin{pmatrix}
e^{2 \pi \mathrm{i} \theta} & \kappa(z)\\[1mm]
0 & e^{-2 \pi \mathrm{i} \theta}
\end{pmatrix}$ for some continuous function $\kappa$ on $\{|\Im z|<\frac{\epsilon_1}{2\pi}\}$. Thus, $L(\alpha,S_{E}^V)$ vanishes identically on the strip $\{|\Im z|<\frac{\epsilon_1}{2\pi}\}$. Combining this fact with the previous claim, this concludes the proof of \eqref{eq vanish lyap exp}.

 Let $(m\alpha, \mathcal{A}_m^y)$ be the $m^{\mathrm{th}}$ iterate of $(\alpha, S_{E}^V(\cdot+{\rm i} y))$. Following Remark 2.1 in \cite{AvilaJito} (see also Lemma 3.1 of \cite{AYZ1}), by subadditivity and compactness of $\Sigma$, we see that for any $\delta>0$, there exists $C=C(\alpha,\epsilon_1, \delta)>0$ such that for any $m \geq 0$,
$$
\sup_{E\in \Sigma_{V,\alpha}}\sup_{ |y|<\frac{\epsilon_1}{2\pi}}|\mathcal{A}_m^y|_\T \leq C+ m\left( \sup_{E\in \Sigma_{V,\alpha}}\sup_{ |y|<\frac{\epsilon_1}{2\pi}}L(\alpha, S_{E}^V(\cdot+{\rm i} y))+\delta\right)=C+m \delta,
$$
which concludes the proof of Lemma \ref{subexp groowth}.
\end{proof}

\subsection{Global theory of one-frequency Schr\"odinger operators}\label{ARCGlobal}

Let us make a short review of  Avila's global theory of one-frequency ${\rm SL}(2,\R)-$cocycles \cite{Aglobal}.  Suppose that $A\in C^\omega(\T,{\rm SL}(2,\R))$ admits a holomorphic extension to $\{|\Im  z|<h\}$. Then for
$|\epsilon|<h$, we define $A_\epsilon \in
C^\omega(\T,{\rm SL}(2,\C))$ by $A_\epsilon(\cdot)=A(\cdot+i \epsilon)$.
 The cocycles which are not uniformly hyperbolic are classified
 into three classes: subcritical, critical, and supercritical. In
 particular, $(\alpha, A)$ is said to be
 subcritical if there exists $h>0$ such that
 $L(\alpha,A_{\epsilon})=0$ for $|\epsilon|<h.$

One main result of Avila's global theory is the following:
\begin{theorem}[Avila \cite{Aglobal}]\label{global-red}
Given any $\alpha\in\R\backslash\Q$, for a (measure-theoretically) typical $V\in C^\omega(\T,\R)$,
there exist $n\geq 1$ and a collection of points $a_1<b_1<\dots<a_n<b_n$ in the spectrum $\Sigma_{V,\alpha}$ such that $\Sigma_{V,\alpha}\subset \bigcup_{i=1}^n[a_i, b_i]$, and energies alternate between supercritical and subcritical along the sequence $\{\Sigma_{V,\alpha}\cap [a_i, b_i]\}_i$. Moreover, for any $i=1,\, \dots,\, n$, the set $\Sigma_{V,\alpha}\cap [a_i, b_i]$ is compact, and it depends continuously (in the Hausdorff topology) on $(\alpha,V).$
\end{theorem}

 A cornerstone  in Avila's global theory is the \textquotedblleft Almost Reducibility Conjecture\textquotedblright (ARC), which says that  $(\alpha,A)$ is almost reducible if it is subcritical.
Recall that the cocycle $(\alpha, A)$ is said to be {\it reducible} if it can be conjugated to a constant cocycle, i.e.,
there exist $Z\in C^\omega(\T^d,{\rm PSL}(2,\R))$ and $B\in {\rm SL}(2,\R)$ such that
$$Z(\cdot+\alpha)^{-1}A(\cdot)Z(\cdot)=B.$$
Moreover, $(\alpha, A)$ is (analytically) {\it almost reducible} if the closure of its analytic conjugates contains a constant.
The complete solution of ARC was recently given by Avila in \cite{A3, A2}. In the case where $\beta(\alpha)=0$, it is the following:

\begin{theorem}[Avila \cite{A2}]\label{arc-conjecture}
Given $\alpha\in\R\backslash\Q$ with $\beta(\alpha)=0$, and $A\in C^\omega(\T,{\rm SL}(2,\R))$, if $(\alpha,A)$ is subcritical, then it is almost reducible.
\end{theorem}

%
%
%

%

\section{Quantitative KAM scheme}\label{sec_Quantitative}
\noindent

In this section, we present a new KAM scheme and give a quantitative almost reducibility result for the local quasi-periodic linear system
\begin{equation}\label{linear_system}
\left\{
\begin{array}{l}
\dot{x}=(A_0+F_0(\theta))x\\
\dot{\theta}= \varpi
\end{array}
\right. ,
\end{equation}
where $A_0\in {\rm sl}(2,\R)$, and $F_0(\theta)$ is a perturbation. We also abbreviate (\ref{linear_system}) as $(\varpi,A_0+F_0)$.
There is a parallel result for the quasi-periodic cocycle introduced in Subsection \ref{subslyap}.
The reason why we chose to present the detailed proof in the continuous case is simply for comparing our Theorem \ref{sharpdecay} with the result of Damanik-Goldstein \cite{DG}.

 Given any $A_1, A_2\in {\CB}_{r_0}:=C^\omega_{r_0}(\T^{d}, {\rm sl}(2,\R))$ and $W\in C^{\omega}( \T^d,{\rm PSL}(2,\R))$, we say that $(\varpi,A_1)$ is {\it conjugated} to $(\varpi,A_2)$ by $W$ if
$\partial_\varpi W=A_1 W- W A_2$,
 where $ \partial_\varpi W:=\langle \varpi, \nabla W\rangle$. The system
 $(\varpi,A_1)$ is called {\it reducible} if it is conjugated to a constant system $(\varpi,B)$ with $B\in {\rm sl}(2,\R)$. It is called {\it almost reducible} if the closure of its analytical conjugates contains a constant system $(\varpi,B)$.

%

\subsection{Iteration}

Suppose that $A\in {\rm sl}(2,\R)$ and $F\in{\CB}_r$ with $|F|_{r}\leq \varepsilon$ for some $r,\varepsilon>0$.
For any given $r_+\in(0,r)$, the aim of the following argument is to find $\hat W\in C^{\omega}_{r_+}(\T^d, {\rm PSL}(2,\R))$, $A_{+}\in{\rm sl}(2,\R)$ and $F_{+}\in{\CB}_{r_{+}}$ with $|F_{+}|_{r_{+}}\ll \varepsilon$ such that
$(\varpi, A+F)$ is conjugated to  $(\varpi, A_++F_+)$ by $\hat W(\theta)$.
\begin{prop}\label{prop_iteration}
Let $\varpi\in{\rm DC}_d(\gamma,\tau)$. Given any $r_+\in(0,r)$, there is a constant $D_0=D_0(\gamma,\tau,d)>0$ such that if $\varepsilon$ satisfies
\begin{equation}\label{smallness_condition}
\varepsilon\leq D_0 \left(1+|A|^{120d(1+\frac{1}{\tau})}\right)(rr_+-r_+^2)^{800d(\tau+1)},
\end{equation}
then there exist $F_{+}\in {\CB}_{r_{+}}$,
$\hat W\in C_{r_{+}}^{\omega}(\T^d, {\rm PSL}(2,\R))$ and $A_{+}\in {\rm sl}(2,\R)$ such that
$(\varpi, A+F )$ is conjugated to $(\varpi, A_+ +F_+ )$ by $\hat W$.
Let $N:=\frac{2|\ln\varepsilon|}{r-r_+}$, and $\pm2\pi{\rm i}\xi$ be the two eigenvalues of $A$. Then we have the following:
\begin{itemize}
\item (Non-resonant case) Assume that
\begin{equation}\label{non-resonant_condition}
|2\xi-\la n,\varpi \ra|\geq \varepsilon^{\frac{1}{15}},  \quad \forall \   n\in\Z^d  \;\ with \;\ 0<|n|\leq N.
\end{equation}
In this case, we have the estimates:
$$|F_{+}|_{r_{+}}\leq \varepsilon^{2},\quad |\hat W- {\rm Id}|_{r_{+}}\leq 2 \varepsilon^{\frac12}, \quad |A_+-A|\leq \varepsilon^{\frac12}.$$
\item (Resonant case) If there exists $n_*\in\Z^{d}$ with  $0<|n_*|\leq N$ such that
\begin{equation}\label{resonant_condition}
|2\xi-\la n_*,\varpi\ra|< \varepsilon^{\frac{1}{15}},
\end{equation}
then $|F_{+}|_{r_{+}}\leq \varepsilon e^{-r_+\varepsilon^{-\frac{1}{18\tau}}}$,
 ${\rm deg} \hat W=n_*$, with the estimate
$$|\hat W|_{r''}\leq \frac{4\sqrt{|A|}}{\sqrt{\gamma}}|n_*|^{\frac\tau2} e^{\pi r''|n_*|}, \quad  \forall \  0< r''\leq r_+,$$
 and $A_+$ satisfies $|A_+|\leq \varepsilon^{\frac{1}{16}}$ with two eigenvalues $\pm 2\pi {\rm i}\xi_+$ satisfying $|\xi_+|\leq \varepsilon^{\frac{1}{16}}$.
Moreover, for $M:=\frac{1}{1+{\rm i}}\left(\begin{matrix}
1 & -{\rm i}\\
1 & {\rm i}
\end{matrix}\right)$, we have
\begin{equation}\label{A+}
A_{+}=2\pi M^{-1}\begin{pmatrix}
{\rm i}(\xi-\frac{\la n_*,\varpi\ra}2+ g_0) &  g_{*}  \\[2mm]
\overline{g_{*}} & -{\rm i}(\xi-\frac{\la n_*,\varpi\ra}2+ g_0)
\end{pmatrix} M
\end{equation}
for some $g_0\in\R$, $g_{*}\in \C$ with $|g_0|\leq \varepsilon^{\frac{15}{16}}$, $|g_{*}|\leq \varepsilon^{\frac{15}{16}} e^{-2\pi r|n_*|}$.
\end{itemize}
\end{prop}

\smallskip

Before giving the proof, we present a decomposition for the space ${\CB}_h$, $h>0$.
Given any $\varpi\in \R^d$, for $\eta>0$ and $\tilde A\in {\rm sl}(2,\R)$, we decompose ${\CB}_h={\CB}_h^{(\rm nre)}(\eta)\oplus {\CB}_h^{(\rm re)}(\eta)$ in such a way that for any $Y\in {\CB}_h^{(\rm nre)}(\eta)$,
\begin{equation}\label{non_resonant_condition}
\partial_\varpi Y, \;\  [\tilde A,Y]\in {\CB}_h^{(\rm nre)}(\eta), \quad |\partial_\varpi Y-[\tilde A,Y]|_h\geq \eta |Y|_h.
\end{equation}
Moreover, we let $\PP_{\rm nre}^{\eta}Y$ and $\PP_{\rm re}^{\eta}Y$ be the standard projections from ${\CB}_h$ onto ${\CB}_h^{(\rm nre)}(\eta)$ and ${\CB}_h^{(\rm re)}(\eta)$ respectively.

Associated with this decomposition, we have
\begin{lemma}[Hou-You \cite{HouYou}]\label{lemma_HouYou}
Given $\tilde F\in{\CB}_h$ with $|\tilde F|_h\leq \tilde \varepsilon$, assume that $\tilde \varepsilon\in (0, 10^{-8})$ and $\eta\geq \tilde \varepsilon^{\frac14}$.
There exist $Y\in {\CB}_{h}$, $G \in {\CB}^{({\rm re})}_{h}(\eta)$, with the estimates $|Y|_{h}\leq \tilde \varepsilon^{\frac12}$, $|G|_{h}\leq 2\tilde \varepsilon$, such that
$(\varpi,\tilde A+\tilde F)$ is conjugated to $(\varpi,\tilde A+G)$ by $e^{Y(\theta)}$.
\end{lemma}

\subsubsection{Non-resonant case}


Consider the linear system $
(\varpi, A+F)$, where $\varpi\in {\rm DC}_d(\gamma,\tau)$, $A\in {\rm sl}(2,\R)$ has two eigenvalues $\pm2\pi{\rm i}\xi$, and $F\in{\CB}_r$ satisfies $|F|_{r}\leq \varepsilon$. For $\tilde A=A$
and $\eta=2\varepsilon^{\frac14}$, we focus on the decomposition
 ${\CB}_r={\CB}_r^{(\rm nre)}(\eta)\oplus {\CB}_r^{(\rm re)}(\eta)$.
 The key observation is the following:

\begin{lemma}\label{lemma_nonresonant_diophantine}
Assume that $(\ref{non-resonant_condition})$ holds. For any $G\in {\CB}^{({\rm re})}_r(\eta)$, we have $\widehat G(n)=0$ if $0<|n|\leq N$.
\end{lemma}

\begin{proof}  Since $\varpi \in{\rm DC}_d(\gamma,\tau)$, we have
\begin{equation}\label{diophantine_diagonal}
|\la n, \varpi \ra|\geq \frac{\gamma}{N^\tau} \geq\varepsilon^{\frac1{15}},\quad  \forall \  n\in\Z^d \ {\rm with}  \  0<|n|\leq N,
\end{equation}
if $\varepsilon$ satisfies $(\ref{smallness_condition})$. Combining $(\ref{diophantine_diagonal})$ with the non-resonant condition $(\ref{non-resonant_condition})$, it is easy to  check that
$$|\partial_\varpi G_N-[A,G_N]|_h\geq \varepsilon^{\frac{1}{5}}  |G_N|_h,$$
holds for any
$G_N:=\sum_{0<|n|\leq N}\widehat G(n)e^{2\pi {\rm i}\la n, \theta \ra}$.
One can consult Lemma 1 of \cite{E92} for details.
Hence $G_N\in {\CB}^{({\rm nre})}_r(\eta)$, which means
$$\widehat{\PP_{\rm re}^{\eta} G }(n)=0 \ {\rm if}  \  0<|n|\leq N. $$
This finishes the proof.\end{proof}

Applying Lemma \ref{lemma_HouYou}, we know that $(\varpi, A+F)$ is congugated to a cocycle $(\varpi, A+G)$ with $G\in {\CB}^{(\rm re)}_r(\eta)$.
By Lemma \ref{lemma_nonresonant_diophantine}, we know that the only non-vanishing Fourier mode $\widehat G(n)$ with $|n|\leq N$ is $\widehat G(0)$.
Recalling that $N=\frac{2|\ln\varepsilon|}{r-r_+}$, for $\varepsilon$ small enough, we have 
$$|G-\widehat G(0)|_{r_+}\leq \sum_{|n|>N}|\widehat G(n)| e^{2\pi r_+ |n| } \leq  |G|_r \, e^{-2\pi N(r-r_+)} \sum_{j=0}^{d-1}\frac{(d-1)!}{j!} \frac{N^j}{(r-r_+)^{d-j}} \leq \varepsilon^{2}.$$
Let $\hat W:=e^Y$, $A_+:=A+\widehat G(0)$ and $F_+:=G- \widehat G(0)$. This concludes the proof of Proposition \ref{prop_iteration} for the non-resonant case.

\subsubsection{Resonant case}

Recall that $A\in {\rm sl}(2,\R)$ has two eigenvalues $\pm 2\pi{\rm i}\xi$ with $\xi \in \R\cup {\rm i}\R$.
In view of the Diophantine property (\ref{dio}) of $\varpi\in\R^d$, it is easy to see that the resonant case of $A$ does not occur unless $A$ is of elliptic type, i.e., $\xi\in \R\backslash\{0\}$.
Moreover, the resonant condition (\ref{resonant_condition}) implies that
$$|\xi|\geq \frac12(|\la n_*, \varpi\ra|-\varepsilon^{\frac1{15}})\geq \frac{\gamma}{2|n_*|^{\tau}}-\frac12\varepsilon^{\frac1{15}} \geq \frac{\gamma}{3|n_*|^{\tau}}.$$ In view of Lemma 8.1 in \cite{HouYou}, there is $C_{A}\in {\rm SL}(2,\R)$ with
$$
|C_A|\leq 2\sqrt{\frac{|A|}{|\xi|}} \leq \frac{2\sqrt{3}\sqrt{|A|}}{\sqrt{\gamma}}|n_*|^{\frac\tau2}
$$
such that $A=C_{A} \begin{pmatrix}
0 & 2\pi\xi \\
-2\pi\xi & 0
\end{pmatrix} C_{A}^{-1}$.
Let $\tilde F:=C_{A}^{-1}F C_{A}$. Then we have
$$|\tilde F|_r\leq \frac{12|A|}{\gamma}|n_*|^\tau\varepsilon=:\tilde\varepsilon.$$
Let $\tilde A:=\begin{pmatrix}
0 & 2\pi\xi \\
-2\pi\xi & 0
\end{pmatrix}$ and $\eta:=2\tilde\varepsilon^{\frac14}$. Applying Lemma \ref{lemma_HouYou}, one can conjugate  $(\varpi, \tilde A+\tilde F)$ to  $(\varpi, \tilde A+G)$  with $G\in {\CB}^{(\rm re)}_r(\eta)$. Let us now characterize the precise structure of $G\in {\CB}_r^{(\rm re)}(\eta)$.

\begin{lemma}\label{lemma_structrue_resonant}
Assume that (\ref{resonant_condition}) holds.
For any $G\in{\CB}^{({\rm re})}_{r}(\eta)$ with $|G|_r\leq 2\tilde\varepsilon$, there exist $g_0\in\R$, $g_{*}\in \C$ and $P\in C^{\omega}_{r}(\T^d, {\rm su}(1,1))$
satisfying
$$
|g_0|\leq \varepsilon^{\frac{15}{16}}, \quad |g_{*}|\leq\varepsilon^{\frac{15}{16}} e^{-2\pi r |n_*|};\qquad |P|_{r_+}\leq  \varepsilon e^{-r_+\varepsilon^{-\frac{1}{16\tau}}}, \quad \forall  \  0<r_+<r,
$$
such that
$M G(\theta) M^{-1}= 2\pi\begin{pmatrix}
{\rm i} g_0 & g_{*}e^{2\pi{\rm i}\la n_*,\theta\ra} \\[1mm]
\overline{g_{*}}e^{-2\pi{\rm i}\la n_*,\theta\ra} & -{\rm i} g_0
\end{pmatrix} + P(\theta)$.
\end{lemma}
\begin{proof}  First we note that $M=\frac{1}{1+{\rm i}}\left(\begin{matrix}
1 & -{\rm i}\\
1 & {\rm i}
\end{matrix}\right)$ induces an isomorphism from ${\rm sl}(2,\R)$ to ${\rm su}(1,1)$, which is  the group of matrices of the form
$\left(
\begin{array}{ccc}
 {\rm i} t &  \nu\\
\bar{ \nu} &  -{\rm i} t
 \end{array}\right)$
 with $t\in \R$, $\nu\in \C$. Thus for any
 $\widehat G(n)\in {\rm sl}(2,\R)$,  there exist $g_{11}(n)\in\R$, $g_{12}(n)\in\C$ such that
$$M \widehat G(n) M^{-1}=\begin{pmatrix}
{\rm i}g_{11}(n) & g_{12}(n)   \\[1mm]
\overline{g_{12}(n)} & -{\rm i}g_{11}(n)
\end{pmatrix}.$$
 By the decay property of Fourier coefficients, we get
\begin{equation}\label{esti_g}
|g_{11}(n)|,\,  |g_{12}(n)|\leq 4 \tilde\varepsilon e^{-2\pi r |n|}.
\end{equation}
Since $\tilde A=M^{-1} \begin{pmatrix}
2\pi {\rm i}\xi & 0   \\
0 & -2\pi {\rm i}\xi
\end{pmatrix} M$, a direct computation shows that
$$M (\partial_\varpi G-[\tilde A,G]) M^{-1}=2\pi{\rm i}\sum_{n\in\Z^d}
\begin{pmatrix}
{\rm i} \la n,\varpi \ra g_{11}(n) & (\la n,\varpi \ra+2\xi)g_{12}(n)\\[1mm]
(\la n,\varpi \ra-2\xi)\overline{g_{12}(n)} & -{\rm i} \la n,\varpi \ra g_{11}(n)
\end{pmatrix} e^{2\pi{\rm i}\la n,\theta\ra}.$$
By the definition of ${\CB}_r^{({\rm nre})}(\eta)$ in (\ref{non_resonant_condition}), for any $G\in  {\CB}^{({\rm nre})}_{r}(\eta)$, $M G(\theta) M^{-1}$ equals 
$$
\sum_{n\notin\Lambda_1} \begin{pmatrix}
{\rm i} g_{11}(n) & 0   \\[1mm]
  0 & -{\rm i} g_{11}(n)
\end{pmatrix} e^{2\pi{\rm i}\la n,\theta\ra}
 + \sum_{n\notin\Lambda_2}\begin{pmatrix}
0 & g_{12}(-n)e^{-2\pi{\rm i}\la n,\theta\ra} \\[1mm]
\overline{g_{12}(n)} e^{2\pi{\rm i}\la n,\theta\ra} & 0
\end{pmatrix}
$$
with $\Lambda_1:=\{n\in\Z^{d}:|\la n,\varpi \ra|< 2\tilde\varepsilon^{\frac{1}{4}}\}$, $\Lambda_2:=\{n\in\Z^{d}:|2\xi-\la n,\varpi \ra|< 2\tilde\varepsilon^{\frac{1}{4}}\}$.
Hence $G\in {\CB}^{({\rm re})}_{r}(\eta)$ means $M G(\theta) M^{-1}$ has the form
\begin{equation}\label{form_resonant_elliptic}
\sum_{n\in\Lambda_1} \begin{pmatrix}
{\rm i} g_{11}(n) & 0   \\[1mm]
  0 & - {\rm i} g_{11}(n)
\end{pmatrix} e^{2\pi{\rm i}\la n,\theta\ra}
 + \sum_{n\in \Lambda_2}\begin{pmatrix}
0 & g_{12}(-n)e^{-2\pi{\rm i}\la n,\theta\ra} \\[1mm]
\overline{g_{12}(n)} e^{2\pi{\rm i}\la n,\theta\ra} & 0
\end{pmatrix}
\end{equation}

\begin{claim} We have the following observations:
\begin{align}
\Lambda_1\cap \{n\in\Z^{d}: |n| \leq  \gamma^{\frac1\tau}\varepsilon^{-\frac{1}{15\tau}}\}&=\{0\},\label{resonant_site_1}\\
\Lambda_2\cap \{n\in\Z^{d}: |n| \leq 2^{-\frac1\tau} \gamma^{\frac1\tau}\varepsilon^{-\frac{1}{15\tau}}-N\}&=\{n_*\}.\label{resonant_site_2}
\end{align}
\end{claim}
\begin{proof}Indeed, given any $n\in\Lambda_1$ and $n\neq0$,
we have 
$$\frac{\gamma}{|n|^\tau}< |\la n,\varpi \ra|<2\tilde\varepsilon^{\frac{1}{4}}< \varepsilon^{\frac{1}{15}}.$$
Therefore, $|n|> \gamma^{\frac{1}{\tau}}\varepsilon^{-\frac{1}{15\tau}}$, which gives (\ref{resonant_site_1}).

For any $n'_*\neq n_*$ with $|2\xi-\la n'_*,\varpi\ra|< \varepsilon^{\frac{1}{15}}$, since $\varpi\in {\rm DC}_d(\gamma,\tau)$, we have
$$ \frac{\gamma}{|n'_*-n_*|^{\tau}}\leq|\la n'_*-n_*,\varpi \ra|< 2\varepsilon^{\frac{1}{15}},$$
which implies $|n'_*|> 2^{-\frac1\tau} \gamma^{\frac1\tau}\varepsilon^{-\frac{1}{15\tau}}-N > N$ under the hypothesis (\ref{smallness_condition}), and thus $(\ref{resonant_site_2})$ follows.
\end{proof}

Let ${\CN}_1:=  \gamma^{\frac1\tau}\varepsilon^{-\frac{1}{15\tau}}$ and ${\CN}_2:=2^{-\frac1\tau} \gamma^{\frac1\tau}\varepsilon^{-\frac{1}{15\tau}}-N$. In view of (\ref{resonant_site_1}) and (\ref{resonant_site_2}),
the two parts of $M G(\theta) M^{-1}$ given in (\ref{form_resonant_elliptic}) can be decomposed as
\begin{align*}
& \sum_{n\in\Lambda_1} \begin{pmatrix}
 {\rm i}g_{11}(n) & 0   \\[1mm]
  0 & -{\rm i}g_{11}(n)
\end{pmatrix} e^{2\pi {\rm i}\la n,\theta\ra}\\
=&\begin{pmatrix}
{\rm i}g_{11}(0) & 0   \\[1mm]
  0 & -{\rm i}g_{11}(0)
\end{pmatrix}+\sum_{n\in \Lambda_1\atop{|n|>{\CN}_1}} \begin{pmatrix}
{\rm i}g_{11}(n) & 0   \\[1mm]
  0 & -{\rm i}g_{11}(n)
\end{pmatrix} e^{2\pi {\rm i}\la n,\theta\ra},\\
&  \sum_{n\in\Lambda_2}\begin{pmatrix}
0 & g_{12}(n)e^{2\pi {\rm i}\la n,\theta\ra} \\[1mm]
\overline{g_{12}(n)} e^{-2\pi {\rm i}\la n,\theta\ra} & 0
\end{pmatrix}\\
=&  \begin{pmatrix}
0 & g_{12}(n_*)e^{2\pi {\rm i}\la n_*,\theta\ra} \\[1mm]
\overline{g_{12}(n_*)}e^{-2\pi {\rm i}\la n_*,\theta\ra} & 0
\end{pmatrix} + \sum_{n\in\Lambda_2\atop{|n|> {\CN}_2}}\begin{pmatrix}
0 & g_{12}(n)e^{2\pi {\rm i}\la n,\theta\ra} \\[1mm]
\overline{g_{12}(n)} e^{-2\pi {\rm i}\la n,\theta\ra} & 0
\end{pmatrix}.
\end{align*}
Let $g_0:=\frac{1}{2\pi}g_{11}(0)$, $g_{*}:=\frac{1}{2\pi} g_{12}(n_*)$ and let $P(\theta)$ be
$$\sum_{n\in \Lambda_1\atop{|n|> {\CN}_1}} \begin{pmatrix}
{\rm i} g_{11}(n) & 0   \\[1mm]
  0 & -{\rm i} g_{11}(n)
\end{pmatrix} e^{2\pi {\rm i}\la n,\theta\ra}+ \sum_{n\in\Lambda_2\atop{|n|> {\CN}_2}}\begin{pmatrix}
0 & g_{12}(-n)e^{-2\pi {\rm i}\la n,\theta\ra} \\[1mm]
\overline{g_{12}(n)} e^{2\pi {\rm i}\la n,\theta\ra} & 0
\end{pmatrix}.$$
By (\ref{esti_g}), and noting that ${\CN}_2<{\CN}_1$, we have
$$|P|_{r_+}\leq (d-1)! \left({\CN}_2+\frac{1}{r-r_+}\right)^d \cdot 4 \tilde\varepsilon e^{-2\pi (r-r_+){\CN}_2}\leq \varepsilon e^{-r_+\varepsilon^{-\frac{1}{16\tau}}}.$$
\end{proof}


Now we define
$$Z(\theta):=e^{-\frac{\la n_*,\theta\ra}{2\xi}\tilde A}= M^{-1}\begin{pmatrix}
e^{-\pi {\rm i}\la n_*,\theta\ra} & 0\\
0 & e^{\pi {\rm i}\la n_*,\theta\ra}
\end{pmatrix} M  .$$
Obviously, $Z\in C^\omega_{r}(\T^d,{\rm PSL}(2,\R))$, and for any $r''\in (0,r)$, $|Z|_{r''}\leq 2e^{\pi r''|n_*|}$.
Given any $G\in{\CB}^{({\rm re})}_{r}(\eta)$, we have thus
$$\partial_\varpi Z= (\tilde A + G) Z-Z\left[\left(1-\frac{\la n_*,\varpi\ra}{2\xi}\right)\tilde A + Z^{-1} G Z\right].$$
By a direct calculation, we get
$$Z(\theta)^{-1}G(\theta)Z(\theta)=2\pi M^{-1}\begin{pmatrix}
 {\rm i} g_0 &  g_{*}  \\[2mm]
\overline{g_{*}} & - {\rm i} g_0
\end{pmatrix}M + Z(\theta)^{-1}P(\theta)Z(\theta).$$
Let $F_{+}:=Z^{-1}PZ$
and $$A_{+}:=\left(1-\frac{\la n_*,\varpi\ra}{2\xi}\right)\tilde A + 2\pi M^{-1}\begin{pmatrix}
 {\rm i} g_0 &  g_{*}  \\[2mm]
\overline{g_{*}} & - {\rm i} g_0
\end{pmatrix}M.$$
Therefore the system
$(\varpi,\tilde A + G)$ is conjugated to $(\varpi,A_{+} + F_{+})$ by $Z$, with estimates
\begin{eqnarray*}
&&|A_+|\leq 2\pi \varepsilon^{\frac{1}{15}}+4\pi \varepsilon^{\frac{15}{16}}(1+e^{-2\pi r |n_*|})\leq \varepsilon^{\frac{1}{16}},\\
&&|F_+|_{r_{+}}\leq
4  e^{2\pi r_+|n_*|}\cdot
\varepsilon e^{-r_+\varepsilon^{-\frac{1}{16\tau}}}
\leq \varepsilon e^{-r_+\varepsilon^{-\frac{1}{18\tau}}},\quad \forall \  0<r_+<r.
\end{eqnarray*}

Let $\hat W:=C_A \cdot e^{Y} \cdot Z$ with $e^Y$ obtained in Lemma \ref{lemma_HouYou}.
Obviously, ${\rm deg} \hat W=n_*$. Then we finish the proof for the resonant case of Proposition \ref{prop_iteration}.
%
%
%
%
%
%

%

\subsection{Reducibility of quasi-periodic linear systems}
Consider the  quasi-periodic linear system $(\varpi, A_0 + F_0)$. Denote by  $\rho(\varpi,A_0+F_0)$  its rotation number (we refer to \cite{E92,JM} for the detailed definition).
In the same way as in \cite{E92}, one can prove that if $\rho(\varpi,A_0+F_0)$ is Diophantine or rational with respect to $\varpi$, then the system $(\varpi, A_0 + F_0)$ is reducible. For our purpose, in this paper we will specially focus our attention on the quantitative reducibility in the case where
 $\rho(\varpi, A_0+F_0)=\frac{\la k, \varpi\ra}2$ for some $k\in\Z^{d}\backslash\{0\}$.


\begin{theorem}\label{thm_gap_edge_algebra}
Assume that $\varpi\in{\rm DC}_d(\gamma,\tau)$, $d\geq 2$. Given any $r\in (0,r_0)$, there is $\varepsilon_*=\varepsilon_*(|A_0|, \gamma, \tau, r_0, r, d)>0$ such that if $|F_0|_{r_0}=\varepsilon_0 <\varepsilon_*$, then the following holds.
\begin{enumerate}
\item The system  $(\varpi,A_0+F_0)$ is almost reducible in the strip $|\Im \theta|<r$.
\item If $\rho(\varpi,  A_0+F_0)=\frac{\la k,\varpi\ra}{2}$ for $k\in\Z^{d}\backslash\{0\}$, and $(\varpi,A_0+F_0)$ is not uniformly hyperbolic, then there exists
$W\in C_{r}^{\omega}(\T^d, {\rm PSL}(2,\R))$ such that
$$\partial_\varpi W=(A_0+F_0)W - W \begin{pmatrix}
0 & \kappa \\
0 & 0
\end{pmatrix}$$
with $|\kappa|\leq \varepsilon_0^{\frac34} e^{-2\pi r|k|}$. Moreover, for any $r''\in (0, r]$, $|W|_{r''}\leq D_1 e^{\frac{3\pi r''}{2} |k|}$ with $D_1=D_1(\gamma, \tau, |A_0|, r_0, d)>0$.
\end{enumerate}
\end{theorem}

\begin{proof} The result can  be  proved by applying Proposition \ref{prop_iteration} iteratively.
Take $\varepsilon_0$, $r_0$ and $r$ as above.
Assume that we are at the $(j+1)^{\rm th}$ KAM step,
where we have $A_{j}\in {\rm sl}(2,\R)$ with eigenvalues $\pm2\pi {\rm i}\xi_{j}$ and $F_j\in {\CB}_{r_j}$ satisfying
$|F_j|_{r_j}\leq \varepsilon_j$ for some $\varepsilon_j\leq \varepsilon_0$.
Let $\tilde r=\frac{r_0+r}{2}$. Then
we define
\begin{equation}\label{sequences}
r_j - r_{j+1}:=\frac{r_0-\tilde r}{4^{j+1}} ,\quad N_j:=\frac{ 2|\ln\varepsilon_j|}{r_j - r_{j+1}}=\frac{2\cdot 4^{j+1} |\ln\varepsilon_j|}{r_0-\tilde r}.
\end{equation}
If $\varepsilon_j$ is sufficiently small such that the condition (\ref{smallness_condition}) is satisfied for $\varepsilon=\varepsilon_j$, $r=r_j$, $r_+=r_{j+1}$ and $A=A_j$, then, by Proposition \ref{prop_iteration},
we can construct
$$
\hat W_{j}\in C_{r_{j+1}}^{\omega}(\T^d, \mathrm{PSL}(2,\R) ),\;\  A_{j+1}\in {\rm sl}(2,\R), \;\ F_{j+1}\in {\CB}_{r_{j+1}},
$$
such that $(\varpi, A_j + F_j)$ is conjugated to $(\varpi, A_{j+1} + F_{j+1})$ by $\hat W_{j}(\theta)$.
Moreover,
\begin{itemize}
\item if for any $n\in\Z^{d}$ with $0<|n|\leq N_j$,  we have
$|2\xi_j-\la n,\varpi\ra|\geq \varepsilon_j^{\frac{1}{15}}$, then
\begin{equation}\label{esti_non-resonant}
|A_{j+1}-A_j|\leq \varepsilon_j^{\frac12},\;\ |\hat W_j- {\rm Id}|_{r_{j+1}}\leq 2\varepsilon_j^{\frac12},\;\ |F_{j+1}|_{r_{j+1}}\leq \varepsilon_{j+1}:= \varepsilon_j^{2};
\end{equation}
\item if there is $n_j\in\Z^{d}$ with $0<|n_j|\leq N_j$ such that
$|2\xi_j-\la n_j,\varpi \ra|< \varepsilon_j^{\frac{1}{15}}$, then
\begin{equation}\label{esti_resonant}
|A_{j+1}|\leq \varepsilon_j^{\frac{1}{16}},\quad |F_{j+1}|_{r_{j+1}}\leq \varepsilon_{j+1}:= \varepsilon_j e^{-r_{j+1} \varepsilon_j^{-\frac{1}{18\tau}}}, \quad {\rm deg}\hat W_j=n_j
\end{equation}
and for any $r''\in(0,r_{j+1}]$,
\begin{equation}\label{esti_resonant_W}
|\hat W_j|_{r''}\leq 4\sqrt\frac{|A_j|}{\gamma}|n_j|^{\frac\tau2} e^{\pi r'' |n_j|}.
\end{equation}
\end{itemize}
In view of (\ref{esti_non-resonant}) and (\ref{esti_resonant}), one sees that $\varepsilon_j\leq \varepsilon_0^{2^j}$ and $|A_j|\leq 2|A_0|$ for any $j\geq 0$.
So, if $\varepsilon_0$ is sufficiently small (depending on $|A_0|,  \gamma,  \tau,  r_0,  r,  d$) such that (\ref{smallness_condition}) holds, then  Proposition \ref{prop_iteration} can be applied iteratively.
Indeed,  $\varepsilon_j$ on the left side of the inequality (\ref{smallness_condition}) decays at least super-exponentially with $j$, while
$(r_j-r_{j+1})^{800d(\tau+1)}$ on the right  side  decays exponentially with $j$.
Hence $(\varpi,A_0+F_0)$ is almost reducible.

Assume that there are at least two resonant steps in the above almost reducibility precedure.
Let us focus on two consecutive resonant steps, say the $(j_i+1)^{\rm th}$ and $(j_{i+1}+1)^{\rm th}$. At the $(j_{i+1}+1)^{\rm th}-$step, the resonance condition  implies
$\left|\xi_{j_{i+1}}- \frac{\la n_{j_{i+1}},\varpi\ra}2\right|\leq \frac12\varepsilon_{j_{i+1}}^{\frac{1}{15}}$,
hence $|\xi_{j_{i+1}}|>\frac{\gamma}{3|n_{j_{i+1}}|^{\tau}}$.
On the other hand, according to Proposition \ref{prop_iteration}, after the $(j_{i}+1)^{\rm th}-$step, $|\xi_{j_{i}+1}|\leq \varepsilon_{j_i}^\frac{1}{16}$. By (\ref{esti_non-resonant}),
$|\xi_{j_{i+1}}|\leq 2\varepsilon^{\frac{1}{16}}_{j_{i}}\leq \frac{\varepsilon^{\frac{1}{18}}_{j_{i}}\gamma}{3|n_{j_i}|^{\tau}}.$
Thus
\begin{equation}\label{skip_of_resonances}
|n_{j_{i+1}}|\geq \varepsilon_{j_i}^{-\frac{1}{18\tau}} |n_{j_{i}}|.
\end{equation}
Recall that ${\rm deg}\hat W_{j+1}=n_{j}$ if the $(j+1)^{\rm th}-$step is resonant.
In view of (\ref{rot-conj}) and (\ref{skip_of_resonances}), we deduce that there are at most finitely many resonant steps in the above almost reducibility procedure under the hypothesis $\rho_{(\varpi, A_0+F_0)}=\frac{\la k,\, \varpi\ra}2$.
This means that we can find a sequence $(\hat W_{l})_{l\in\N}$ with $\hat W_{l}\in C^\omega_{r_{l+1}}(\T^d, {\rm PSL}(2,\R))$, in which the resonant case occurs only finitely many times.
By the estimate of $\hat W_j$ in \eqref{esti_non-resonant} and the sequence $(r_j)_{j\in\N}$ given in (\ref{sequences}), we see that the product
$\prod_{l=0}^j\hat W_{l+1}$ converges to some
$W\in C^\omega_{\tilde r}(\T^d, {\rm PSL}(2,\R))$ such that
$\partial_\varpi W=(A_0+F_0)W-W B$ for some $B\in {\rm sl}(2,\R)$ with $\rho_{(\varpi,B)}=0$.

Assuming that there are $s+1$ resonant steps, associated with integers vectors
$$n_{j_0},\dots , n_{j_s}\in \Z^{d}, \qquad 0<|n_{j_i}|\leq N_{j_i}, \;\  i=0,1,\dots,s,$$
then $k=n_{j_0}+ \dots + n_{j_s}$.
In view of the inequalities (\ref{skip_of_resonances}) and the fact that
$$|n_{j_s}|-\sum_{i=0}^{s-1}|n_{j_i}|\leq |k| \leq |n_{j_s}|+\sum_{i=0}^{s-1}|n_{j_i}|,$$
we get
$(1-2\varepsilon_{0}^{\frac{1}{18\tau}})|n_{j_s}|\leq |k| \leq (1+2\varepsilon_{0}^{\frac{1}{18\tau}})|n_{j_s}|$.
By (\ref{esti_resonant_W}), for any $r''\in (0,r]$,
$$|W|_{r''}
\leq  2|\hat W_{j_0+1}|_{r''}\cdots|\hat W_{j_s+1}|_{r''}
\leq \frac{2^{2s+3}}{\sqrt{\gamma^{s+1}}}\prod_{i=0}^s |A_{j_i}| |n_{j_i}|^{\frac\tau2}  e^{\pi r'' |n_{j_i}|}
\leq D_1 e^{\frac{3\pi r''}{2} |k|}$$
for some $D_1=D_1(\gamma, \tau, |A_0|, r_0, d)>0$.

Now we  estimate  the constant matrix $B$. Since we have assumed that the initial system $(\varpi,A_0+F_0)$ is not uniformly hyperbolic, one concludes that $B$ can not be a hyperbolic matrix. As we have proved, $\rho(\varpi,B)=0$, thus ${\rm det}B=0$. Assume that
 $B=
\begin{pmatrix}
B_{11} & B_{12} \\
B_{21} & -B_{11}
\end{pmatrix}.$
Then there exists $\phi\in\T$ such that $R_{-\phi}B R_{\phi}=
\begin{pmatrix}
0 & B_{21}-B_{12} \\
0 & 0
\end{pmatrix}$.
By taking $W R_{\phi}$ instead of $W$ and $\kappa=B_{21}-B_{12}$, we know that $(\varpi, A_0+F_0)$ is conjugated to  $\tilde B=\begin{pmatrix}
0 & \kappa \\
0 & 0
\end{pmatrix}$.

To estimate $|\kappa|$, let us focus on $(\varpi, A_{j_s+1}+F_{j_s+1})$, i.e., the system just after the last resonant step. In view of (\ref{A+}), we have 
$$A_{j_s+1}=2\pi M^{-1}\begin{pmatrix}
{\rm i}\left(\xi_{j_s}-\frac{\la n_{j_s},\, \varpi\ra}2+ q_{0}\right) & q_{j_s}  \\[2mm]
\overline{q_{j_s}} & -{\rm i}\left(\xi_{j_s}-\frac{\la n_{j_s},\, \varpi\ra}2 + q_{0}\right)
\end{pmatrix} M,$$
with $q_{0}\in \R, \, q_{j_s}\in \C$ satisfying
\begin{equation}\label{off-diagonal_js+1}
\left|\xi_{j_s}-\frac{\la n_{j_s},\, \varpi\ra}2 + q_{0}\right|\leq  \varepsilon^{\frac{1}{15}}_{j_s}+\varepsilon^{\frac{15}{16}}_{j_s}\leq 2\varepsilon^{\frac{1}{15}}_{j_s},\quad |q_{j_s}|\leq\varepsilon^{\frac{15}{16}}_{j_s} e^{-2\pi r_{j_s}|n_{j_s}|}.
\end{equation}
Since $(j_s+1)^{\rm th}-$step is the last resonant step, we have $|A_{l+1}-A_l|\leq \varepsilon_{l}^{\frac12}$, $l\geq j_s+1$. Hence, noting that $\varepsilon_{j_s+1}=\varepsilon_{j_s} e^{-r_{j_s+1} \varepsilon_{j_s}^{-\frac{1}{18\tau}}}$, we get
$$|A_{j_s+1}-B|\leq \sum_{l=j_s+1}^{\infty}  |A_{l+1}-A_l|\leq 2\varepsilon^{\frac12}_{j_s} e^{-\frac{r_{j_s+1}}{2}\varepsilon_{j_s}^{-\frac{1}{18\tau}}}.$$
Rewrite $B$ as $B=M^{-1}\begin{pmatrix}
{\rm i} \beta_{11} & \beta_{12}  \\[1mm]
\overline{\beta_{12}} & -{\rm i} \beta_{11}
\end{pmatrix} M$ with $\beta_{11} \in \R$, $\beta_{12} \in\C$.
In view of (\ref{off-diagonal_js+1}), we have
$$|\beta_{12}| \leq
2\pi \varepsilon^{\frac{15}{16}}_{j_s} e^{-2\pi r_{j_s} |n_{j_s}|} +4\varepsilon^{\frac12}_{j_s} e^{-\frac{r_{j_s+1}}{2}\varepsilon_{j_s}^{-\frac{1}{18\tau}}}
\leq \varepsilon^{\frac78}_{j_s}e^{-2\pi  r_{j_s+1}|n_{j_s}|}.$$
Then we have $|\beta_{11}| \leq  \varepsilon^{\frac78}_{j_s}e^{-2\pi r_{j_s+1}|n_{j_s}|}$ since ${\rm det}B=0$.
So
$$|B_{12} |, \, |B_{21} |\leq  2 \varepsilon^{\frac78}_{j_s}e^{-2\pi r_{j_s+1}|n_{j_s}|} \leq \frac12\varepsilon^{\frac34}_{j_s}e^{-2\pi r_{j_s+1}|n_{j_s}|}. $$
Hence, in view of the fact  $|k|\leq (1+2\varepsilon_{0}^{\frac{1}{18\tau}})|n_{j_s}|$,
$$|\kappa|=|B_{21}-B_{12}|\leq \varepsilon^{\frac34}_{j_s}e^{-2\pi r_{j_s+1}|n_{j_s}|}\leq \varepsilon^{\frac34}_{j_s}e^{-\frac{2\pi \tilde r |k|}{1+2\varepsilon_{0}^{1/18\tau}}}\leq \varepsilon^{\frac34}_{j_s}e^{-2\pi r|k|}.$$
\end{proof}

\subsection{Reducibility of quasi-periodic cocycles}

In analogy with Theorem \ref{thm_gap_edge_algebra} for quasi-periodic linear systems,
we obtain a similar result for  quasi-periodic cocycles
$$\begin{pmatrix}
    u_{n+1} \\
    u_n
  \end{pmatrix}=(A_0+F_0(\theta+n\alpha))\begin{pmatrix}
    u_n \\
    u_{n-1}
  \end{pmatrix}.$$

\begin{theorem}\label{thm_gap_edge_SL}
Let $\alpha\in{\rm DC}_d(\gamma,\tau)$ and $A_0\in {\rm SL}(2,\R)$.
Given $r\in (0,r_0)$, there is $\varepsilon_*=\varepsilon_*(|A_0|, \gamma, \tau, r_0, r, d)>0$ such that if $|F_0|_{r_0}=\varepsilon_0 <\varepsilon_*$, then the following holds:
\begin{enumerate}
\item The quasi-periodic ${\rm SL}(2,\R)$ cocycle $(\alpha,A_0+F_0)$ is almost reducible in the strip $|\Im \theta|<r$.
\item  If $2\rho(\alpha,\, A_0+F_0)-\la k,\alpha\ra\in\Z$ for $k\in\Z^{d}\backslash\{0\}$, and $(\alpha,A_0+F_0(\cdot))$ is not uniformly hyperbolic, then
 there exists $W\in C_{r}^{\omega}(\T^d, {\rm PSL}(2,\R))$, such that
$$ W(\cdot +\alpha)^{-1}(A_0+F_0(\cdot)) W(\cdot)=B=\begin{pmatrix}
1 & \kappa \\
0 & 1
\end{pmatrix},$$
 with $|\kappa|\leq \varepsilon_0^{\frac34} e^{-2\pi r|k|}$.
Moreover, for any $r''\in(0,r]$, $|W|_{r''}\leq  D_1 e^{\frac{3\pi r''}{2} |k|}$ with $D_1=D_1(\gamma, \tau, |A_0|,r_0, d)>0$.
\end{enumerate}
\end{theorem}

\begin{remark}\label{uniformcons}
Let $\gamma,\tau>0$ be fixed. If $A_0$ varies in ${\rm SO}(2,\R)$, then $\varepsilon_*=\varepsilon_*( \gamma, \tau, r_0, r, d)>0$ can be taken uniform with respect to $A_0$.
\end{remark}

\section{Almost localization and duality argument}

While in the previous part, we considered the case of a Diophantine vector of frequencies, here, we let $\alpha\in \R$ be a frequency satisfying $\beta(\alpha)=0$, and we study the reducibility of associate quasi-periodic Schr\"odinger cocycles by non-perturbative methods. The quantitative  statement we prove is based on two importants ingredients: quantitative almost localization properties of dual Schr\"odinger operators and quantitative Aubry duality.

From now on, in the formulations and proofs of various assertions about quasi-periodic Schr\"odinger cocycles and Schr\"odinger operators, we shall encounter several positive constants depending on the potential $V$, the frequency $\alpha$, etc.
For the convenience of notation, we denote by $C$ constants depending only on $V$, $\alpha$ (or only on $\lambda$, $\alpha$ for the almost Mathieu case).
And we use other notations ($c_1$, $c_2$, $\cdots$, $C_1'$, $C_2'$ $\cdots$, $C_3$, $C_4$, $\cdots$) to denote absolute constants or constants depending on some other quantities (e.g., the given radius of analyticity).

\subsection{Quantitative reducibility -- general analytic potential}

Assume that the family $\{\widehat H_{V,\alpha,\theta}\}_{
\theta}$ is almost localized. Given  $E\in \Sigma_{V,\alpha}$ on the boundary of a spectral gap, it is well-known that the cocycle $(\alpha, S_E^V)$ can be reduced to a constant parabolic cocycle. The main goal of this section is to  show that the off-diagonal coefficient of the parabolic matrix is exponentially small in terms of the label of the spectral gap. We stress that the exponential decay rate of the off-diagonal element is directly related to the exponential decay rate of the spectral gaps (c.f. Theorem \ref{thm_upperbound}). If one just wants to show that the  off-diagonal element is  exponentially small, one may consult \cite{A1} where a more concise proof is given. However, if one wants to explicitly estimate the decay rate as in our paper, it is more technically involved  (consult Remark \ref{compare} for more discussions).

 As in  \cite{A1}, we use truncations to obtain lower bounds on the Bloch waves involved in the definition of  conjugacies. As was first realized by Avila \cite{A1}, a crucial fact to obtain sharp estimates in the non-perturbative regime is the Corona Theorem (with the Uchiyama estimates), whose statement we now recall.

\begin{theorem}[\cite{A1}, see also \cite{Trent}, \cite{Uchi}]\label{corona theorem}
Let $M\in C^{\omega}(\T,  \C^2)$. Assume that for some constants $a,\delta_1,\delta_2>0$, we have $\delta_1\leq |M(z)| \leq \delta_2^{-1}$ for $|\Im z|<a$. Then there exists $Z \in C^\omega( \mathbb{T} , \mathrm{SL}(2,\C))$ with first column $M$ and such that $|Z|_a \lesssim \delta_1^{-2} \delta_2^{-1} (1 - \ln (\delta_1 \delta_2))$.
\end{theorem}

We first use the parametrization by some auxiliary phase $\theta(E)\in \R$, and the estimates are expressed in terms of its last resonance. Then we show how they can be translated in terms of the label of the spectral gap.
A key fact is that the constants in the following statements are independent of the spectral gap we focus on. Indeed, our proof is based on almost localization, which provides constants that are uniform with respect to the energy. Our main statement is as follows.

\begin{theorem}\label{prop_duality_para}
Let $\alpha\in\R\backslash\Q$ satisfy $\beta(\alpha)=0$.
Given $r_0\in (0,1)$, let $V \in C_{r_0}^\omega(\T,\R)$ with $|V|_{r_0}\leq c_0 r_0^{3}$ and
 take $\epsilon_1=\epsilon_1(r_0)\in(0,r_0)$ as in Theorem \ref{almostredth}.
For any $r\in (0,\frac{\epsilon_1}{2\pi})$, there exists $k_1=k_1(\alpha, r_0, r)>0$ such that
for any $E\in \Sigma_{V,\alpha}$ satisfying $2\rho{(\alpha, S_{E}^V)}- k\alpha \in \Z$ with $|k|\geq k_1$, there exist $U\in C_{r}^\omega(\mathbb{T}, \mathrm{PSL}(2,\R))$, $\varphi\in\R$ and $n=n(k) \in \Z$ satisfying $|n|\geq\frac{|k|}{4}$
such that
\begin{equation}\label{eqrota}
U(\cdot +\a)^{-1}  S_{E}^{V}(\cdot) U(\cdot)=
\begin{pmatrix}
1 & \varphi \\[1mm]
0 & 1
\end{pmatrix}
\end{equation}
with $|\varphi| \leq C_3 e^{- \frac{ \pi r}{5} |n|}$ for some $C_3=C_3(\alpha,r_0,r)>0$.
Moreover, for any $r''\in(0, r]$,
$|U|_{r''} \leq C_4 e^{22 \pi r'' |n|}$ for some
$C_4=C_4(\alpha,r_0,r'')>0$.
\end{theorem}

\begin{proof}
If $2\rho{(\alpha, S_{E}^V)}- k\alpha \in \Z$, then by Theorem 3.3 in \cite{AvilaJito}, for some phase $\theta=\theta(E) \in \R$,   there exists a solution $\widehat u$ to $\widehat H_{V,\alpha,\theta} \widehat u= E \widehat u$ with $\widehat u_0=1$ and $|\widehat u_j|\leq 1$ for every $j \in \Z$.
By Theorems 2.5 (also Theorem 4.2) in \cite{AvilaJito}, $\theta=\pm \rho{(\alpha, S_E^V)}+\frac{l \alpha}2$ for some $l \in \Z$. Set $n=n(k):=\pm k +l \in \mathbb{Z}$.
Since $|V|_{r_0}\leq c_0 r_0^{3}$, by Theorem \ref{almostredth}, $\{\widehat H_{V,\alpha,\theta'}\}_{\theta'}$ is almost localized:
let $\epsilon_1:=\epsilon_1(r_0,2)>0$, $C_1:=C_1(\alpha,r_0,2)>0$ be the constants
defined in Theorem \ref{almostredth}. Then it follows from \eqref{almost_localization} that
\begin{equation}\label{localize_duality}
|\widehat u_j| \leq C_1 e^{-\epsilon_1 |j|}, \quad \forall  \   |j|\geq 2 |n|.
\end{equation}
Therefore, the function $u\colon z \mapsto \sum_{j \in \Z} \widehat u_j e^{2 \pi {\rm i} j z}$ is analytic on the strip $\{|\Im z|<\frac{\epsilon_1}{2 \pi}\}$, and
the analytic Bloch wave ${\CU}\colon z \mapsto
\begin{pmatrix}
e^{2 \pi {\rm i}\theta} u(z)\\
u(z-\a)
\end{pmatrix}$ satisfies 
$$
S_{E}^{V}(z) \, {\CU}(z) = e^{2\pi {\rm i}\theta} {\CU}(z+\alpha),\quad \forall  \ z\in \C/\Z \ {\rm with}  \ |\Im z|<\frac{\epsilon_1}{2 \pi}.  
$$
In particular, by the minimality of $x \mapsto x+\alpha$, and the fact that $\widehat u_0=1$,  we see that $\mathcal{U}$ does not vanish.
Define
${\CU}^{(1)}(z):=e^{ \pi {\rm i} n z} {\CU}(z)\in \C^2\backslash\{0\}$.
Since $2\theta -n\alpha \in \Z$, we get
\begin{equation}\label{eq widedtiel u}
S_{E}^{V}(\cdot) \, {\CU}^{(1)}(\cdot) = e^{\pi {\rm i}(2\theta -n\alpha)} \, {\CU}^{(1)}(\cdot+\alpha)= \pm \, {\CU}^{(1)}(\cdot+ \alpha). 
\end{equation}
Without loss of generality, we assume that $S_{E}^{V}(\cdot) \, {\CU}^{(1)} (\cdot)=   {\CU}^{(1)}(\cdot+ \alpha)$. Set $\mathcal{S}:=\Re({\CU}^{(1)})$ and $\mathcal{T}:=\Im( {\CU}^{(1)})$.
Then, by the minimality of $x \mapsto x+\alpha$, the map $x \mapsto \mathrm{det}(\mathcal{S}(x),\mathcal{T}(x))$ is constant on $\T$, equal to $\pm d_0$ for some $d_0 \geq 0$.

If $d_0>0$, let $\sigma=\pm1$ be chosen such that $d_0^{-1/2} (\mathcal{S}, \sigma \mathcal{T})\colon \T\to \mathrm{PSL}(2,\R)$. Note that in this case, $d_0^{-1/2} (\mathcal{S}, \sigma \mathcal{T})$ conjugates $(\alpha,S_{E}^{V})$ to $(\alpha,\mathrm{Id})$.
Otherwise, there exist $\psi\colon \T \to \C$ with $|\psi|=1$ and $\mathcal{V}\colon \T \to \R^2\backslash\{0\}$ such that ${\mathcal{U}^{(1)}}=\psi \, \mathcal{V}$ on $\T$.
By \eqref{eq widedtiel u}, we have
$$S_{E}^{V} (x) \, \mathcal{V}(x)=\frac{\psi(x+\alpha)}{\psi(x)} \mathcal{V}(x+\alpha),\quad \forall \ x \in \T, $$ hence $\frac{\psi(x+\alpha)}{\psi(x)}\in \R$.
By the minimality of $x \mapsto x+\alpha$, we deduce that $\psi|_\T\equiv e^{2 \pi \mathrm{i} \theta_0}\in \C$ for some $\theta_0\in \R$. The map \begin{equation}\label{defv}\mathcal{V}\colon z \mapsto e^{ \pi \mathrm{i}(n z- 2 \theta_0)}\mathcal{U}(z)\end{equation} is analytic on $\{|\Im z|<\frac{\epsilon_1}{2\pi}\}$ and satisfies
\begin{equation}\label{invarireim2}
S_{E}^{V}(z) \, \mathcal{V}(z) = \mathcal{V}(z+ \alpha),\quad   \forall  \ z\in \C/\Z \  {\rm with} \ |\Im z|<\frac{\epsilon_1}{2\pi}.
\end{equation}
\begin{prop}\label{eq1estimee}
For any $r\in (0,\frac{\epsilon_1}{2\pi})$, there is $k_1=k_1(\alpha, r_0, r)>0$ such that for any $E\in \Sigma_{V,\alpha}$ satisfying $2\rho{(\alpha, S_{E}^V)}- k\alpha \in \Z$ with $|k|\geq k_1$, there exist $U\in C_{r}^\omega( \T, \mathrm{PSL}(2,\R))$, $\varphi\in\R$ and $n=n(k) \in \Z$ with $|n|\geq\frac{|k|}{4}$
such that
\begin{equation}\label{eqrota}
U(\cdot+\a)^{-1}  S_{E}^{V}(\cdot) U(\cdot)=
\begin{pmatrix}
1 & \varphi \\[1mm]
0 & 1
\end{pmatrix}.
\end{equation}
Moreover, for any $r''\in(0, r]$,
$|U|_{r''} \leq C_4 e^{22 \pi r'' |n|}$ for some
$C_4=C_4(\alpha,r_0,r'')>0$.
\end{prop}

\begin{proof} 
Fix $r\in (0,\frac{\epsilon_1}{2\pi})$, choose some $\delta\in(0,\frac{\epsilon_1}{2}-\pi r)$, and set $h:=r+\frac{\delta}{2\pi}$.
Recall that $|\widehat u_j|\leq 1$, for all $j \in \Z$. In view of \eqref{localize_duality} and \eqref{defv}, for any $r''
\in(0, h]$, there exists $C'_1=C'_1(\alpha,r_0,r'')>0$ such that for $|n|$ sufficiently large, we have
\begin{align}
|\mathcal{V}|_{r''} &\leq 4|n|e^{5 \pi r''|n|}+4 C_1 e^{\pi r'' |n|} \sum_{j \geq 2|n|}  e^{-(\epsilon_1- 2\pi r'') j}\leq C'_1 e^{5 \pi (r''+\delta) |n|}.\label{upper bound on v 2}
\end{align}

Let us now show the lower bounds  on $\mathcal{V}$. Set $I:=[-2|n|+1,2|n|-1]$ and consider the trigonometric polynomial $u^I\colon z\mapsto \sum_{j \in I} \widehat u_j e^{2 \pi {\rm i} j z}$. We define $\mathcal{U}^I, \mathcal{V}^I$ accordingly, for $u^I$ in place of $u$. 
By \eqref{localize_duality}, for any $r''\in (0,h]$, we have
\begin{equation}\label{diff U and truncat}
|\mathcal{U}- \mathcal{U}^I|_{r''} \leq \frac{4 C_1}{1-e^{-(\epsilon_1- 2 \pi r'')}}e^{-2(\epsilon_1- 2 \pi r'')|n|}.
\end{equation}

Given any analytic function $f$ defined on the strip $\{|\Im z|<\frac{\epsilon_1}{2\pi}\}$, any $|y| \leq h$, let $f_y\colon  x \mapsto f(x+\mathrm{i}y)$, $x\in\T$.
Recall that $\int_\T u_y (x) dx =\widehat u_0=1$. By \eqref{diff U and truncat}, we thus get
\begin{equation}\label{lowww bounddd cui}
\int_{\T} \left| \CU_y^I(x) \right|dx\geq  \left| \int_{\T} \CU_y(x) dx \right|-\frac{4 C_1}{1-e^{-(\epsilon_1- 2 \pi h)}}e^{-2(\epsilon_1- 2 \pi h)|n|}\geq \frac{3}{4}
\end{equation}
for $|n|\geq n_1(\alpha,\epsilon_1,h)$ large enough.
Let us denote by $(\widehat{v}_j)_{j \in \mathbb{Z}}$ the Fourier coefficients of  $V$, and let $\chi_{I}$ be the characteristic function of $I$. We get
\begin{equation}\label{equation inv first bloch wave}
S_{E}^{V}(\cdot)\, \mathcal{U}^I(\cdot)= e^{2 \pi \mathrm{i} \theta} \mathcal{U}^I(\cdot+\alpha)+e^{2 \pi\mathrm{i}\theta} \begin{pmatrix}
g(\cdot)\\
0
\end{pmatrix}
\end{equation}
for some function $g\in C^\omega(\T,\C)$ whose Fourier coefficients $(\widehat{g}_j)_{j \in \mathbb{Z}}$ satisfy
$$
\widehat g_j=\chi_{I}(j) \left(E-2 \cos2 \pi (\theta+j\a)\right) \widehat u_j - \sum_{l\in \Z} \chi_{I}(j-l) \widehat u_{j-l} \widehat v_l.
$$
Since $\widehat H \widehat u=E \widehat u$, we also have
\begin{equation}\label{equation coeffc gg}
\widehat g_j=-\chi_{\Z \backslash I}(j) \left(E-2 \cos2 \pi (\theta+j\a)\right) \widehat u_j + \sum_{l\in \Z} \chi_{\Z \backslash I}(j-l) \widehat u_{j-l} \widehat v_l.
\end{equation}
If $E \in \Sigma_{V,\alpha}$, then $|E|\leq 2+ |V|_\T \leq 2+ c_0 r_0^3$. By \eqref{localize_duality} and \eqref{equation coeffc gg}, we therefore obtain
\begin{align*}
|g|_{r''} \leq& \ \sum_{j\in \Z} \chi_{\Z \backslash I}(j) (4+ c_0 r_0^3) \, C_1 e^{-(\epsilon_1-2 \pi r'')|j|}\\
&\ + \sum_{j,l\in \Z} \chi_{\Z \backslash I}(j-l) c_0 r_0^3 \,  C_1 e^{-(\epsilon_1-2 \pi r'')|j-l|} e^{- 2\pi (r_0-r'')|l|}\\
\leq&\ C_1(4+  c_0 r_0^3)\left(1+ \sum_{l\in \Z} e^{- 2\pi (r_0-r'')|l|}\right)\sum_{|j|\geq 2|n|} e^{-(\epsilon_1- 2 \pi r'')|j|}\\
\leq&\ C'_2(\alpha,r_0,r'') e^{-2(\epsilon_1- 2 \pi r'')|n|}.
\end{align*}
Together with \eqref{equation inv first bloch wave}, this implies that for all $z\in \C/\Z$ with $|\Im z| \leq r''\leq h$, all $m \geq 1$:
\begin{equation}\label{est2cw}
|\mathcal{U}^I(z+m\alpha)| \leq |\mathcal{A}_m(z)|\, |\mathcal{U}^I(z)|+ \sum_{j=1}^{m}|\mathcal{A}_{m-j}(z+j\alpha)|\cdot C'_2 e^{-2(\epsilon_1- 2 \pi r'')|n|}. 
\end{equation}

To get a lower bound on $\mathcal{V}$, we will use the following result of Avila-Jitomirskaya.
\begin{lemma}[Theorem 6.1 in \cite{AvilaJito}]\label{theo lagrang interp}
Let $\ell\geq 1$ and $1 \leq p \leq \lfloor q_{\ell+1}/q_\ell\rfloor$.
If $P$ has essential degree\footnote{Recall that a trigonometric polynomial $P_0\colon \T \to \C$ has  \textit{essential degree} at most $d\geq 1$ if its Fourier coefficients outside an interval of length $d$ are vanishing.}
at most $p q_\ell -1$ and $x_0 \in \T$, then for some absolute constant $K_0>0$,
\begin{equation*}
|P|_{\T} \leq K_0 q_{\ell+1}^{K_0 p} \sup\limits_{0 \leq m \leq p q_\ell -1}|P(x_0+m\alpha)|.
\end{equation*}
\end{lemma}

Choose $\ell\geq 1$ and $1 \leq p \leq \lfloor q_{\ell+1}/q_\ell\rfloor$ such that $(p-1) q_\ell -1 \leq 4 | n|< p q_\ell -1 \leq q_{\ell+1}$.
In particular, under the assumption $\beta(\alpha)=0$, it implies that
$$
|P|_{\T} \leq K_0 e^{o(|n|)}  \sup\limits_{0 \leq m \leq 4 | n|+q_l}|P(x+m\alpha)|\leq  K_0 e^{o(|n|)}  \sup\limits_{0 \leq m \leq 8 | n|}|P(x+m\alpha)|.$$
Since $\mathcal{U}^I$ has essential degree
at most $4| n|$, by Lemma \ref{theo lagrang interp}, for any  $x \in \T$, we have
\begin{equation}\label{ineqq}
|\mathcal{U}^I_y|_{\T} \leq K_0 e^{o(|n|)}  \sup\limits_{0 \leq m \leq 8 | n|}|\mathcal{U}^I_y(x+m\alpha)|,\quad \forall  \ |y|\leq r''  .
\end{equation}

Let us show that for any $\delta>0$, there exists $C'_3=C'_3(\alpha,r_0,r'',\delta)>0$ such that
\begin{equation}\label{lowboundu}
\inf_{|\Im z|\leq r''} |\mathcal{U}^I(z)| \geq 2 C'_3  e^{-2\pi \delta |n|}.
\end{equation}
Else for some $\delta'\in (0,\frac{\epsilon_1}{2\pi}-r'')$, we would have
$|\mathcal{U}_y^I(x)| \leq e^{- 4\pi \delta' |n|}$ for $|n|$ arbitrarily large, and $y=y(n)\in[-r'',r'']$.
We deduce from Corollary \ref{subexp groowth}, \eqref{est2cw} and \eqref{ineqq} that
 $|\mathcal{U}_{y(n)}^I|_\T\leq e^{-2\pi\delta'|n|}\leq \frac{1}{2}$ for $|n|$ large enough, which contradicts \eqref{lowww bounddd cui}.
Combining with \eqref{defv} and \eqref{diff U and truncat}, for $\delta>0$ arbitrarily small, we get
 \begin{equation}\label{subexpo}
\inf_{|\Im z|\leq r''} |{\mathcal{V}}(z)| \geq C'_3  e^{-\pi(\delta+r'') |n|},\quad \forall \ r''\in(0,h].
\end{equation}

Applying Theorem \ref{corona theorem} to $\mathcal{V}$, combining with \eqref{upper bound on v 2} and \eqref{subexpo}, we deduce that there exists 
$U_1=(\mathcal{V}, \mathcal{W}) \in C_{h}^\omega(\mathbb{T}, \mathrm{PSL}(2,\C))$ such that for all $r''\in(0,h]$,
\begin{equation}\label{rstim U1bisss}
 |U_1|_{r''} \leq C'_4(\alpha,r_0,r'',\delta) e^{7\pi(r''+\delta) |n|}.
\end{equation}
  Indeed, one can choose $U_1\in C_{h}^\omega(\mathbb{T}, \mathrm{PSL}(2,\R))$, since $\mathcal{V}|_{\mathbb{T}}$ takes values in $\R^2\backslash\{0\}$, then one only need to replace $U_1$ by $U_1=(\mathcal{V},\widetilde{\mathcal{W}})$, with $\widetilde{\mathcal{W}}\colon z \mapsto \frac{1}{2}(\mathcal{W}(z)+\overline{\mathcal{W}(\overline{z}}))$.

By \eqref{invarireim2}, there exists $\varphi^{(1)}\ \in C_{h}^\omega( \mathbb{T}, \R)$ such that \begin{equation}\label{eqatori}
U_1(\cdot+\a)^{-1} S_{E}^{V}(\cdot) U_1(\cdot)=
\begin{pmatrix}
1 & \varphi^{(1)}(\cdot)\\[1mm]
0 & 1
\end{pmatrix}.
\end{equation}
By \eqref{rstim U1bisss} and \eqref{eqatori}, we have:
$$|\varphi^{(1)}|_{h} \leq (4+c_0 r_0^3) (C_4')^2 e^{14 \pi (h+\delta)|n|}.$$
Since $\beta(\a)=0$, we can solve the cohomological equation
\begin{equation}\label{cohomolllfognri}
\phi(z+\a) - \phi(z) =\varphi^{(1)}(z)-\int_\T \varphi^{(1)}(x) \, dx,
\end{equation}
with $\int_\T \phi(x)dx=0$. Moreover, $\phi\ \in C_{r}^\omega( \mathbb{T}, \R)$, and for any $r''\in(0,r]$, one has
\begin{equation}\label{esti_phi}
 |\phi|_{r''} \leq C_5'(\alpha,r_0,r'',\delta) e^{14\pi(r''+\delta) |n|}.
\end{equation}
 Let  $U:=U_1
\begin{pmatrix}
1 & \phi\\
0 & 1
\end{pmatrix}$ and $\varphi:= \int_\T \varphi^{(1)}(x) \, dx$.
\eqref{eqatori} and \eqref{cohomolllfognri} implies that
\begin{equation}\label{equation conjugaison finale}
U(\cdot+\a)^{-1}S_{E}^{V}(\cdot)  U(\cdot)=
\begin{pmatrix}
1 & \varphi\\
0 & 1
\end{pmatrix} .
\end{equation}
Obviously, $U\in C_{r}^\omega(\mathbb{T}, \mathrm{PSL}(2,\R))$. By \eqref{rstim U1bisss} and \eqref{esti_phi}, we get the estimate
$$|U|_{r''} \leq C_4 (\alpha,h_0,r'') e^{22\pi r''|n|},\quad \forall \ r''\in(0, r].$$

To find a relation between $k$ and $n=n(k)$, we estimate the topological degree of the conjugacy map $U$. Recall that $2\rho{(\alpha,S_E^V)} -  k\alpha \in\Z  $, and then, by  \eqref{equation conjugaison finale}, we have $|k|=|\mathrm{deg} U|$.
Since $x \mapsto
\begin{pmatrix}
1 & \phi(x)\\
0 & 1
\end{pmatrix}$ is homotopic to the identity, we also know that $|k|=|\mathrm{deg} U_1|$. Thus, it remains to estimate $\mathrm{deg}U_{1}$. For this purpose, we look at the degree of the first column ${\mathcal{V}}\colon \T \rightarrow \R^2\backslash\{0\}$ of $U_1$. By \eqref{subexpo}, for any $\delta>0$,
\begin{equation}\label{estttttttti}
\inf\limits_{x \in \T} |{\mathcal{V}}(x)| \geq e^{-\pi \delta |n|}
\end{equation}
for all sufficiently large $|n|$.
Consider the truncated vector $\mathcal{V}^I$ as defined above.
Since $\mathcal{V}(z)= e^{ \pi \mathrm{i}(n z- 2 \theta_0)}\mathcal{U}(z)$, we deduce from \eqref{diff U and truncat} that
$$|\mathcal{V}-\mathcal{V}^I|_\T \leq \frac{4 C_1}{1-e^{-\epsilon_1}} e^{-2\epsilon_1|n|}.$$
Comparing with \eqref{estttttttti}, for any sufficiently large $|n|$,  we obtain
$$
|\mathcal{V}(x)-\mathcal{V}^I(x)| \leq |{\mathcal{V}}(x)|, \quad \forall \  x\in \T.
$$
By Rouch\'{e}'s theorem, we deduce that $\mathrm{deg}{\mathcal{V}}=\mathrm{deg}\mathcal{V}^I$. Consider a coordinate of $\mathcal{V}^I$ which is not identically vanishing. It is a trigonometric polynomial of degree less than $4|n|$, so it has at most $4|n|$ zeros in $\T$, and we get
$|\mathrm{deg} {\mathcal{V}}| \leq 4 |n|$.
Therefore, for $|k|$ sufficiently large, we conclude that $|k| = |\mathrm{deg} U_1| \leq 4 |n|$.
\end{proof}

Let us now estimate the size of  $\varphi$. We will first need the following.

\begin{prop}\label{lemma first conjugaison}
 Let $\{n_{l}\}_{l}$ be the set of resonances of $\theta$.  For any $\delta\in(0,\frac{\epsilon_1}{20 \pi})$, there exist constants $C'_i=C'_i(\alpha,r_0,\delta)>0$, $i=6,7,8$, such that the following holds.
 There exists $B\in C^{\omega}(\T, \mathrm{PSL}(2,\C))$ with
$|B|_{\frac{\epsilon_1}{20 \pi}} \leq C'_6 e^{\delta |n_l|}$ such that
\begin{equation}\label{conj par complexe bis}
B(\cdot+\a)^{-1} S_E^V(\cdot) B(\cdot)=\begin{pmatrix}
e^{2 \pi \mathrm{i} \theta} & 0\\[1mm]
0 & e^{-2 \pi \mathrm{i} \theta}
\end{pmatrix}+\begin{pmatrix}
\beta_1(\cdot) & \beta(\cdot)\\[1mm]
\beta_3(\cdot) & \beta_4(\cdot)
\end{pmatrix},
\end{equation}
with  $|\beta_j|_{\frac{\epsilon_1}{20 \pi}} \leq C'_7 e^{-(\frac{\epsilon_1}{10}-2\pi\delta) |n_l|} $ for $j=1,3,4$ and $|\beta|_{\T} \leq C'_8 e^{-(\frac{\epsilon_1}{10}-2\pi \delta) |n_l|}$.
\end{prop}

\begin{proof}
Without loss of generality, we assume in the following that $n_l\geq 0$. Let $u^J\colon z \mapsto \sum_{j \in J} \widehat u_j e^{2 \pi \mathrm{i} j z}$ and ${\CU}^J\colon z \mapsto
\begin{pmatrix}
e^{2 \pi {\rm i}\theta} u^J(z)\\
u^J(z-\a)
\end{pmatrix}$, where $J:=[-\lfloor\frac{n_l}{4}\rfloor, \lfloor\frac{n_l}{4} \rfloor]$. Consequently, we have
\begin{equation}\label{equation inv first bloch wave bis bis}
S_E^V(\cdot)\, \mathcal{U}^J(\cdot)= e^{2 \pi \mathrm{i} \theta} \mathcal{U}^J(\cdot+\alpha)+e^{2 \pi\mathrm{i}\theta} \begin{pmatrix}
g^*(\cdot)\\
0
\end{pmatrix}
\end{equation}
for some analytic function $g^*\in C^\omega(\T,\C)$ whose Fourier coefficients $(\widehat g^*_j)_{j\in \Z}$ satisfy
\begin{equation}\label{fourier coeff gj bis 1}
\widehat g^*_j=\chi_{J}(j) \left(E-2 \cos2 \pi (\theta+j\a)\right) \widehat u_j - \sum_{m\in \Z} \chi_{J}(j-m) \widehat u_{j-m} \widehat v_m. 
\end{equation}
Since $\widehat H \widehat u=E \widehat u$, we also have
\begin{equation}\label{fourier coeff gj bis 2}
\widehat g^*_j=-\chi_{\Z \backslash J}(j) \left(E-2 \cos2 \pi (\theta+j\a)\right) \widehat u_j + \sum_{m\in \Z} \chi_{\Z \backslash J}(j-m) \widehat u_{j-m} \widehat v_m. 
\end{equation}
Let us assume that $j\notin J$, i.e., $|j|> \lfloor\frac{n_l}{4}\rfloor$. By \eqref{localize_duality}, and since $2 n_{l-1}=o(\frac{n_l}{8})$ we have  $|\widehat u_{j-m}| \leq C_1 e^{-\epsilon_1 |j-m|}$ for $\frac{n_l}{8}<|j-m|< \frac{n_l}{2}$, while $|\widehat u_{j-m}|\leq 1$ in other cases. Besides, $|\widehat v_m| \leq c_0 r_0^3 e^{-2 \pi r_0 |m|}$ for all $m \in \Z$.  Thus, we deduce from \eqref{fourier coeff gj bis 1} that
\begin{align*}
|\widehat g^*_j| &\leq \sum_{ |j-m| \leq \frac{n_l}{8}} |\widehat u_{j-m}| |\widehat v_m| + \sum_{\frac{n_l}{8} < |j-m| \leq \frac{n_l}{4}} |\widehat u_{j-m}| |\widehat v_m| \\
&\leq c_0 r_0^3\left(\sum_{ |m| \geq |j|-\frac{n_l}{8}}  e^{-2 \pi r_0 |m|} + \sum_{\frac{n_l}{8}< |j-m| \leq \frac{n_l}{4}}  C_1 e^{-\epsilon_1 |j-m|} e^{-2 \pi r_0 |m|}\right) \\ &\leq c_0 r_0^3\left(\sum_{|m| \geq \frac{|j|}{2}} e^{-2 \pi r_0 |m|} + \frac{n_l}{4} C_1  e^{-\epsilon_1 |j|}\right)\\
&\leq C'_9 e^{- \frac{\epsilon_1}2 |j|}
\end{align*}
for some constant $C'_9=C'_9(\alpha,r_0)>0$. Similarly, if $j \in J$, by \eqref{fourier coeff gj bis 2} we get
\begin{align*}
|\widehat g^*_j| &\leq \sum_{\frac{n_l}{4} \leq  |j-m|< \frac{n_l}{2}} |\widehat u_{j-m}| |\widehat v_m| + \sum_{|j-m| \geq \frac{n_l}{2}} |\widehat u_{j-m}| |\widehat v_m| \\ &\leq c_0 r_0^3\left(\frac{n_l}{2} C_1  e^{-\frac{\epsilon_1}{4}n_l}+ \sum_{|m|\geq \frac{n_l}{4}} e^{-2 \pi r_0 |m|}\right)\\
&\leq C'_{10} e^{-\frac{\epsilon_1}{4}n_l}
\end{align*}
for some constant $C'_{10}=C'_{10}(\alpha,r_0)>0$.
We thus obtain
\begin{equation}\label{eq applica a bloch}
\sup_{|\Im z|\leq \frac{\epsilon_1}{20\pi}}|S_E^V(z)\, \mathcal{U}^J(z)- e^{2 \pi \mathrm{i} \theta} \mathcal{U}^J(z+\alpha)|=|g^*|_{\frac{\epsilon_1}{20\pi}} \leq C'_{11}(\alpha,r_0) e^{- \frac{\epsilon_1}{10} n_l}.
\end{equation}
Arguing as in \eqref{lowboundu}, we deduce that for any $\delta>0$, one has \begin{equation}\label{uestimeeef bisbis}
\inf_{|\Im z|\leq  \frac{\epsilon_1}{20\pi}} |\mathcal{U}^J(z)| \geq C'_{12}(\alpha,r_0,\delta)  e^{-\delta n_l}.
\end{equation}
On the other hand, since $|\widehat u_j|\leq C_1 e^{-\epsilon_1 |j|}$ for $2 n_{l-1} \leq |j| \leq \frac{1}{2} n_l$\footnote{In the case of an almost Mathieu operator $H_{\lambda,\alpha,\theta}$ with $\lambda<1$, then by Theorem \ref{thm_almost_almost-1}, for all $r \in(0,-\frac{\ln \lambda}{2 \pi})$ and for $\eta>0$ arbitrarily small, we have $|\widehat u_j|\leq C_1 e^{-2 \pi r |j|}$ for $2 n_{l-1}+\eta n_l \leq |j| \leq \frac{1}{2} n_l$,  which is also sufficient for our purpose.}, $|\widehat u_j| \leq 1$ in other cases, and since $n_{l-1}=o(n_l)$, we  also have
\begin{equation}\label{uestimeeef bisbister}
\sup_{|\Im z|\leq  \frac{\epsilon_1}{20\pi}} |\mathcal{U}^J(z)| \leq C'_{13}(\alpha,r_0,\delta)  e^{\delta n_l}.
\end{equation}
Combining \eqref{eq applica a bloch}$-$\eqref{uestimeeef bisbister}, and by Theorem \ref{corona theorem}, we can define $U_2 \in C_{ \frac{\epsilon_1}{20\pi}}^{\omega}(\T,\mathrm{PSL}(2,\C))$ with $\mathcal{U}^J$ as first column which satisfies
\begin{equation}\label{conj par complexe bis}
U_2(\cdot+\a)^{-1} S_E^{V}(\cdot) U_2(\cdot)=\begin{pmatrix}
e^{2 \pi \mathrm{i} \theta} & 0\\[1mm]
0 & e^{-2 \pi \mathrm{i} \theta}
\end{pmatrix}+\begin{pmatrix}
\widetilde{\beta}_1(\cdot) & \varphi^{(2)}(\cdot)\\[1mm]
\widetilde{\beta}_3(\cdot) & \widetilde{\beta}_4(\cdot)
\end{pmatrix}.
\end{equation}
Besides, we have $|U_2|_{\frac{\epsilon_1}{20\pi}} \leq C'_{14}(\alpha,r_0,\delta)  e^{\delta  n_l}$, $|\varphi^{(2)}|_{ \frac{\epsilon_1}{20\pi}} \leq C'_{15}(\alpha,r_0,\delta) e^{2\delta n_l}$, and for $j=1,3,4$, $|\widetilde{\beta}_j|_{ \frac{\epsilon_1}{20\pi}} \leq C'_{16}(\alpha,r_0) e^{- \frac{\epsilon_1}{10} n_l}$.

Let us write $\varphi^{(2)}(z)=\sum_j\widehat \varphi_j e^{2 \pi \mathrm{i} j z}$, and let  $\tau$ satisfy $$\varphi^{(2)} (z)-e^{-2 \pi \mathrm{i} \theta}\tau(z+\alpha)+e^{2 \pi \mathrm{i} \theta} \tau(z)= \sum_{|j| \geq n_l} \widehat \varphi_j e^{2 \pi \mathrm{i} j z}.$$ We have $\tau(z)=\sum_{|j| < n_l} \widehat \tau_j e^{2 \pi \mathrm{i} j z}$, where
$
\widehat\tau_j:=\frac{-\widehat \varphi_j e^{-2 \pi \mathrm{i} \theta}}{1-e^{-2 \pi \mathrm{i} (2\theta-j\alpha)}}
$.
By the assumption   $\beta(\alpha)=0$, and the definition of resonances, for $j \neq n_l$, we have
$$\|2\theta- j \alpha\|_{\T}  \geq \|(j-n_l)\alpha\|_{\T}- \|2\theta- n_l\alpha\|_{\T} \geq  e^{-o(|j-n_l|)}-e^{-\epsilon_0n_l }\geq \frac{1}{2}e^{-o(|j-n_l|)}.$$
Therefore,  we deduce that $|\tau|_{h} \leq C'_{17}(\alpha,r_0,\delta) e^{2\delta n_l}$.

Let $B:=U_2
\begin{pmatrix}
1 & \tau\\
0 & 1
\end{pmatrix}$ conjugate the initial cocycle to the following:
$$
B(\cdot+\a)^{-1} S_E^V(\cdot) B(\cdot)=\begin{pmatrix}
e^{2 \pi \mathrm{i} \theta} & 0  \\[1mm]
0 & e^{-2 \pi \mathrm{i} \theta}
\end{pmatrix}+
\begin{pmatrix}
\beta_1(\cdot) & \beta_2(\cdot)+\varsigma(\cdot)\\[1mm]
\beta_3(\cdot) & \beta_4(\cdot)
\end{pmatrix}
$$
with $\varsigma:z\mapsto \sum_{|j| \geq n_l} \widehat \varphi_{j} e^{2 \pi \mathrm{i} j z}$,
 $\beta_1(\cdot):=\widetilde{\beta}_1(\cdot)-\widetilde{\beta}_3(\cdot)\tau(\cdot+\alpha)$,  $\beta_3(\cdot):=\widetilde{\beta}_3(\cdot)$,  $\beta_4(\cdot):=\widetilde{\beta}_4(\cdot)+\widetilde{\beta}_3(\cdot)\tau(\cdot)$ and
$$
\beta_2(\cdot):=\widetilde{\beta}_1(\cdot) \tau(\cdot)- \widetilde{\beta}_4(\cdot)\tau(\cdot+\alpha)+\widetilde{\beta}_3(\cdot)\tau(\cdot)\tau(\cdot+\alpha).
$$
By the estimates on $\widetilde{\beta}_1,\widetilde{\beta}_3,\widetilde{\beta}_4$, we have $|\beta_j|_{\frac{\epsilon_1}{20\pi}}\leq C'_{18}(\alpha,r_0,\delta) e^{-(\frac{\epsilon_1}{10} -4 \delta) n_l}$, for all $j=1,2,3,4$.
On the other hand,
$$\big|\varsigma\big|_{\T}\leq \sum_{|j| \geq n_l} |\varphi^{(2)}|_{h}e^{-\frac{\epsilon_1}{10} |j|} \leq C'_{19}(\alpha,r_0,\delta) e^{-(\frac{\epsilon_1}{10}-2 \delta) n_l},$$
which concludes.
\end{proof}

Let us apply the previous result for $n_l=n$ when $2 \theta-n\alpha \in \Z$.
According to Proposition \ref{lemma first conjugaison}, for any $\delta>0$, there exists $C'_{20}=C'_{20}(\alpha,r_0,\delta)>0$ such that
\begin{equation}\label{eqgorwth}
\sup_{0 \leq l \leq e^{\frac{\epsilon_1|n|}{10}}}|{{\mathcal A}}_l|_{\T} \leq C'_{20} e^{\delta |n|}.
\end{equation}
On the other hand, for any $l \in \mathbb{N}$, we obtain by iterating \eqref{equation conjugaison finale}:
\begin{equation*}
\begin{pmatrix}
1 & l \varphi\\[1mm]
0 & 1
\end{pmatrix}=U(\cdot+l \alpha)^{-1} {{\mathcal A}}_l(\cdot) \, U(\cdot).
\end{equation*}
Take $l:=\left\lfloor e^{\frac{\epsilon_1|n|}{10}}\right\rfloor$. By Proposition \ref{eq1estimee},  $|U|_{\T},|U^{-1}|_{\T} \leq C'_{21} e^{\delta |n|}$ for some $C'_{21}=C'_{21}(\alpha,r_0,\delta)>0$, and by \eqref{eqgorwth},  $|{{\mathcal A}}_{l}|_{\T} \leq C'_{20} e^{\delta |n|}$. We conclude that $e^{\frac{\epsilon_1|n|}{10}} |\varphi| \leq C'_{20}(C'_{21})^2 e^{3 \delta |n|}$, and the desired estimate on $\varphi$ follows.
\end{proof}

\begin{remark}\label{compare}
The readers should compare our Proposition \ref{lemma first conjugaison} with Theorem 3.8 in \cite{A1}. In Theorem 3.8 in \cite{A1},  we only know that $\beta(z)$ has exponential decay, while the decay rate is very small. However, in our result, $\beta(z)$ has a large and explicit decay, which allows us to show that $\varphi$ has large exponential decay.
\end{remark}


%
%
%
%
%

\subsection{Quantitative almost reducibility -- almost Mathieu case}

For almost Mathieu operators, we have the following improved result.

\begin{theorem}\label{thm_almost_almost-2}
Let $\alpha \in \R$ satisfy $\beta(\alpha)=0$. If $0<\lambda<1$, then for any $r\in (0,-\frac{1}{2\pi}\ln \lambda)$, $E \in \Sigma_{\lambda,\alpha}$, the following holds  on $\{|\Im z|<r\}$:
\begin{enumerate}
\item

either $(\alpha,S_E^\lambda)$ is almost reducible to $(\alpha,R_{\theta})$
 for some
$\theta=\theta(E)\in \R$:  for any $\varepsilon>0$, there exists $U\in C^\omega_{r}(\T,\mathrm{PSL}(2,\R))$ such that
\begin{equation*}
|U(\cdot+\alpha)^{-1} S_E^\lambda(\cdot) U(\cdot)-R_{ \theta}|_{r} < \varepsilon;
\end{equation*}

\item or $(\alpha,S_E^\lambda)$ is reducible:
\begin{enumerate}
\item if $2 \rho(\alpha,S_E^\lambda)-j\alpha \notin \Z$ for any $j \in \Z$, then $(\alpha,S_E^\lambda)$ is reducible to $(\alpha,R_\theta)$ for some
$\theta=\theta(E)\in \R$;
\item if $2 \rho(\alpha,S_E^\lambda)-k\alpha \in \Z$ for some $k \in \Z$, then there is $k_2=k_2(\lambda,\alpha, r)>0$ such that if $|k|\geq k_2$, then there exist $\varphi \in \R\backslash\{0\}$ and $U \in C_{r}^\omega(\T,{\rm PSL}(2,\R))$ such that
\begin{equation*}\label{reduci_parabis}
U(\cdot+\a)^{-1} S_{E}^{\lambda}(\cdot) U(\cdot) =
\begin{pmatrix}
1 & \varphi \\
0 & 1
\end{pmatrix},
\end{equation*}
 and there is $n=n(k) \in \Z$ with $|n|\geq\frac{|k|}{4}$, such that $|\varphi| \leq C_5 e^{- \frac{2 \pi r}{3} |n|}$ for some $C_5=C_5(\lambda,\alpha,r)>0$ and for any $0<r'' \leq r$, $|U|_{r''} \leq C_6 e^{22 \pi r'' |n|}$ for some $C_6=C_6(\lambda,\alpha,r'')>0$.
\end{enumerate}
\end{enumerate}
\end{theorem}

\begin{proof}
Fix $0< r<-\frac{\ln \lambda}{2\pi}$, and set $\tilde r:=\frac12(-\frac{\ln\lambda}{2\pi}+r)$. Hence $-\frac{\ln\lambda}{2\pi}-\tilde r=\tilde r-r$.
The proof that follows is similar to those of Proposition \ref{eq1estimee} and Proposition \ref{lemma first conjugaison}, but with the improved estimates obtained in Theorem \ref{thm_almost_almost-1} for almost Mathieu operators. For any $E \in \Sigma_{\lambda,\alpha}$,  there exist $\theta=\theta(E)\in \R$ and $(\widehat u_j)_{j \in \Z}$ satisfying
 $\widehat H_{\lambda,\alpha,\theta} \widehat u= E \widehat u$ with
 $\widehat u_0=1$ and $|\widehat u_j|\leq 1$ for every $j \in \Z$.
 Let us denote by $\{n_{l}\}_{l}$ the set of $\epsilon_0-$resonances of $\theta$. We fix $\eta\in(0, \frac{\pi(\tilde r-r)}{-4\ln\lambda})$ sufficiently small and let $\delta\in(-\eta\ln\lambda, \frac{\pi(\tilde r-r)}{4})$. By Theorem \ref{thm_almost_almost-1}, there exists $N_0\geq 0$ such that for $|n_{l}| \geq N_0$, we have
\begin{equation}\label{decay expllc}
|\widehat u_j| \leq  e^{-2 \pi r |j|},\quad \forall  \ 2 |n_{l-1}|+\eta |n_{l}| < |j| < \frac12 |n_{l}|.
\end{equation}

Let $n_l\geq N_0$ and set $J:=[-\lfloor\frac{n_l}{2}\rfloor+2, \lfloor\frac{n_l}{2}\rfloor-2]$. We define $u^J\colon z \mapsto \sum_{j \in J} \widehat u_j e^{2 \pi \mathrm{i} j z}$ and ${\CU}^J\colon z \mapsto
\begin{pmatrix}
e^{2 \pi {\rm i}\theta} u^J(z)\\
u^J(z-\a)
\end{pmatrix}$.  Then
\begin{equation*}
S_E^{\lambda}(\cdot)\, \mathcal{U}^J(\cdot)= e^{2 \pi \mathrm{i} \theta} \mathcal{U}^J(\cdot+\alpha)+e^{2 \pi\mathrm{i}\theta} \begin{pmatrix}
g^*(\cdot)\\
0
\end{pmatrix}
\end{equation*}
for some analytic function $g^*$ whose Fourier coefficients $(\widehat g^*_j)_{j\in \Z}$ satisfy
\begin{equation}\label{fourier coeff gj ter 11}
\widehat g^*_j=\chi_{J}(j) \left(E-2 \cos2 \pi (\theta+j\a)\right) \widehat u_j - \lambda \sum_{m=\pm1}  \chi_{J}(j-m) \widehat u_{j-m}.  
\end{equation}
Since $\widehat H \widehat u=E \widehat u$, we also have
\begin{equation}\label{fourier coeff gj ter 22}
\widehat g^*_j=-\chi_{\Z \backslash J}(j) \left(E-2 \cos2 \pi (\theta+j\a)\right) \widehat u_j + \lambda \sum_{m=\pm1} \chi_{\Z \backslash J}(j-m) \widehat u_{j-m} . 
\end{equation}
Applying either \eqref{fourier coeff gj ter 11} or \eqref{fourier coeff gj ter 22}, we see that $\widehat{g}_j^*=0$ if $|j|\not\in(\lfloor\frac{n_l}{2}\rfloor-4,\lfloor\frac{n_l}{2}\rfloor)$.  Then by \eqref{decay expllc}, there exists a constant $c'_{1}=c'_{1}(\lambda,\alpha,\tilde r)>0$ such that
\begin{equation}\label{def sup g' on r'}
\sup_{|\Im z|\leq \tilde r}|S_E^\lambda(z)\, \mathcal{U}^J(z)- e^{2 \pi \mathrm{i} \theta} \mathcal{U}^J(z+\alpha)|=|g^*|_{\tilde r} \leq c'_{1} e^{-\pi (-\frac{\ln\lambda}{2\pi}-\tilde r) n_l}=c'_{1} e^{-\pi (\tilde r-r) n_l}.
\end{equation}
Similar to (\ref{uestimeeef bisbis}) and (\ref{uestimeeef bisbister}), we deduce that  \begin{equation*}
c'_{2}(\alpha,\lambda,\delta,\tilde r)  e^{-\delta n_l} \leq \inf_{|\Im z|\leq \tilde r} |\mathcal{U}^J(z)|  \leq \sup_{|\Im z|\leq \tilde r} |\mathcal{U}^J(z)| \leq c'_{3}(\alpha,\lambda,\delta,\tilde r)  e^{\delta n_l}.
 \end{equation*}
As previously in \eqref{conj par complexe bis}, we can define $U_2 \in C^{\omega}_{\tilde r}(\T,\mathrm{PSL}(2,\C))$ with first column $\mathcal{U}^J$ such that
\begin{equation*}
U_2(\cdot+\a)^{-1} S_E^\lambda(\cdot) U_2(\cdot)=\begin{pmatrix}
e^{2 \pi \mathrm{i} \theta} & 0\\[1mm]
0 & e^{-2 \pi \mathrm{i} \theta}
\end{pmatrix}+\begin{pmatrix}
\widetilde{\beta}_1(\cdot) & \varphi^{(2)}(\cdot)\\[1mm]
\widetilde{\beta}_3(\cdot) & \widetilde{\beta}_4(\cdot)
\end{pmatrix}
\end{equation*}
for some function $\varphi^{(2)}\colon z \mapsto\sum_j\widehat \varphi_j e^{2 \pi \mathrm{i} j z}$. Besides, $|U_2|_{\tilde r} \leq c'_4(\alpha,\lambda,\delta,\tilde r)  e^{\delta  n_l}$, $|\varphi^{(2)}|_{\tilde r} \leq c'_5(\alpha,\lambda,\delta,\tilde r) e^{2\delta n_l}$, and for $j=1,3,4$, $|\widetilde{\beta}_j|_{\tilde r} \leq c'_6(\alpha,\lambda,\tilde r) e^{- \pi (\tilde r-r) n_l}$.

Consequently, one can define $B \in C^{\omega}_{\tilde r}(\T,\mathrm{PSL}(2,\C))$ such that
\begin{equation}\label{con-2}
B(\cdot+\a)^{-1} S_E^{\lambda}(\cdot) B(\cdot)=\begin{pmatrix}
e^{2 \pi \mathrm{i} \theta} & 0  \\[1mm]
0 & e^{-2 \pi \mathrm{i} \theta}
\end{pmatrix}+
\begin{pmatrix}
\beta_1(\cdot) & \beta_2(\cdot)+\varsigma(\cdot)\\[1mm]
\beta_3(\cdot) & \beta_4(\cdot)
\end{pmatrix}
\end{equation}
with $|B|_{\tilde r} \leq c'_7(\alpha,\lambda,\delta,\tilde r) e^{2\delta n_l},$ and $|\beta_j|_{\tilde r}\leq c'_8(\alpha,\lambda,\delta,\tilde r) e^{-( \pi(\tilde r-r) -4 \delta) n_l}$, for all $j=1,2,3,4$. Moreover,
 $\varsigma(z):=\sum_{|j| \geq n_l} \widehat \varphi_{j} e^{2 \pi \mathrm{i} j z}$. Therefore, we have 
 $$\left|\varsigma\right|_{r}\leq \sum_{|j| \geq n_l} |\varphi^{(2)}|_{\tilde r}\,  e^{-2 \pi (\tilde r-r) |j|} \leq c'_{5} (\alpha,\lambda,\delta,\tilde r)e^{-( \pi(\tilde r-r)- 4\delta)n_l}.$$

If $\theta$ is $\epsilon_0-$resonant, which means that the collection $\{n_l\}$ is infinite, then we set $U:=\frac{1}{1+ \mathrm{i}}B \begin{pmatrix}
\mathrm{i} & -1\\
\mathrm{i} & 1
\end{pmatrix}\in C^\omega_{r}(\T, \mathrm{PSL}(2,\R))$, so that
$$
|U(\cdot+\alpha)^{-1} S_E^{\lambda}(\cdot)  U(\cdot)- R_\theta |_{r} \leq c'_9(\alpha,\lambda,\delta,\tilde r) e^{-( \pi(\tilde r-r) -4 \delta) n_l}.
$$
 This concludes the proof of the almost reducibility statement in (1).

Assume now that $\theta$ is not $\epsilon_0-$resonant, then  $\theta=\pm \rho{(\alpha, S_E^{\lambda})}+\frac{k' \alpha}2$ for some $k' \in \Z$ (see Remark 4.2 in \cite{AvilaJito}).  If $2 \rho{(\alpha, S_E^\lambda)} - j\alpha \notin \Z$ for all $j\in \Z$, then $2 \theta - j\alpha \notin \Z$ for all $j \in \Z$. Theorem \ref{thm_almost_almost-2} (a) actually follows from Theorem 2.5 of \cite{AvilaJito}.
Now if  $2 \rho{(\alpha, S_E^\lambda)} - k\alpha \in \Z$ for some $k \in \Z$, we thus have $2 \theta-n\alpha \in \Z$ for some $n=n(k) \in \Z$. As in Proposition \ref{eq1estimee}, there exist $U \in C^\omega_{\tilde r} ( \T,\mathrm{PSL}(2,\C))$, $\varphi \in \R$ such that
\begin{equation*}
U(\cdot+\a)^{-1}  S_E^\lambda(\cdot) \, U(\cdot)=
\begin{pmatrix}
1 & \varphi \\[1mm]
0 & 1
\end{pmatrix}.
\end{equation*}
Note that the case $\varphi=0$ cannot happen. Otherwise $\widehat H_{\lambda,\alpha,\theta}=\lambda H_{ \lambda^{-1},\alpha,\theta}$ would have an eigenvalue with two linearly independent eigenvectors in $\ell^2(\Z)$, which is impossible by the limit-point character of Schr\"odinger operators.  The estimate on $U$ and the relation between $k$ and $n$ are obtained as in Proposition \ref{eq1estimee}. To estimate $\varphi$, we argue as in the proof of Theorem \ref{prop_duality_para}, but using the following improved version of Proposition \ref{lemma first conjugaison} in the case of almost Mathieu operators. \end{proof}

\begin{lemma}
Assume that $2\theta-n\alpha\in\Z$ for some $n \in \Z$.
For any $\delta \in(0, -\frac{\ln \lambda}{6 \pi})$, there exist $c'_i=c'_i(\alpha,\lambda,\delta)>0$, $i=10,11,12$, and there exists $B\in C^{\omega}(\T, \mathrm{PSL}(2,\C))$ with
$|B|_{-\frac{\ln \lambda}{6 \pi}} \leq c'_{10} e^{\delta |n|}$ such that
\begin{equation}\label{conj par complexe bis-2}
B(\cdot+\a)^{-1} S_E^\lambda(\cdot) B(\cdot)=\begin{pmatrix}
e^{2 \pi \mathrm{i} \theta} & 0\\[1mm]
0 & e^{-2 \pi \mathrm{i} \theta}
\end{pmatrix}+\begin{pmatrix}
\beta_1(\cdot) & \beta(\cdot)\\[1mm]
\beta_3(\cdot) & \beta_4(\cdot)
\end{pmatrix},
\end{equation}
with  $|\beta_j|_{-\frac{\ln \lambda}{6 \pi}}\leq c'_{11} e^{- (-\frac{\ln \lambda}{3}-2 \pi\delta) |n|} $ and $|\beta|_{\T} \leq c'_{12} e^{-(-\frac{\ln \lambda}{3}-2 \pi\delta) |n|}$.
\end{lemma}

\begin{proof}
Set $h:=-\frac{\ln \lambda}{6 \pi}$ and let $\delta >0$ be taken arbitrarily small. With the same notations as above, for $|n|$ which is large enough, \eqref{def sup g' on r'} becomes now
$$
|g^*|_{h} \leq c'_1(\alpha,\lambda,\delta) e^{-2 \pi (h-\delta) |n|}.
$$
Consequently, \eqref{con-2} holds with $|B|_{h} \leq c'_7 (\alpha,\lambda,\delta)  e^{\delta  |n|}$,
$|\beta_j|_{h}\leq c'_8(\alpha,\lambda,\delta) e^{-2 \pi( h -\delta) |n|}$, for  $j=1,2,3,4$, while
$$|\varsigma|_{\T}\leq \sum_{|j| \geq |n|} |\varphi^{(2)}|_{h}e^{-2 \pi h |j|} \leq c'_{5}(\alpha,\lambda,\delta) e^{-2 \pi( h- \delta) |n|},$$ which concludes.
\end{proof}

\section{Global to local reduction}\label{subsectlocalglobal}

In this section, we extend the reducibility results for Schr\"odinger cocycles, which were obtained previously for small potentials, to the global subcritical regime.

\subsection{General subcritical potential}

 We first consider typical $V\in C^\omega(\T,\R)$ such that  $\Sigma_{V,\alpha}$  presents the structure given in Theorem \ref{global-red}.
We denote by $\{I_i\}_{1\leq i \leq m}$ the intervals such that the energies in $\Sigma_{V,\alpha}\cap I_i$ are subcritical, and by $\{J_i\}_{1\leq i \leq m'}$ (the  number $m'$ of intervals can be $m-1$, $m$ or $m+1$) the intervals such that the energies in $\Sigma_{V,\alpha}\cap J_i$ are supercritical. Let $\Sigma^{\mathrm{sup}}_{V,\alpha}:=\bigcup_{i}(\Sigma_{V,\alpha} \cap J_i)$ and
 $\Sigma^{\rm sub}_{V,\alpha}:=\bigcup_{i}(\Sigma_{V,\alpha} \cap I_i) $.

Let us focus on the case where $E\in\Sigma^{\rm sub}_{V,\alpha}$.
By Theorem \ref{arc-conjecture},  the Schr\"odinger cocycle $(\alpha, S^V_E)$ is almost reducible.
While almost reducibility allows one to conjugate the dynamics of a cocycle close to constant, it is convenient to have the conjugated cocycle in Schr\"odinger form, since many results (particularly those depending on Aubry duality) are obtained only in this setting.

\begin{prop}\label{prop_ar_1}
Let $\alpha\in\R\backslash\Q$ satisfy $\beta(\alpha)=0$. There exists $h_1=h_1(V,\alpha)>0$ such that for any $\eta>0$, $E\in \Sigma_{V,\alpha}^{\rm sub}$, one can find $\Phi_E\in C^\omega(\T,{\rm PSL}(2,\R))$ with $|\Phi_E|_{h_1}<\Lambda$ for some
$\Lambda=\Lambda(V,\alpha, \eta,h_1)>0$, $E_*=E_*(E)$ locally constant (as a function of $E$), and $V_{*}=V_*(E) \in C^\omega_{h_1}(\T,\R)$, $|V_*|_{h_1}<\eta$, such that \begin{equation}\label{conjPhi}
\Phi_E(\cdot+\alpha)^{-1} S_E^V(\cdot) \Phi_E(\cdot)=S_{E_*}^{V_*}(\cdot).
\end{equation}
\end{prop}

\begin{remark}
The crucial fact in this proposition is that we can choose $h_1$ to be independent of $E$ and $\eta$, and choose $\Gamma$ to be independent of $E$.
\end{remark}

\begin{proof}
The key ingredient  for us  is the following:
 \begin{lemma}[Avila-Jitomirskaya \cite{AvilaJitoHolder}]\label{lemma conjugaaison socrhd}
Let $\alpha\in\R\backslash\Q$ and $A \in C_{h_*}^\omega(\T,{\rm SL}(2,\R))$ for some $h_*>0$, such that $(\alpha,A)$ is almost reducible. There exists $h_0\in(0,h_*)$ such that for any $\eta>0$,
one can find $V \in C_{h_0}^\omega(\T,\R)$ with $|V|_{h_0} < \eta$, $E \in \R$,
and $Z \in C_{h_0}^\omega(\T,{\rm PSL}(2,\R))$ such that $$Z(\cdot+\alpha)^{-1}A(\cdot)Z(\cdot)=S_E^V(\cdot).$$
Moreover, for every $0< h\leq h_0$, there is $\delta>0$ such that if $A' \in C_{h}^\omega(\T,{\rm SL}(2,\R))$ satisfies $|A-A'|_{h} < \delta$, then there exist $ V' \in C_{h}^\omega(\T,\R)$ with $|V'|_{h} < \eta$ and $Z' \in   C_{h}^\omega(\T,{\rm PSL}(2,\R))$ such that $|Z-Z'|_{h} < \eta$ and $$ Z'(\cdot+\alpha)^{-1}A'(\cdot) Z'(\cdot)=S_E^{V'}(\cdot).$$
\end{lemma}

For any $E_0\in\Sigma^{\rm sub}_{V,\alpha}$, the cocycle $(\alpha, S_{E_0}^V)$ is subcritical, hence almost reducible by Theorem  \ref{arc-conjecture}. Fix $\eta>0$.
By Lemma \ref{lemma conjugaaison socrhd}, there is $h_0= h_0(E_0,V,\alpha)>0$, such that
one can find $V_*(E_0) \in C_{h_0}^\omega(\T,\R)$ with $|V_*(E_0)|_{h_0}\leq \eta$,
$E_*= E_*(E_0)\in\R$ and $\Psi_{E_0}\in C_{h_0}^\omega(\T,\mathrm{PSL}(2,\R))$ with
$|\Psi_{E_0}|_{h_0}\leq\tilde\Lambda$ for $\tilde\Lambda=\tilde\Lambda(V,\alpha, \eta, h_0,E_0)>0$
such that
$$\Psi_{E_0}(\cdot+\alpha)^{-1} S_{E_0}^V(\cdot) \Psi_{E_0}(\cdot)= S_{ E_*}^{V_*(E_0)}(\cdot).$$
Moreover, for every $0<h \leq h_0$, there exists $\delta > 0$ such that for  any $E\in (E_0-\delta, E_0 +\delta)$, one can find $\Psi_{E}\in C_{h}^\omega(\T,{\rm PSL}(2,\R))$ with $|\Psi_{E}-\Psi_{E_0}|_{h}\leq \eta$,   and $V'_*(E)\in C_{h}^\omega(\T,\R)$ with $|V'_*(E)|_{h}\leq \eta$ satisfying
$$
\Psi_{E}(\cdot + \alpha)^{-1} S_{E}^V(\cdot) \Psi_{E} (\cdot) = S_{ E_*}^{ V'_*(E)}(\cdot).
$$

By Theorem \ref{global-red},  $\Sigma^{\rm sub}_{V,\alpha}$ is compact, by compactness argument, we obtain $h_1=h_1(V,\alpha)$ and $\Lambda=\Lambda(V,\alpha,\eta,h_1)$, both independent of $E$, such that for any $E \in \Sigma^{\rm sub}_{V,\alpha}$, there exist $\Phi_E\in C_{h_1}^\omega(\T,{\rm PSL}(2,\R))$ with $|\Phi_E|_{h_1} < \Lambda$, and $V_*=V_*(E) \in C^\omega_{h_1}(\T,\R)$ with $|V_*(E)|_{h_1}<\eta$, such that  $(\ref{conjPhi})$ holds.
\end{proof}

%
%

\begin{corollary}\label{prop_duality_global}
Let  $\alpha\in\R\backslash\Q$ satisfy $\beta(\alpha)=0$. There exist  $h_1=h_1(V,\alpha)>0$, $0<c=c(V,\alpha)<h_1$, $\tilde{k}=\tilde{k}(V,\alpha)>0$, such that
for any $E \in \Sigma^{\rm sub}_{V,\alpha}$ satisfying $2\rho(\alpha, S_E^{V})- k\alpha\in \Z$ with $|k| \geq \tilde{k}$,  there exist $Y \in C^\omega_{\frac{c}{2\pi}}(\T,{\rm PSL}(2,\R))$ and $\varphi\in \R$  s.t.
$$
Y(\cdot+\a)^{-1} S_{E}^{V}(\cdot) Y(\cdot) =
\begin{pmatrix}
1 & \varphi\\
0 & 1
\end{pmatrix}.
$$
 Moreover, there is $n=n(k)\in \Z$ satisfying $ |n|\geq \frac{|k|}{5}$ such that
 $|\varphi| \leq C e^{-\frac{c}{10}|n|}$ and
 $$|Y|_{r''} \leq C_7(V,\alpha,r'') e^{22\pi r''|n|},\quad \forall \  r''\in(0,\frac{c}{2\pi}).$$
\end{corollary}

\begin{proof}
Let $(\eta_n)_n$  be a sequence of positive numbers going to zero. For any  $E \in \Sigma^{\rm sub}_{V,\alpha}$,
we can apply Proposition \ref{prop_ar_1} and get $h_1=h_1(V,\alpha)>0$,  $\Lambda_n=\Lambda_n(V,\alpha,\eta_n,h_1)>0$, $\Phi_E^{n}\in C^\omega(\T,{\rm PSL}(2,\R))$, $E^n=E^n(E)\in\R$ and $V^n=V^n(E) \in C^\omega_{h_1}(\T,\R)$ such that
$$
\Phi_E^{n}(\cdot+\alpha)^{-1} S_E^V(\cdot) \Phi_E^{n}(\cdot)=S_{E^{n}}^{V^n}(\cdot),
$$
with $|\Phi_E^{n}|_{h_1}<\Lambda_n$ and  $|V^n|_{h_1}<\eta_n$.
Note that since $h_1(V,\alpha)$ is fixed, one can always find $N_*=N_*(h_1)$ large enough such that
$\eta_{N_*} \leq  c_0 h_1^3$,
where $c_0>0$ is the absolute constant given by Theorem \ref{almostredth}.  It follows that
$$
|\Phi_E^{N_*}|_{h_1} \leq \Lambda = \Lambda(V,\alpha, c_0h^3_1, h_1).
$$
By footnote 5 of \cite{A3},  $|{\rm deg}\Phi_E^{N_*}|\leq  C|\ln \Lambda|$ with $C=C(V,\alpha)$ independent of $E$.

Assume that $2\rho{(\alpha, S_E^{V})}-k\alpha\in\Z$. We abbreviate $E_*=E^{N_*}(E)$,  $V_*=V^{N_*}(E)$,  and $k_*={\rm deg}\Phi_E^{N_*}$. Then
 $2\rho{(\alpha, S_{ E_*}^{ V_{*}})}-( k -k_*)\alpha\in\Z$.
Clearly, $ E_* \in \Sigma_{ V_{*}, \alpha}$ since uniform hyperbolicity is invariant under conjugacy, and then $(\alpha,\, S^{V_{*}}_{ E_*})$ is not uniformly hyperbolic.
Applying Theorem \ref{prop_duality_para},  we thus get $c=c(h_1)>0$,  $k_1=k_1(\alpha, h_1, \frac{c}{2\pi})$, $\varphi\in \R$ and $U \in C^\omega_{\frac{c}{2\pi}}(2\T, {\rm SL}(2,\R))$ such that
$$U(\cdot+\alpha)^{-1}S^{V_{*}}_{E_*}(\cdot)U(\cdot)=\begin{pmatrix}
1 & \varphi \\
0 & 1
\end{pmatrix}.$$
 Moreover, if $|k-k_*|\geq k_1$,  then we have
 $$|\varphi| \leq C_3(\alpha,h_1, \frac{c}{2\pi}) e^{-\frac{c}{10}|n|}< C(V,\alpha) e^{-\frac{c}{10}|n|}$$
and for any $r''\in(0, \frac{c}{2\pi})$,
$$|U|_{r''} \leq C_4 (\alpha, h_1, r'') e^{22\pi r'' |n|}$$
for some $n$ satisfying  $|{\rm deg}U|=|k-k_*| \leq 4|n|$.
Set $Y := \Phi^{N_*}_{E_*} U$. Thus, if $$|k|\geq k_1+C |\ln\Lambda|:=\tilde k(V,\alpha),$$ and hence $|k-k_*| \geq  k_1$, then
$$|{\rm deg}Y|=|k| \leq |k_*|+|k-k_*|\leq C |\ln\Lambda| + 4 |n|\leq 5 |n|.$$
Furthermore, $$|Y|_{r''}\leq \Lambda(V,\alpha, c_0h^3_1, h_1)
 C_4(\alpha,h_1,r'')  e^{22\pi r'' |n|} \leq C_7(V,\alpha,r'') e^{22\pi r'' |n|}.$$
 Finally, let us emphasize that the constants $c=c(h_1)$, $C_3=C_3(\alpha, h_1,\epsilon)$, which come from almost localization estimates, only depend on the sizes of the strip, but not on the potential $V_*$, thus not on our choice of $E$. It is the main reason why the estimates  we get are uniform with respect to $k$.
\end{proof}

\subsection{Almost Mathieu operator}

Now we focus on the subcritical almost Mathieu operator. Compared with Corollary \ref{prop_duality_global}, we obtain even stronger results for $\alpha \in {\rm DC}$:

\begin{prop}\label{redu amo case}
Let $\alpha\in {\rm DC}$.  Given $0<\lambda<1$, we consider
the operator $H_{\lambda, \alpha,\theta}$.
For any $E\in \Sigma_{\lambda,\alpha}$ satisfying $2\rho{(\alpha, S_E^\lambda)} - k \alpha\in \Z$ with  $k\in \Z \backslash\{0\}$, for any $0< r<\frac{1}{2\pi}|\ln \lambda|$,
there exist $U \in C_{r}^\omega(\T,{\rm PSL}(2,\R))$ and $\varphi \in \R\backslash\{0\}$ such that
\begin{equation}\label{reduci_para}
U(\cdot+\a)^{-1} S_{E}^{\lambda}(\cdot) U(\cdot) =
\begin{pmatrix}
1 & \varphi \\
0 & 1
\end{pmatrix}.
\end{equation}
Moreover,  there exist $C_8, C_9>0$, depending on $\lambda,\alpha,r$, such that $|\varphi| \leq C_{8} e^{- 2\pi r |k|}$ and for any $r''\in (0,r]$, $|U|_{r''} \leq C_{9}  e^{\frac{3 \pi r''}2  |k|}$.
\end{prop}

\begin{proof} Assume that $\alpha\in {\rm DC}(\gamma,\tau)$.
Fix $r\in (0,\frac{1}{2\pi}|\ln \lambda|)$.
Let $\tilde r:=\frac12(-\frac{\ln \lambda}{2\pi}+r)$.
By Theorem \ref{thm_almost_almost-2}, for any sequence of positive numbers $(\eta_n)_n$ going to zero, there are $\Phi_E^{n}\in C_{\tilde r}^\omega(\T,{\rm PSL}(2,\R))$, $F_n\in C_{\tilde r}^\omega(\T,{\rm gl}(2,\R))$  and $\phi_n=\phi_n(E)\in\T$ such that
$$\Phi_E^{n}(\cdot+\alpha)^{-1} S_E^{\lambda }(\cdot) \Phi_E^{n}(\cdot)=R_{\phi_n}+F_n(\cdot),$$
with $|F_n|_{\tilde r}< \eta_n/2$ and $|\Phi_E^{n}|_{\tilde r}<\Gamma_n$ for some $\Gamma_n=\Gamma_n(\lambda,\alpha, \eta_n, \tilde r,E)>0$.
As a consequence, for any $E'\in \R$, one has
$$
\left|\Phi_E^{n}(\cdot+\alpha)^{-1} S_{E'}^{\lambda}(\cdot) \Phi_E^{n}(\cdot) -  R_{\phi_n}\right|_{\tilde r} < \frac{\eta_n}{2}+ |E-E'| \, | \Phi_E^{n}|_{\tilde r}^2.
$$
It follows that with the same $\Phi_E^{n}$, we have $|\Phi_E^{n}(x+\alpha)^{-1} S_{E'}^{\lambda}(x) \Phi_E^{n}(x) - R_{\phi_n}|_{ \tilde r} < \eta_n$ for any energy $E'$ in a neighborhood $\mathcal{U}(E)$ of $E$.

One can always take $N_*$ large enough such that $\eta_{N_*} \leq \varepsilon_*(\gamma, \tau, \tilde r, r, 1)$,  where  $\varepsilon_*(\gamma, \tau, \tilde r, r, 1)$ is define in Theorem \ref{thm_gap_edge_SL} (see also Remark \ref{uniformcons}).  It follows that \begin{equation}\label{norm}|\Phi_E^{N_*}|_r \leq \Gamma:=\Gamma_{N_*}(\lambda,\alpha, \eta_{N_*},\tilde r,E).\end{equation}
By the compactness of $\Sigma_{\lambda,\alpha}$, $\Gamma>0$ can be chosen  independently of $E$.

Let $k_*:={\rm deg}\Phi_E^{N_*}$, the assumption $2\rho{(\alpha, S_E^{\lambda})}-k\alpha\in\Z$ implies $2\rho{(\alpha,R_{\phi_{N_*}}+F_{N_*})}- (k-k_*)\alpha\in \Z$.
Since we have chosen $\eta_{N_*} \leq \varepsilon_*(\gamma, \tau, \tilde r, r, 1)$, then by Theorem \ref{thm_gap_edge_SL}, we get $\phi\in\R$ and $W\in C_{r}^\omega(\T, {\rm PSL}(2,\R))$ such that
$$W(\cdot+\alpha)^{-1} (R_{\phi_{N_*}}+F_{N_*}(\cdot)) W(\cdot)=\begin{pmatrix}
1 & \phi \\
0 & 1
\end{pmatrix}.$$
Letting $U:=\Phi_E^{N_*} W\in C^\omega_{r}(\T, {\rm PSL}(2,\R))$, we have (\ref{reduci_para}).
Moreover,
 \begin{eqnarray*} && |\varphi|\leq \varepsilon_*^{\frac34}e^{2\pi r|k_*|}e^{-2\pi r|k|} \leq C_{8}(\lambda,\alpha,r) e^{-2\pi r|k|},\\
&&  |U|_{r''} \leq  \Gamma \cdot D_1 e^{\frac{3\pi r}{2}|k_*|}e^{\frac{3\pi r''}{2}|k|}\leq C_{9}(\lambda,\alpha,r) e^{\frac{3\pi r''}{2}|k|},\quad \forall \ r''\in(0,r].\end{eqnarray*}
The above inequalities follows since by $(\ref{norm})$ and footnote 5 of \cite{A3}, we have $|{\rm deg}\Phi_E^{N_*}|\leq  C |\ln \Gamma|$.
\end{proof}

\section{Gap estimates via Moser-P\"oschel argument}\label{Sec_bounds}

\noindent

We consider the quasi-periodic Schr\"odinger operator on $\ell^2(\Z)$:
$$
(H_{V,\alpha,\theta} u)_n= u_{n+1}+u_{n-1} + V( \theta+n\alpha) u_n,
$$
with $\alpha\in\T^d$ such that $(1,\alpha)$ is rationally independent, and $V\in C^\omega(\T^d, \R)$ non-constant. Based on Moser-P\"oschel argument \cite{Moser-Poschel},
we will estimate the size of the spectral gap $G_k(V)=(E_k^-, E_k^+)$ via quantitative reducibility of corresponding Schr\"odinger cocycle at its edge points.

Now we assume that the cocycle $(\alpha,S_{E_k^+}^{V})$ is reducible, i.e., there exist $X \in C_R^{\omega}(\T^d, {\rm PSL}(2,\R))$ for some $0<R<1$ and a constant matrix $B$, such that
$$
X(\cdot+\alpha)^{-1}S_{E_k^+}^{V}(\cdot)X(\cdot)=B.
$$
Since $ E_k^+\in \Sigma_{V,\alpha}$ is a right edge point of a gap,  $(\alpha,S_{E_k^+}^{V})$ is reduced to a  constant parabolic cocycle $B=
\begin{pmatrix}
1 & \zeta \\
0 & 1
\end{pmatrix}$ with $0\leq \zeta< \frac12$.
Recall that $\zeta=0$  if and only if the corresponding gap is  collapsed. We will show that the size of gap is determined by $X$ and $\zeta$.

For any $0<\delta<1$, a direct calculation yields
$$X(\cdot+\alpha)^{-1}S_{E_k^+ -\delta}^{V}(\cdot)X(\cdot)= B-\delta P(\cdot)$$ with
$$P(\cdot):=
\begin{pmatrix}
X_{11}(\cdot) X_{12}(\cdot) - \zeta X_{11}^2(\cdot) & -\zeta X_{11}(\cdot) X_{12}(\cdot) + X_{12}^2(\cdot) \\[1mm]
- X_{11}^2(\cdot)  & - X_{11}(\cdot) X_{12}(\cdot)
\end{pmatrix}.$$
Obviously,
\begin{equation}\label{esti_PX_on_T}
|P|_{r''}\leq (1+\zeta)|X|^2_{r''}< 2|X|^2_{r''},\quad  \forall \  r''\in (0, R].
\end{equation}
 In fact, moving the energy $E$ from the right end of the gap $E_k^+$ to $E_k^+-\delta$, we can determine the other edge point of the spectral gap according to the variation of the rotation number $\rho{(\alpha, X(\cdot+\alpha)^{-1}S^{V}_{E_k^+-\delta}(\cdot)X(\cdot))}$.
Note that the rotation number of the constant cocycle $(\alpha, B)$ vanishes since $B$ is parabolic. Then, as shown symbolically in Figure \ref{f.graph}, we have the following:
\begin{itemize}
  \item If the rotation number of $(\alpha, X(\cdot+\alpha)^{-1}S^{V}_{E_k^+-\delta_1}(\cdot)X(\cdot))$ is positive, then $E_k^+-\delta_1$ is beyond the left edge of $G_k(V)$, thus $|G_k(V)|\leq \delta_1$.
  \item If the rotation number of $(\alpha, X(\cdot+\alpha)^{-1}S^{V}_{E_k^+-\delta_2}(\cdot)X(\cdot))$ vanishes, then $E_k^+-\delta_2$ is still in $\overline{G_k(V)}$ and hence $|G_k(V)|\geq \delta_2$.
\end{itemize}
Of course, one can estimate the size of spectral gap $G_k(V)$ similarly by starting from the left edge point $E_k^-$.

\begin{figure}
\begin{center}
\begin{tikzpicture}[yscale=1.5]
\draw [->] (-0.5,0) -- (6.5,0);
\draw [->] (0,-0.8) -- (0,1.5);
\draw [thick,domain=1.22:1.3] plot (\x, {0.49+0.7*sqrt(1.3-\x)});
\draw [thick,domain=1.3:1.47] plot (\x, {0.49});
\draw [thick,domain=1.47:1.48] plot (\x, {0.42+0.7*sqrt(1.48-\x)});
\draw [thick,domain=1.48:1.52] plot (\x, {0.42});
\draw [thick,domain=1.52:1.56] plot (\x, {0.28+0.7*sqrt(1.56-\x)});
\draw [thick,domain=1.56:1.86] plot (\x, {0.28});
\draw [thick,domain=1.86:1.9] plot (\x, {0.14+0.7*sqrt(1.9-\x)});
\draw [thick,domain=1.9:1.96] plot (\x, {0.14});
\draw [thick,domain=1.96:2] plot (\x, {0.7*sqrt(2-\x)});
\draw [thick,domain=2:3.5] plot (\x, {0});
\draw [thick,domain=3.5:3.54] plot (\x, {-0.7*sqrt(\x-3.5)});
\draw [thick,domain=3.54:3.6] plot (\x, {-0.14});
\draw [thick,domain=3.6:3.64] plot (\x, {-0.14-0.7*sqrt(\x-3.6)});
\draw [thick,domain=3.64:3.94] plot (\x, {-0.28});
\draw [thick,domain=3.94:3.98] plot (\x, {-0.28-0.7*sqrt(\x-3.94)});
\draw [thick,domain=3.98:4.02] plot (\x, {-0.42});
\draw [thick,domain=4.02:4.03] plot (\x, {-0.42-0.7*sqrt(\x-4.02)});
\draw [thick,domain=4.03:4.2] plot (\x, {-0.49});
\draw [thick,domain=4.2:4.28] plot (\x, {-0.49-0.7*sqrt(\x-4.2)});
\draw [fill=red] (1.9,0) circle [radius=0.025];
\draw [fill=red] (3.5,0) circle [radius=0.025];
\draw [fill=red] (2.1,0) circle [radius=0.025];
\node [above] at (8,1.3) {$E':=E_k^+-\delta_1$};
\node [above] at (8,0.9) {$E'':=E_k^+-\delta_2$};
\node [above] at (0,1.5) {$\rho$};
\node [below] at (-0.15,0) {$0$};
\node [right] at (6.5,0) {$E$};
\node [above] at (3.5,0) {$E_k^+$};
\node [below] at (1.85,-0.01) {$E'$};
\node [above] at (2.3,0) {$E''$};
\end{tikzpicture}
\caption{Rotation number of cocycle $(\alpha, X(\cdot+\alpha)^{-1}S^{V}_{E_k^+ -\delta}(\cdot)X(\cdot))$}
\label{f.graph}
\end{center}
\end{figure}
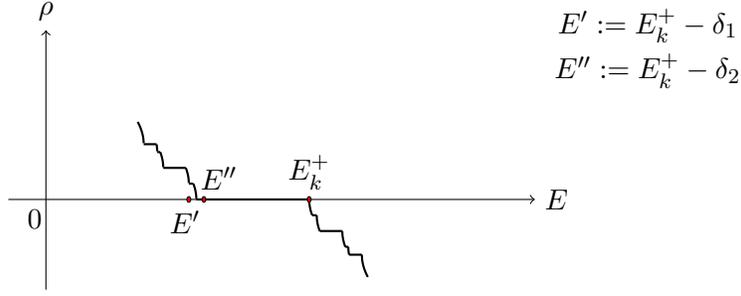



Although we focus on the case of a Diophantine frequency, our approach also works for a Liouvillean frequency. For any  rationally independent $\alpha\in\T^d$, we set
$$
\beta=\beta(\alpha):=\limsup_{k\rightarrow \infty} \frac{1}{|k|} \ln \frac{1}{\|\la k,\alpha\ra\|_{\T}},
$$
which is a generalization of $(\ref{equibeta})$ to the multi-frequency case. For convenience, we let
\begin{equation}\label{D1}
D_{\alpha,R}:= 2+40\sum_{n\in\Z^d} \frac{e^{-(R+3\beta)|n|/2}}{|e^{{\rm i}\la n,\alpha\ra}-1|^3},\end{equation}
which is finite if $R>3\beta$. For $\tau>d-1$, let
\begin{equation}\label{D2}
D_{\tau}:= 2^{4\tau+9} \, \Gamma(4\tau+2).\end{equation}

In the following, we first apply one standard KAM step to the cocycle $(\alpha, B-\delta P(\cdot))$, which is the starting point of our estimate on the size of the  gap.

\begin{lemma}\label{prop_ave}
Given $\alpha \in \T^d$ with $R>3\beta(\alpha)\geq 0$. We have the following:
\begin{enumerate}
\item If $0<\delta<  D_{\alpha,R}^{-1}|X|^{-2}_R$, then there exist $\tilde X\in C_{\frac{R-3\beta}2}^{\omega}(\T^d, {\rm SL}(2, \R))$ and $P_1\in C_{\frac{R-3\beta}2}^\omega(\T^d, {\rm gl}(2, \R))$ such that
\begin{equation}\label{first_ave}
\tilde X(\cdot+\alpha)^{-1}(B-\delta P(\cdot))\tilde X(\cdot)
=e^{b_0-\delta b_1}+ \delta^2 P_1(\cdot),
\end{equation}
where $b_0:=\begin{pmatrix}
0 & \zeta \\[1mm]
0 & 0
\end{pmatrix}$ and
$$b_1:=
\begin{pmatrix}
[X_{11} X_{12}] - \frac{\zeta}{2} [X_{11}^2] &  -\zeta [X_{11} X_{12}] + [X_{12}^2] \\[1mm]
-[X_{11}^2]  &   - [X_{11} X_{12}] + \frac{\zeta}{2} [X_{11}^2]
\end{pmatrix},$$
with the estimates
\begin{equation}\label{first_ave_esti}
|\tilde X- {\rm Id}|_{\frac{R-3\beta}2}\leq 2 D_{\alpha,R} \, \delta| X|^2_{R},
\quad |P_1|_{\frac{R-3\beta}2}  \leq  2 D_{\alpha,R}^2 \left| X\right|^4_{R}.
\end{equation}
\item In particular, for $\alpha\in{\rm DC}_d(\gamma,\tau)$, if $0<\delta< D^{-1}_{\tau} \gamma^{3} R^{4\tau+1}|X|^{-2}_R$, then  (\ref{first_ave}) holds with
\begin{equation}\label{first_ave_esti_DC}
|\tilde X- {\rm Id}|_{\frac{R}2}\leq 2 D_{\tau}\gamma^{-3} R^{-(4\tau+1)}\delta|X|^2_{R},
\quad |P_1|_{\frac{R}2}  \leq 2 D^2_{\tau}\gamma^{-6}  R^{-2(4\tau+1)} |X|^4_{R}.
\end{equation}
\end{enumerate}
\end{lemma}
\begin{proof}Let $G:=-\delta B^{-1} P$. Noting that $B^{-1}=\begin{pmatrix}
1 & -\zeta \\
0 & 1
\end{pmatrix}$, we can see that ${\rm tr}(B^{-1} P)=0$, hence $G\in {\CB}_R$.
By a standard KAM step, we can construct $Y\in {\CB}_{\frac{R-3\beta}{2}}$ such that
\begin{equation}\label{eq_homo}
Y(\cdot+\alpha)B-BY(\cdot)=B(G(\cdot)-[G]).
\end{equation}
Indeed, by identifying the Fourier coefficients of the two sides of (\ref{eq_homo}), we have
\begin{equation}\label{homo_Y}
\left \{\begin{array}{l}
 \widehat Y_{21}(n)=\frac{\widehat G_{21}(n)}{e^{{\rm i}\la n,\alpha \ra}-1}\\[1mm]
 \widehat Y_{11}(n)=\frac{\widehat G_{11}(n)+\zeta\widehat Y_{21}(n)}{e^{{\rm i}\la n,\alpha \ra}-1}\\[1mm]
 \widehat Y_{12}(n)=\frac{\widehat G_{12}(n)-\zeta(1+e^{{\rm i}\la n,\alpha \ra} )\widehat Y_{11}(n)}{e^{{\rm i}\la n,\alpha \ra}-1}
\end{array}\right.  , \quad \forall \  n\in\Z^d\backslash\{0\}.
\end{equation}
Hence, by the decay property of the Fourier coefficient $\widehat G(n)$, we have
$$|Y|_{\frac{R-3\beta}{2}}= \sum_{n\in\Z^d}  |\widehat Y(n)| e^{\frac{R-3\beta}{2}|n|} \leq \frac12(D_{\alpha,R}-2) \, \delta |P|_{R}.$$

In the same manner as in Proposition 2 of \cite{HA}, for $\tilde X:=e^Y$, we have
 $$\tilde X(\cdot+\alpha)^{-1}(B-\delta P(\cdot))\tilde X(\cdot)=Be^{[G]}+ \tilde P(\cdot),$$
 where
 \begin{align*}
\tilde P(\cdot)&:=BY(\cdot)-Y(\cdot+\alpha)B-B [G]- \delta P(\cdot) +\sum_{m+n\geq 2}\frac{1}{m!}(-Y(\cdot+\alpha))^m B \frac{1}{n!}Y(\cdot)^n\\
&+ \delta\sum_{m+n\geq 1} \frac{1}{m!}(-Y(\cdot+\alpha))^m P(\cdot) \frac{1}{n!}Y(\cdot)^n +B \sum_{n\geq 2}\frac{1}{n!} [G]^n.
\end{align*}
Obviously, $|\tilde X-{\rm Id}|_{\frac{R-3\beta}{2}}\leq 2|Y|_{\frac{R-3\beta}{2}}\leq D_{\alpha,R} \delta |P|_{R}.$
Since $\sum_{m+n=k}\frac{k!}{m!n!}=2^k$ and $|G|_R\leq \delta|P|_R$, we get
\begin{eqnarray*}
\left|\sum_{m+n\geq 2}\frac{1}{m!}(-Y(\cdot+\alpha))^m B \frac{1}{n!}Y(\cdot)^n\right|_{\frac{R-3\beta}{2}}&\leq& (D_{\alpha,R}-2)^2 \, \delta^2 |P|^2_{R},\\
\left|\delta\sum_{m+n\geq 1} \frac{1}{m!}(-Y(\cdot+\alpha))^m P(\cdot) \frac{1}{n!}Y(\cdot)^n\right|_{\frac{R-3\beta}{2}}&\leq&(D_{\alpha,R}-2) \, \delta^2 |P|^2_{R},\\
\left|B \sum_{n\geq 2}\frac{1}{n!} [G]^n\right|_{\frac{R-3\beta}{2}}&\leq& \delta^2 |P|^2_{R}.
\end{eqnarray*}
Note that (\ref{eq_homo}) implies $BY(\cdot)-Y(\cdot+\alpha)B-B [G]- \delta P(\cdot)=0$. We thus get
$$|\tilde P|_{\frac{R-3\beta}{2}}\leq  D_{\alpha,R}^2 \, \delta^2 |P|^2_{R}.$$
With $\tilde P_1:=\delta^{-2} \tilde P+\sum_{j\geq2}\frac{(-\delta)^{j-2}}{j!} B[B^{-1}P]^j$, we have
$$Be^{[G]}+\tilde P(\cdot)=B-\delta [P] +\delta^2 \tilde P_1(\cdot).$$
By a direct calculation, we can see that
$$B-\delta [P] ={\rm Id}+(b_0-\delta b_1)-\frac{\delta}{2}(b_0b_1+b_1b_0).$$
Then, with $P_1:=\tilde P_1-\frac12b_1^2-\delta^{-2}\sum_{j\geq 3} \frac{1}{j!}(b_0-\delta b_1)^j$, we obtain (\ref{first_ave}).
Note that $b_0$ is nilpotent. Thus, combining with (\ref{esti_PX_on_T}), we get (\ref{first_ave_esti}).

If $\alpha\in {\rm DC}_d(\gamma, \tau)$, then, by (\ref{homo_Y}), we have
\begin{eqnarray*}
|Y|_{\frac{R}{2}}
&\leq&10 \delta |P|_{R} \sum_{n\in\Z^d} \frac{e^{-\frac{R}{2}|n|}}{|e^{{\rm i}\la n,\alpha\ra}-1|^3}\\
&\leq& 20 \gamma^{-3}\delta |P|_{R} \sum_{n\in\Z^d} |n|^{3\tau}e^{-\frac{R}{2}|n|} \\
&\leq& 40 \gamma^{-3}\delta |P|_{R} \int_{0}^{+\infty}  x^{d-1}x^{3\tau}e^{-\frac{R}{2}x} dx,
\end{eqnarray*}
where the above integral can be estimated as
$$\int_{0}^{+\infty}  x^{d-1}x^{3\tau}e^{-\frac{R}{2}x} dx\leq 2+ \int_{0}^{+\infty} x^{4\tau}e^{-\frac{R}{2}x} dx\leq 2^{4\tau+2}\, \Gamma(4\tau+2)\cdot R^{-(4\tau+1)}.$$
The rest proof of (\ref{first_ave_esti_DC}) is similar to that of (\ref{first_ave_esti}).
\end{proof}

Since $\tilde X$ is homotopic to identity by construction,
we have $$\rho(\alpha, B-\delta P(\cdot))=\rho(\alpha, e^{b_0-\delta b_1}+ \delta^2 P_1(\cdot)).$$
Let $d(\delta):={\rm det} (b_0-\delta b_1)$.
By a direct calculation, we get
\begin{equation}\label{d_delta}
d(\delta)=-\delta[X^2_{11}]\zeta+\delta^2\left([X^2_{11}][X^2_{12}]-[X_{11}X_{12}]^2\right).
\end{equation}
As we will see, $d(\delta)$ is the key quantity in our estimates on the size of the   gaps.

\subsection{Criterion for quantitative bounds of spectral gaps}

In this subsection, we give a criterion to obtain bounds on the size of the gaps in terms of the information provided by quantitative reducibility. With this criterion at our disposal, the exponential decay of the spectral gaps in various settings follows at once.

\begin{theorem}\label{thm_upperbound}
Let $\alpha\in \T^d$ with $R> 3\beta(\alpha)\geq 0$, $\kappa\in(0,\frac{1}{4})$, and $V\in C^{\omega}(\T^d, \R)$ be a non-constant function.
Let $E$ be an edge point of the spectral gap $G(V)$.
Assume that there are $\zeta\in(0,\frac12)$ and $X\in C_R^\omega(\T^d, {\rm PSL}(2,\R))$ such that
\begin{equation}\label{reduce_to_parabolic}
X(\cdot+\alpha)^{-1}S_E^{V}(\cdot)X(\cdot)=
\begin{pmatrix}
1 & \zeta \\
0 & 1
\end{pmatrix}.
\end{equation}
Then the following holds:
\begin{enumerate}
\item If
\begin{equation}\label{smallness_zeta_beta}
 |X|_R^{14} \zeta^\kappa\leq 10^{-5}D^{-4}_{\alpha,R},
\end{equation}
then $\zeta^{1+\kappa}\leq |G(V)|\leq \zeta^{1-\kappa}$, where $D_{\alpha, R}$ is the constant defined in (\ref{D1}).
\item In particular, for $\alpha\in {\rm DC}_d(\gamma,\tau)$, if
\begin{equation}\label{smallness_zeta_diophantine}
 |X|_R^{14} \zeta^\kappa\leq 10^{-5}D_{\tau}^{-4} \gamma^{12} R^{4(4\tau+1)},
 \end{equation}
then $\zeta^{1+\kappa}\leq |G(V)|\leq \zeta^{1-\kappa}$, where $D_{\tau}$ is the constant  defined in (\ref{D2}).
\end{enumerate}
\end{theorem}

\begin{remark}\label{remark_frequency}
 We remark that the optimal condition for reducibility at the edge points of spectral gaps was assumed to be $R>2\beta(\alpha)$, which was first conjectured by Avila-Jitomirskaya \cite{AvilaJito1}. For technical reasons, we have to require $R>3\beta(\alpha)$ in this approach (due to Lemma \ref{prop_ave}).\end{remark}

Before giving the proof of Theorem \ref{thm_upperbound},  we first make some technical preparations:
\begin{lemma}\label{z1-estimate}
For any $X\in C^\omega(\T^d,{\rm PSL}(2,\R))$, $[X_{11}^2]\geq (2|X|_{\T^d})^{-2}$.
\end{lemma}
\begin{proof}  The proof is essentially contained in  Lemma 4.2 of \cite{AYZ1}, we include the proof here for completeness.   Let $$u_1(\theta):=\begin{pmatrix}
X_{11}(\theta)\\
X_{21}(\theta)
\end{pmatrix},
 \quad u_2(\theta):=\begin{pmatrix}
X_{12}(\theta)\\
X_{22}(\theta)
\end{pmatrix}
.$$ Since $|{\rm det}X(\theta)|=1$,  we have
$\|u_1\|_{L^2(\T^d)} \|u_2\|_{L^2(\T^d)}> 1$, which implies that
$$\|X_{11}\|_{L^2(\T^d)}+\|X_{21}\|_{L^2(\T^d)}=\|u_1\|_{L^2(\T^d)}>\|u_2\|_{L^2(\T^d)}^{-1}>(|X|_{\T^d})^{-1}.$$
By (\ref{reduce_to_parabolic}), we know $X_{21}(\cdot+\alpha)=X_{11}(\cdot)$. So
$[X_{11}^2]=\|X_{11}\|_{L^2(\T^d)}^2\geq (2|X|_{\T^d})^{-2}$.
\end{proof}

Once we have Lemma \ref{z1-estimate}, then we have the following key observation for the transformation $X(\cdot)$.
\begin{lemma}\label{esti_X11_X12}
 For any $\kappa\in (0,\frac1{4})$, if
 \begin{equation}\label{x2} |X|_{R}\, \zeta^{\frac\kappa2}\leq \frac{1}{4},\end{equation} then the following holds:
\begin{align}\label{ineqzz}
0< \frac{[X^2_{11}]}{[X^2_{11}][X^2_{12}]-[X_{11}X_{12}]^2} &\leq \frac12\zeta^{-\kappa},   \\ \label{ineqzz-1}
[X^2_{11}][X^2_{12}]-[X_{11}X_{12}]^2 &\geq  8\zeta^{2\kappa}.
\end{align}
\end{lemma}

\begin{proof}
Assume by contradiction that $$\frac{[X^2_{11}]}{[X^2_{11}][X^2_{12}]-[X_{11}X_{12}]^2}> \frac12 \zeta^{-\kappa}.$$
The quadratic polynomial
$$Q(z):=[(X_{12}-zX_{11})^2]=[X_{11}^2]z^2-2[X_{11}X_{12}]z+[X_{12}^2]$$
attains its minimum when $z=\frac{[X_{11}X_{12}]}{[X_{11}^2]}$, and we have
$$Q\left(\frac{[X_{11}X_{12}]}{[X_{11}^2]}\right)=\left[\left(X_{12}- \frac{[X_{11}X_{12}]}{[X_{11}^2]} X_{11}\right)^2\right]=\frac{[X^2_{11}][X^2_{12}]-[X_{11}X_{12}]^2}{[X^2_{11}]}< 2\zeta^{\kappa}.$$
Hence, $X_{12}=\frac{[X_{11}X_{12}]}{[X_{11}^2]}X_{11}+\sigma$ for some $\sigma: \T^d \to \R$ with $[\sigma^2]< 2\zeta^{\kappa}$.

By (\ref{reduce_to_parabolic}), we can check that
$$
X_{11}(\cdot+\alpha)X_{12}(\cdot)-X_{11}(\cdot)X_{12}(\cdot+\alpha)=1+\zeta X_{11}(\cdot+\alpha)X_{11}(\cdot).
$$
Hence, we obtain
$$
X_{11}(\cdot+\alpha)\sigma(\cdot)-X_{11}(\cdot)\sigma(\cdot+\alpha)=1+\zeta X_{11}(\cdot+\alpha)X_{11}(\cdot).
$$
By Cauchy-Schwarz inequality and $(\ref{x2})$, we have
\begin{equation}\label{contradiction_1}
\left|[X_{11}(\cdot+\alpha)\sigma(\cdot)-X_{11}(\cdot)\sigma(\cdot+\alpha)]\right|\leq  \frac{\sqrt2}2.
\end{equation}
On the other hand, $\zeta |X_{11}(\cdot+\alpha)X_{11}(\cdot)|_{\T^d}\leq \frac{1}{16}\zeta^{1-\kappa}$, which implies
\begin{equation}\label{contradiction_2}
\left|[1+\zeta X_{11}(\cdot+\alpha)X_{11}(\cdot)]\right|>1-\frac{1}{16}\zeta^{1-\kappa}.
\end{equation}
By (\ref{contradiction_1}) and (\ref{contradiction_2}), we reach a contradiction.

Combining with Lemma \ref{z1-estimate}, we get $[X_{11}^2]\geq \frac14 |X|_{\T^d}^{-2}\geq 4 \zeta^{\kappa}$, which implies $(\ref{ineqzz-1})$.\end{proof}

\begin{proof}[Proof of Theorem \ref{thm_upperbound}.]
By (\ref{d_delta}), the quantity $d(\delta)={\rm det}(b_0-\delta b_1)$ satisfies
\begin{align*}
  d(\delta) &= -\delta [X^2_{11}]\zeta+ \delta^2([X^2_{11}][X^2_{12}]-[X_{11}X_{12}]^2) \\
   &=  \delta([X^2_{11}][X^2_{12}]-[X_{11}X_{12}]^2)\left(\delta-\frac{[X^2_{11}]\zeta}{[X^2_{11}][X^2_{12}]-[X_{11}X_{12}]^2}\right).
\end{align*}

Fix $\kappa\in(0,\frac{1}{4})$, and let $\delta_1=\zeta^{1-\kappa}$.
If $\zeta>0$ satisfies (\ref{smallness_zeta_beta}), then it is obvious that $0<\delta_1\leq D^{-1}_{\alpha,R}|X|_R^{-2}$.
In particular, for $\alpha\in {\rm DC}_d(\gamma,\tau)$, (\ref{smallness_zeta_diophantine}) implies that $0<\delta_1\leq D^{-1}_{\tau} \gamma^3 R^{4\tau+1} |X|_R^{-2}$. Hence, we can apply Lemma \ref{prop_ave}, and conjugate the system to the cocycle $(\alpha, e^{b_0-\delta_1 b_1}+\delta_1^2 P_1)$.

As shown symbolically in Figure \ref{f.graph}, in order to show  $|G(V)|\leq \delta_1$,
it is sufficient to show that $\rho(\alpha, e^{b_0-\delta_1 b_1}+\delta_1^2 P_1)>0$. By $(\ref{smallness_zeta_beta})$ or $(\ref{smallness_zeta_diophantine})$,
one has $|X|_{R} \zeta^{\frac\kappa2} \leq \frac{1}{4}$. Then we can  apply
 Lemma \ref{esti_X11_X12}, and get $$\frac{[X^2_{11}]\zeta}{[X^2_{11}][X^2_{12}]-[X_{11}X_{12}]^2} \leq \frac{1}{2}\delta_1.$$
Hence, for $d(\delta_1)={\rm det}(b_0-\delta_1 b_1)$, we have
\begin{equation}\label{lower_determinant}
d(\delta_1) \geq   \zeta^{1-\kappa} \cdot 8 \zeta^{2\kappa}\cdot \frac12 \zeta^{1-\kappa}
   =4 \zeta^{2}.
\end{equation}
Following the expressions of $b_0$ and $b_1$ in Lemma \ref{prop_ave}, we have
\begin{equation}\label{Mdelta1}
|b_0-\delta_1 b_1|\leq \zeta + \delta_1 (1+\zeta) |X|_{\T^d}^2 \leq 2 \, \zeta^{1-\kappa}  |X|_{R}^2.
\end{equation}
In view of Lemma 8.1 in \cite{HouYou}, there exists ${\CP}\in {\rm SL}(2,\R)$, with
$|{\CP}|\leq 2\left(\frac{|b_0-\delta_1 b_1|}{\sqrt{d(\delta_1)}}\right)^{\frac12}$ such that
$${\mathcal P}^{-1} e^{b_0-\delta_1 b_1} {\CP}= R_{\sqrt{d(\delta_1)}}.$$
Combining \eqref{lower_determinant} and \eqref{Mdelta1}, we have
$$\frac{|b_0-\delta_1 b_1|}{\sqrt{d(\delta_1)}}\leq \frac{2\,  \zeta^{1-\kappa} |X|_{R}^2 }{\sqrt{4 \zeta^{2}}}= |X|_{R}^2  \zeta^{-\kappa}.$$
Then, according to Lemma \ref{esti_rot_num} and Lemma \ref{prop_ave},
$$|\rho(\alpha, e^{b_0-\delta_1 b_1}+\delta_1^2 P_1)-\sqrt{d(\delta_1)}|\leq \delta_1^2 |{\mathcal P}|^2 |P_1|_{\T^d}
\leq  8 D^2_{\alpha,R} |X|_{R}^6  \zeta^{2-3\kappa}.$$
Under the assumption (\ref{smallness_zeta_beta}), combining with (\ref{lower_determinant}), we have
$$
4 D^2_{\alpha,R} |X|_{R}^6  \zeta^{1-3\kappa}< 1,
$$
which implies that
$$\rho{(\alpha, e^{b_0-\delta_1 b_1}+\delta_1^2 P_1)}\geq \sqrt{d(\delta_1)}-|\rho{(\alpha, e^{b_0-\delta_1 b_1}+\delta_1^2 P_1)}-\sqrt{d(\delta_1)}|>0.$$
In particular, when $\alpha\in {\rm DC}_d(\gamma,\tau)$, in view of (\ref{first_ave_esti_DC}), we have
$$|\rho{(\alpha, e^{b_0-\delta_1 b_1}+\delta_1^2 P_1)}-\sqrt{d(\delta_1)}|\leq 8 D^2_{\tau} \gamma^{-6} R^{-2(4\tau+1)}  |X|_{R}^6  \zeta^{2-3\kappa}.$$
Since (\ref{smallness_zeta_diophantine}) implies that
$$
4 D^2_{\tau} \gamma^{-6} R^{-2(4\tau+1)}  |X|_{R}^6  \zeta^{1-3\kappa}< 1,
$$
and we get $\rho{(\alpha, e^{b_0-\delta_1 b_1}+\delta_1^2 P_1)}>0$.
This concludes the proof of the upper bound estimate.

Let us now consider the lower bound estimate on the size of the gap. Let $\delta_2:=\zeta^{1+\kappa}$. We are going to show that $|G(V)| \geq \delta_2$. We first note that
$$\delta_2^2\left|[X^2_{11}][X^2_{12}]-[X_{11}X_{12}]^2\right|\leq  2 \zeta^{2+2\kappa}|X|_{R}^{4},$$
and, by Lemma \ref{z1-estimate}, one has
$\delta_2[X^2_{11}]\zeta\geq \frac{1}{4} \zeta^{2+\kappa} |X|_{R}^{-2}$.
Thus, if $\zeta$ is small enough such that
$
|X|_{R}^{6}\zeta^{\kappa}\leq \frac1{40}$ (which can be deduced from (\ref{smallness_zeta_beta}) or (\ref{smallness_zeta_diophantine})),
then
$$d(\delta_2)=-\delta_2[X^2_{11}]\zeta+\delta_2^2\left([X^2_{11}][X^2_{12}]-[X_{11}X_{12}]^2\right)<-\frac{1}{5}\zeta^{2+\kappa} |X|_{R}^{-2},$$
and hence
\begin{equation}\label{upper_determinant} \sqrt{-d(\delta_2)}> \frac{1}{\sqrt{5}}\zeta^{1+\frac\kappa2} |X|_{R}^{-1}.\end{equation}
In view of Proposition 18 of \cite{Puig06}, there exists ${\CP}\in {\rm SL}(2,\R)$, with
$|{\CP}|\leq 2\left(\frac{|b_0-\delta_2 b_1|}{\sqrt{-d(\delta_2)}}\right)^{\frac12}$ such that
$${\CP}^{-1} e^{b_0-\delta_2 b_1} \, {\CP}=\begin{pmatrix}
                                     e^{\sqrt{-d(\delta_2)} }& 0  \\[1mm]
                                     0 & e^{-\sqrt{-d(\delta_2)}}
                                   \end{pmatrix}.$$
Since $|X|_{R}^{6} \zeta^{\kappa}\leq \frac{1}{8}$, we have 
$$
|b_0-\delta_2 b_1|\leq \zeta + \zeta^{1+\kappa} (1+\zeta) |X|_{\T^d}^2 \leq 2 \zeta,
$$
and then, by $(\ref{upper_determinant})$, one has
$$\frac{|b_0-\delta_2 b_1|}{\sqrt{-d(\delta_2)}}\leq \frac{\sqrt{5}\cdot 2  \zeta}{\zeta^{1+\frac\kappa2} |X|_{R}^{-1}}= 2\sqrt{5}|X|_{R}  \zeta^{-\frac\kappa2}.$$
By $(\ref{first_ave_esti})$ of Lemma \ref{prop_ave}, we have
$${\CP}^{-1} \delta_2^2|P_1|_{(R-3\beta)/2}\CP  \leq 16 \sqrt{5} D^2_{\alpha,R}  \zeta^{2+\frac{3\kappa}2} |X|^5_{R}.$$
Then, under the condition (\ref{smallness_zeta_beta}), we have
$${\CP}^{-1} \delta_2^2|P_1|_{(R-3\beta)/2}\CP  \leq   -d(\delta_2),$$
 consequently,  the cocycle $(\alpha, e^{b_0-\delta_2 b_1}+\delta_2^2 P_1)$ is uniformly hyperbolic, and $E-\delta_2 \not \in \Sigma_{V,\alpha}$, which means that $|G(V)|\geq \zeta^{1+\kappa}$.
 In particular, if $\alpha\in {\rm DC}_{d}(\gamma, \tau)$,  by $(\ref{first_ave_esti_DC})$ of Lemma \ref{prop_ave},  we have
$${\CP}^{-1} \delta_2^2|P_1|_{R/2}\CP  \leq 16\sqrt{5} D^2_{\tau} R^{-2(4\tau+1)} \zeta^{2+\frac{3\kappa}2}   |X|^5_{R}.$$
Similarly as above, under the condition (\ref{smallness_zeta_diophantine}), we have
$|G(V)|\geq \zeta^{1+\kappa}$.\end{proof}

\subsection{Applications of the criterion -- upper bound}

As the first application of Theorem \ref{thm_upperbound},  for discrete  quasi-periodic Schr\"odinger operator with small potential, we get exponentially decaying upper bounds on the size of spectral gaps.  As we mentioned before, the result is perturbative for a multifrequency. However, it is non-perturbative in the case of a one-dimensional frequency.

\begin{corollary}\label{cor-local}
Consider the operator $H_{V,\alpha,\theta}$ with $V\in C_{r_0}^{\omega}(\T^d,\R)$ non-constant.
 \begin{enumerate}
 \item If  $\alpha\in {\rm DC}_d(\gamma,\tau)$, then for any $r\in(0,r_0)$, there exists $\varepsilon_*=\varepsilon_*(\gamma,\tau, r_0, r, d )>0$ such that if
$|V|_{r_0}=\varepsilon_0<\varepsilon_*$, then
$$|G_k(V)|\leq \varepsilon_0^{\frac23} e^{-2\pi r |k|},\quad \forall \ k\in\Z^d\backslash\{0\}.$$
 \item If $d=1$, $\beta(\alpha)=0$, and $|V|_{r_0}\leq c_0 r_0^3$ with $c_0$ the absolute constant in Theorem \ref{almostredth}, then there are
$C_{10}=C_{10}(r_0,\alpha)>0$ and $\vartheta=\vartheta(r_0) \in (0,r_0)$ such that
$$|G_k(V)|\leq C_{10} e^{-\vartheta |k|},\quad \forall \ k\in\Z\backslash\{0\}.$$
 \end{enumerate}
\end{corollary}

\begin{proof} We fist consider the case $\alpha\in {\rm DC}_d(\gamma,\tau)$. Fix $r\in(0,r_0)$ and set $\tilde r:=\frac{r_0+r}{2}$. Write the Schr\"odinger cocycle $(\alpha, S_{E}^{V}(\cdot))$ as $(\alpha,A_E+F_0(\cdot))$, where
$$A_E=\begin{pmatrix}
E & -1 \\
1 & 0
\end{pmatrix},\quad F_0(\cdot)=\begin{pmatrix}
-V(\cdot) & 0 \\
0 & 0
\end{pmatrix}.$$
Since we consider the case where $V$ is small, we have
$$E\in \Sigma_{V,\alpha} \subset [-2-\inf V(\theta), 2+\sup V(\theta)]\subset [-3,3].$$
Then the norm of $A_E$ is bounded uniformly with respect to $E$. Hence one can apply Theorem \ref{thm_gap_edge_SL} to obtain a uniform $\varepsilon_*=\varepsilon_*(\gamma,\tau, r_0, \tilde r, d )>0$ which is independent of $E$, such that if $|V|_{r_0}=\varepsilon_0<\varepsilon_*$, then $(\alpha,A_E+F_0(\cdot))$ is almost reducible. Moreover, since $2\rho(\alpha, S_{E_k^+}^{V}) -\la k,\alpha \ra \in\Z$, by Theorem \ref{thm_gap_edge_SL}, we have
\begin{equation*}
X(\cdot+\alpha)^{-1}S_{E_k^+}^{V}(\cdot)X(\cdot)=
\begin{pmatrix}
1 & \zeta \\
0 & 1
\end{pmatrix},
\end{equation*}
with
 $\zeta\leq \varepsilon^{\frac34}_{0} e^{-2\pi \tilde r|k|}$ and $ |X|_{r''}\leq  D_1(\gamma,\tau, r_0, d)e^{\frac32\pi r'' |k|}$ for any $r''\in (0,\tilde r)$.

 Let $\kappa:=\frac{\tilde r-r}{9\tilde r}$ and $R:=\varepsilon_0^{\frac{\tilde r -r}{60\tilde r(4\tau+1)}}$. Then for any $k\in\Z^d\backslash\{0\}$,
we have
$$|X|_{R}^{14} \zeta^\kappa\leq   D_1^{14} e^{21\pi R |k|} \cdot
\varepsilon_0^{\frac{\tilde r- r}{12\tilde r}} e^{-\frac{2\pi(\tilde r- r)}{9}|k|}\leq 10^{-5} D_{\tau}^{-4}\gamma^{12} R^{4(4\tau+1)}.$$
The above inequality is possible since $\varepsilon_0$ is sufficiently small (the smallness only depend on $\gamma$, $\tau$,  $r_0$,  $\tilde r$, $d$).
Hence, by Theorem \ref{thm_upperbound}, we have
$$|G_k(V)|\leq \zeta^{1-\kappa}\leq \varepsilon_0^{\frac{8\tilde r+r}{12\tilde r}} e^{-\frac{2\pi}{9}(8\tilde r+r)|k|}\leq \varepsilon_0^{\frac23} e^{-2\pi r|k|},\quad \forall \  k\in\Z^d\backslash\{0\}. $$
This concludes the proof of the first statement.

Now we consider the case where $d=1$, $\beta(\alpha)=0$ and $|V|_{r_0}\leq c_0 r_0^3$.
By Theorem \ref{prop_duality_para}, there exists $r_1=r_1(r_0)\in (0,r_0)$, such that for any $r\in(0,r_1)$, one has $X\in C^{\omega}_{r}(\T,{\rm PSL}(2,\R))$ and $\zeta\in \R$ such that
$$
X(\cdot+\alpha)^{-1}S_{E_k^+}^{V}(\cdot)X(\cdot)=
\begin{pmatrix}
1 & \zeta \\
0 & 1
\end{pmatrix}.
$$
Moreover, there exists $k_1=k_1(\alpha,r_0,r)>0$ such that if $|k| \geq k_1$, then for some $n=n(k)\in\Z$ with $|n|\geq \frac{|k|}{4}$, one has
$\zeta\leq C_3(\alpha,r_0,r) e^{-\frac{\pi r}{5} |n|}$, and for any $r''\in (0,r]$,
$|X|_{r''}\leq  C_4(\alpha, r_0, r'') e^{22\pi r'' |n|}$.

Fix any $\kappa\in (0,\frac{1}{4})$ and let $R:=\frac{\kappa r_1}{2100}$, $r'=\frac56 r_1$.
By a direct calculation, if $|k|$ is large enough (hence $|n|$ is large enough), then
$$|X|_{R}^{14} \zeta^\kappa\leq C_4^ {14} C_3^{\kappa} e^{-\pi\kappa r_1 (\frac{1}{6}-\frac{11}{75})|n|}    \leq 10^{-5}D_{\alpha,R}^{-4}.$$
Thus, by Theorem \ref{thm_upperbound}, we have
$$|G_k(V)|\leq \zeta^{1-\kappa}\leq C_{10}(\alpha, r_0)e^{-\frac{\pi r_1}{6}(1-\kappa)|n|}\leq C_{10}(\alpha, r_0) e^{-\frac{3\pi r_1}{24}|k|}.$$
Modifying the constant coefficient $C_{10}$, we get the exponential upper bound for all $k\in\Z\backslash\{0\}$. This thus concludes the whole proof. 
\end{proof}

As the second application of Theorem \ref{thm_upperbound}, we get exponential decay of the upper bounds of the spectral gaps for subcritical quasi-periodic Schr\"odinger operators.

\begin{corollary}\label{cor_global}
Consider the Schr\"odinger operator $H_{V,\alpha,\theta}$ with $\beta(\alpha)=0$.
For a typical potential $V\in C^{\omega} (\T,\R)$,
 there exist constants $C,\vartheta>0$ depending only on $V$ and $\alpha$, such that
 $$|G_k(V)| \leq C e^{-\vartheta |k|}, \quad  \forall \  k\in \Z\backslash \{0\} \; { with } \;  \overline{G_k(V)} \cap \Sigma_{V,\alpha}^{\mathrm{sub}}\neq \emptyset.$$
\end{corollary}

\begin{proof}
The proof is the same as that of Corollary \ref{cor-local} (2). One only needs to replace Proposition \ref{prop_duality_para} with  Corollary \ref{prop_duality_global}.
\end{proof}

If we restrict ourselves to subcritical almost Mathieu operators, we obtain  much better estimates:

\begin{corollary}\label{cor_mathieu_upperbound}
Consider the almost Mathieu operator $H_{\lambda,\alpha,\theta}$ with $0<\lambda<1$.   For any  $0<\xi<1$, the following assertions hold.
\begin{enumerate}
\item For $\alpha\in \R\backslash\Q$ with $\beta(\alpha)=0$, there exists
$C_{11}=C_{11}(\lambda,\alpha,\xi)>0$ such that
$$|G_k(\lambda)|\leq C_{11}(\lambda,\alpha,\xi) \lambda^{ \frac{\xi}{12} |k|}, \quad \forall \ k\in \Z\backslash\{0\}.$$
\item For $\alpha\in{\rm DC}$, there exists
$C_{12}=C_{12}(\lambda,\alpha,\xi)>0$ such that
$$|G_k(\lambda)|\leq C_{12}(\lambda,\alpha,\xi) \lambda^{\xi |k|}, \quad \forall \ k\in \Z\backslash\{0\}.$$
\end{enumerate}
\end{corollary}

\begin{proof} We first consider the case $\beta(\alpha)=0$. If $2\rho{(\alpha,S_E^{\lambda})}-k\alpha \in\Z$, then by Theorem \ref{thm_almost_almost-2}, for any $0<\tilde r<-\frac{1}{2\pi}\ln\lambda$, there exist $X\in C_{\tilde r}^{\omega}(\T, {\rm PSL}(2,\R))$,  $k_2=k_2(\lambda, \alpha, \tilde r)>0$ such that if  $|k|\geq k_2$, then we have
$$
X(\cdot+\alpha)^{-1}S_{E_k^+}^{\lambda}(\cdot)X(\cdot)=
\begin{pmatrix}
1 & \zeta \\
0 & 1
\end{pmatrix}
$$
with $\zeta \leq C_5(\lambda,\alpha, \tilde r) e^{-\frac{2\pi \tilde r}{3} |n|}$ and $|X|_{r''}\leq C_6(\lambda,\alpha, r'')  e^{22\pi r'' |n|}$ for any $r''\in(0,\tilde r]$, where $n=n(k)\in \Z$ satisfies $|n|\geq \frac{|k|}{4}$.

For any $\xi\in(0,1)$, let $r:=-\frac{\xi\ln \lambda}{2\pi}$, $\tilde r:=\frac12(r-\frac{\ln \lambda}{2\pi})$, $\kappa:=\frac{\tilde r-r}{10\tilde r}$ and $R:=\frac{\kappa \tilde r}{700}$.
By a direct calculation, if $k$ is large enough (thus $n$ is large enough too), then we have
$$  |X|_{R}^{14} \zeta^\kappa\leq C_6^ {14} C_5^{\kappa} e^{-\pi\kappa\tilde r(\frac{2}{3}-\frac{11}{25})|n|}\leq 10^{-5} D_{\alpha,R}^{-4}.$$
Hence, by Theorem \ref{thm_upperbound}, we have
$$|G_k(\lambda)|\leq \zeta^{1-\kappa}\leq  C_5^{1-\kappa} e^{-\frac{\pi}{15}(9\tilde r+r)|n|}\leq C_{11}(\lambda,\alpha,\xi) \lambda^{ \frac{\xi}{12} |k|},$$
since $|n|\geq \frac{|k|}{4}$.
Modifying the constant coefficient $C_{11}$, we get the exponential upper bound for all $k\in\Z\backslash\{0\}$.

If $\alpha\in{\rm DC}$,  the proof is similar to that of Corollary \ref{cor-local} (1). One only needs to replace Theorem \ref{thm_gap_edge_SL} with Proposition \ref{redu amo case} and corresponding arguments. Therefore, for any $0<r <-\frac{ \ln \lambda}{2\pi}$, one has $|G_k(\lambda)|\leq C_8(\lambda ,\alpha, r)  e^{-2\pi r|k|}$. Now for any $0<\xi<1$, let $r:=-\frac{\xi \ln \lambda}{2\pi}$,  which gives the desired result.
 \end{proof}

%
%

In the same way, we can also derive a criterion to obtain quantitative upper bounds on the size of spectral gaps for continuous quasi-periodic Schr\"odinger operators on $L^2(\R)$:
$$({\CL}_{V,\varpi}\psi)(x)=-\psi''(x)+V(\varpi x)\psi(x)$$
with $V\in C^{\omega}(\T^{d},\R)$ sufficiently small and $\varpi\in {\rm DC}_{d}$, $d\geq 2$.
As the proof is the same as Theorem \ref{thm_upperbound}, we state the result without proof. In fact, as the reader can see, our result is based on  Moser-P\"oschel argument \cite{Moser-Poschel}, which was first stated in the case of a continuous operator.

%
%

\begin{theorem}[Continuous version of Theorem \ref{thm_upperbound} -- upper bound]
\label{thm_upperbound_continuous}
Consider the operator ${\CL}_{V,\varpi}$ with $V\in C^{\omega}(\T^{d},\R)$ non-constant and $\varpi\in {\rm DC}_{d}(\gamma, \tau)$, $d\geq 2$.
Assume that there are $\zeta>0$ and $X\in C_{R}^\omega(\T^d, {\rm PSL}(2,\R))$ for some $R>0$ such that
$$\partial_{\varpi}X=\begin{pmatrix}
0 & 1\\
V(\theta) -E & 0
\end{pmatrix}X-X\begin{pmatrix}
0 & \zeta\\
0 & 0
\end{pmatrix}.$$
Fix $\kappa\in (0, \frac{1}{4})$.
If $ |X|_{R}^{14} \zeta^\kappa\leq 10^{-5} D^{-4}_{\tau} \gamma^{12} R^{4(4\tau+1)} $,
then $|G(V)|\leq  \zeta^{1-\kappa}$.
\end{theorem}

\noindent

\begin{theorem}\label{sharpcon}
Let $\varpi  \in {\rm DC}_d$ and $V\in C_{r_0}^{\omega}(\T^{d},\R)$. For any $r\in(0,r_0)$, there exists $\varepsilon_0= \varepsilon_0(V, \varpi, r_0, r)>0$  such that if $|V|_{r_0} < \varepsilon_0$, then for the operator ${\CL}_{V,\varpi}$,
$$ |G_k(V)| \leq \varepsilon_0^{\frac23} e^{- r |k|}, \quad \forall \  k\in \Z^d\backslash\{0\}.$$
\end{theorem}

\begin{proof}The proof is exactly the same as that of Corollary \ref{cor-local} (1). One only needs to replace Theorem \ref{thm_gap_edge_SL}  with Theorem \ref{thm_gap_edge_algebra}, and replace Theorem \ref{thm_upperbound} with Theorem \ref{thm_upperbound_continuous}. \end{proof}

%

\subsection{Applications of the criterion -- lower bound}\label{subsec_lower_bound}

For general Schr\"odinger operators, the spectral gaps may collapse since the corresponding off-diagonal element $\zeta$ may vanish.
However, this is not true for non-critical almost Mathieu operators $H_{ \lambda,\alpha,\theta}$ \cite{AvilaJito,AYZ}. Now we further derive exponentially decaying lower bounds on the  size of the gaps $G_k(\lambda)$.


\begin{corollary}\label{thm_lowerbound}
Consider the almost Mathieu operator $H_{\lambda,\alpha,\theta}$ with $0<\lambda<1$, $\alpha\in {\rm DC}(\gamma,\tau)$. There exists an absolute constant $\tilde\xi>1$ such that
$$|G_k(\lambda)| \geq C(\lambda,\alpha) \lambda^{\tilde\xi |k|}, \quad \forall \ k\in\Z\backslash\{0\}.$$
\end{corollary}

\smallskip


The following proposition plays a key role in Avila-You-Zhou's proof in solving the non-critical ``Dry Ten Martini Problem''  \cite{AYZ}. 
We point out that it works for all irrational frequencies while in \cite{AYZ} the authors mainly deal with Liouvillean frequencies.

\begin{prop}[Avila-You-Zhou \cite{AYZ}]\label{estimate-kappa}
Let $\alpha\in\R\backslash \Q$, $0<\lambda<1$, $E\in\Sigma_{\lambda,\alpha}$ and $0<R<-\frac{1}{2\pi}\ln \lambda$.
There exists $T=T(R,\lambda)>0$ such that for $\varepsilon>0$ sufficiently small, there is no $Z \in C_{R}^\omega(\T, {\rm PSL}(2,\R))$ satisfying
\begin{equation}\label{almost}
Z(\cdot+\alpha)^{-1}S_{E}^{\lambda}(\cdot)Z(\cdot)=\mathrm{Id}+F(\cdot),
\end{equation}
with
$|Z|_R \leq \varepsilon^{-1}$, $|F|_R\leq\varepsilon^T$.
\end{prop}

\begin{remark}\label{remark_T}
If one checks the argument in  \cite{AYZ}, it  gives us $T(R,\lambda)=C' \left(\frac{\ln \lambda}{2\pi R}\right)^2$ where $C'$ is a large absolute constant. If (\ref{almost}) in the above proposition can be promoted to  reducibility, i.e.,
$$Z(\cdot+\alpha)^{-1}S_{E}^{\lambda}(\cdot)Z(\cdot)={\rm Id}+F$$
for some constant $F$, then one can actually obtain more precise estimates on $C'$.
\end{remark}

\begin{proof}
[Proof of Corollary \ref{thm_lowerbound}]
By Proposition \ref{redu amo case} (2), for $r=-\frac{1}{4\pi}\ln\lambda$, and any $k\in\Z\backslash\{0\}$,
there exist $\zeta\in\R$ and
$X\in C^\omega_r(\T,{\rm PSL}(2,\R))$ such that
$$X(\cdot+\alpha)^{-1}S_{E_k^+}^{\lambda}(\cdot)X(\cdot)
= \begin{pmatrix}
1 & \zeta \\
0 & 1 \\
\end{pmatrix}
$$
with $\zeta\leq C_{8}(\lambda, \alpha) e^{-2\pi r |k|}$ and $|X|_{r''} \leq C_{9}(\lambda, \alpha) e^{\frac{3\pi r''}{2}|k|}$ for any $0<r''< r$.
Then for $R:=\frac{r}{210}$ and $\kappa:=\frac{1}{10}$, it is easy to see that for $|k|$ large enough,
$$|X|_{R}^{14} \zeta^\kappa \leq C_{9}^{14} C_{8}^{\frac{1}{10}} e^{-\frac{\pi r}{10} |k|}\leq 10^{-5} D^{-4}_{\tau} \gamma^{12} R^{4(4\tau+1)}.$$
Hence, by Theorem \ref{thm_upperbound}, we get $|G_k(\lambda)|\geq \zeta^{\frac{11}{10}}$.

Now it suffices to obtain a lower bound on $\zeta$.
By Proposition \ref{estimate-kappa} (see also Remark \ref{remark_T}), for $|k|$ large enough, we have $$\zeta> C_{9}^{-T(R,\lambda)} e^{-\frac{\pi T(R,\lambda) }{140} r |k|} >C(\lambda,\alpha)e^{- \frac{40\pi}{11}\tilde{\xi} r|k|},$$ where $\tilde{\xi}>1$ is an absolute constant.
As a consequence, one can conclude that
$$|G_k(\lambda)|>C(\lambda,\alpha) \lambda^{\tilde{\xi}|k|}.$$
By modifying the constant $C$, we get the above lower bound for every $k\in\Z\backslash\{0\}$. This finishes the proof of Corollary \ref{thm_lowerbound}.
\end{proof}

\begin{proof}[Proof of Theorem \ref{thm_bounds_Mathieu}]
By Aubry duality, it is enough to consider the case where $0<\lambda<1$. Then the result follows from
 the assertion (2) of Corollary \ref{cor_mathieu_upperbound} and Corollary \ref{thm_lowerbound}.
\end{proof}

\section{Homogeneous spectrum}\label{sec_homo}

\subsection{Criterion for homogeneity of spectrum}

We first present a general criterion for establishing the homogeneous spectrum via gap estimates. The idea first appeared in Corollary 3 of \cite{HA} and then Theorem H of \cite{DGL1}. The philosophy is that the H\"older continuity together with some decay of the spectral gaps should yield homogeneity of the spectrum. In the following,  contrary to \cite{DGL1,HA}, we give a criterion which works for a large potential and a Liouvillean frequency. We emphasize that the homogenous spectrum for $\beta(\alpha)=0$  can be obtained if the exponential decay of the spectral gaps is established.

\begin{theorem}\label{thm_homo_spec}
Let $\alpha\in\T^d$ with $\beta=\beta(\alpha)\geq 0$, and let $V\in C^{\omega}(\T^d,\R)$ be non-constant. Assume that
\begin{itemize}
\item [(H1)] $N=N_{V,\alpha}$ is $\sigma-$H\"older continuous on $[a,b]$ with $0<\sigma<1$.
\end{itemize}
Then for any $\tilde\epsilon>0$, there exists $C_{12}=C_{12}(V,\alpha,\sigma,\tilde\epsilon)>0$ such that
for any two spectral gaps $G_k(V)$ and $G_{k'}(V)$ with
$$\overline{G_k(V)}\cap [a,b], \;\  \overline{G_{k'}(V)}\cap [a,b] \neq \emptyset,$$
we have
\begin{align}
\mathrm{dist}(G_k(V),G_{k'}(V)) &\geq  C_{12} e^{-(\frac{\beta}{\sigma}+\tilde\epsilon) |k-k'|}, \;\  if  \    k\neq k',\label{distance_1}\\
|E_k^--\underline{E}|&\geq C_{12} e^{-(\frac{\beta}{\sigma}+\tilde\epsilon) |k|},\;\ if \ a=\underline{E}, \label{distance_01}\\
|E_k^+-\overline{E}| &\geq C_{12} e^{-(\frac{\beta}{\sigma}+\tilde\epsilon) |k|}, \;\ if \ b=\overline{E}. \label{distance_02}
\end{align}
Furthermore, for an interval  $[a, b]$  with $a=\underline{E}$ or $E_{m}^+$ for some $m\in\Z^d\backslash\{0\}$ and $b=\overline{E}$ or $E_{n}^-$ for some $n\in\Z^d\backslash\{0\}$, if
\begin{itemize}
\item [(H2)]  there exist $C$, $\vartheta>0$, which only depend on $V,\alpha$, such that $|G_k(V)|\leq C e^{-\vartheta |k|}$ if $\overline{G_k(V)}\cap [a,b] \neq \emptyset$,
\item [(H3)] $ \beta=\beta(\alpha)\leq \frac{\sigma \vartheta}{2}$,
\end{itemize}
 hold, then there exists $\mu=\mu(a,b,V, \alpha, \sigma, C, \vartheta, d)\in (0,1)$, such that
$$| (E-\epsilon,E+\epsilon) \cap \Sigma_{V, \alpha}| > \mu\epsilon,\quad \forall \ E\in \Sigma_{V, \alpha}\cap [a,b], \;\  \forall \ 0<\epsilon\leq {\rm diam}\Sigma_{V, \alpha}.$$
\end{theorem}

\begin{remark}
Since the decay rate $\vartheta$ is related to the analytic radius of the potential, ${\rm (H3)}$ means that the radius should be relatively large compared to $\beta$. This kind of condition is necessary for the  homogeneity of the spectrum, see counterexamples of Avila-Last-Shamis-Zhou \cite{ALSZ}.
\end{remark}

\begin{proof} For two different gaps $G_k(V)$, $G_{k'}(V)$ with $\overline{G_k(V)}\cap [a,b]\neq \emptyset$ and $\overline{G_{k'}(V)}\cap [a,b]\neq \emptyset$, without loss of generality, we assume that $E_{k}^+\leq E_{k'}^-$. Hence $$\mathrm{dist}(G_k(V),G_{k'}(V) )=E_{k'}^--E_k^+.$$
On the one hand, the $\sigma-$H\"older continuity of $N$ on $[a,b]$ implies
$$|N(E_{k}^-)-N(E_{k'}^+)|\leq c' (E_k^--E_{k'}^+)^{\sigma}$$
for some constant $c'>0$ independent of $E$.
On the other hand, the definition of $\beta=\beta(\alpha)$ means that for any $\tilde\epsilon>0$, there exists $\tilde c=\tilde c(\alpha, \sigma\tilde\epsilon)>0$ such that
$$
|N(E_{k}^-)-N(E_{k'}^+)| \geq \|\la k-k', \alpha \ra \|_{\T} \geq  \tilde c \, e^{-(\beta+\sigma\tilde\epsilon)|k-k'|}.$$
Combining the above estimates, we conclude that
$$\mathrm{dist}(G_k(V),G_{k'}(V) ) \geq  \left(\frac{\tilde c}{c'}\right)^{\frac1\sigma} e^{-(\frac{\beta}{\sigma}+\tilde\epsilon) |k-k'|},$$
which gives (\ref{distance_1}).  The proof of  (\ref{distance_01}) and (\ref{distance_02}) is similar,  we omit the details.

By assumption (H3), we have $D:=\frac{\vartheta\sigma+2\beta}{2\beta}>2$.
Given any $E\in \Sigma_{V, \alpha}\cap [a,b]$ and any $\epsilon>0$,
let $${\CN}={\CN}(E,\epsilon):=\{k \in \Z^d \backslash\{0\}:  G_k(V) \cap (E-\epsilon, E+\epsilon)\neq \emptyset, \;  \overline{G_k(V)} \cap [a,b] \neq \emptyset \},$$ and let $k_0\in {\CN}$ be such that $|k_0|=\min_{k\in {\CN}}|k|$.
By (\ref{distance_1})--(\ref{distance_02}), there exists a constant $C_{12}=C_{12}(V,\alpha,\sigma,\frac{\vartheta}{D}-\frac{\beta}{\sigma})>0$ such that
\begin{eqnarray}
&&{\rm dist}(G_k(V),G_{k_0}(V) )\geq C_{12} e^{-\frac{\vartheta}{D} |k-k_0|}\geq C_{12} e^{-\frac{2\vartheta}{D} |k|},\   \forall \ k\in{\CN}\backslash\{k_0\},\label{dist_1}\\
&&|E_k^--\underline{E}|\geq C_{12} e^{-\frac{\vartheta}{D} |k|},\  \forall \  k\in{\CN}, \quad {\rm if}  \   a=\underline{E},\label{dist_01}\\
&&|E_k^+-\overline{E}| \geq C_{12} e^{-\frac{\vartheta}{D} |k|},\  \forall \  k\in{\CN}, \quad {\rm if}  \   b=\overline{E}.\label{dist_02}
\end{eqnarray}
Since $E \in \Sigma_{V,\alpha}$, it is easy to see that
\begin{align*}
&|G_{k_0}(V) \cap (E-\epsilon,E+\epsilon)|\leq \epsilon,\\
&|(-\infty, \underline{E})\cap(E-\epsilon, E+\epsilon)|\leq \epsilon, \quad {\rm if} \;\ a=\underline{E}, \\
&|(\overline{E}, +\infty)\cap(E-\epsilon, E+\epsilon)|\leq \epsilon,\quad {\rm if} \;\  b=\overline{E}.
\end{align*}

Without loss of generality, assume that $\epsilon< \frac{b-a}{2}$. Then $[a,b] \not\subset (E-\epsilon, E+\epsilon)$. We consider the following three cases.

\textbf{ Case 1. $(E-\epsilon, E+\epsilon)\subset [a,b]$.}    By the definition of ${\CN}$, we have
$${\rm dist}(G_k(V),G_{k_0}(V) )\leq 2\epsilon,\quad \forall \  k\in {\CN}.$$
Combining with (\ref{dist_1}), we get $|k|\geq \frac{D}{2\vartheta}\left|\ln\frac{2\epsilon}{C_{12}}\right|$ for any $k\in{\CN}\backslash\{k_0\}$. Thus,
$$\sum\limits_{k \in {\CN}\backslash \{k_0\}} |G_k(V) \cap (E-\epsilon,E+\epsilon)|
\leq C \sum_{|k|\geq \frac{D}{2\vartheta}\left|\ln\frac{2\epsilon}{C_{12}}\right|} e^{-\vartheta|k|}\leq \epsilon^{\frac{D+2}4},$$
provided that $0<\epsilon\leq\epsilon_1$ for some $\epsilon_1=\epsilon_1(V, \alpha, \sigma, C, \vartheta, d)>0$ (but independent of the choice of $E$).
So we have
\begin{align}
 & |(E-\epsilon,E+\epsilon)\cap \Sigma_{V,\alpha}|\nonumber\\
\geq&\ 2 \epsilon - |G_{k_0}(V)\cap(E-\epsilon, E+\epsilon)|
 -\, \sum_{k \in {\CN}\backslash\{k_0\} } |G_k(V) \cap (E-\epsilon,E+\epsilon)|\nonumber\\
\geq&\ 2\epsilon  - \epsilon -  \epsilon^{\frac{D+2}4}\nonumber\\
\geq&\ \frac34 \epsilon, \quad   \forall \   0<\epsilon\leq \epsilon_1.\label{local_homo_0}
\end{align}

\textbf{ Case 2. $(E-\epsilon, E+\epsilon)\cap (-\infty,a)\neq\emptyset$.}
In this case, one has
$$|E_{k}^--a|\leq 2\epsilon,\quad \forall \  k\in {\CN}.$$
We need to distinguish two cases:
if $a=\underline{E}$, then by  $(\ref{dist_01})$, we get $|k|\geq \frac{D}{\vartheta}\left|\ln\frac{2\epsilon}{C_{12}}\right|$.
If $a=E^{+}_{m}$, then by (\ref{dist_1}), if $\epsilon$ is small enough (the smallness depends on $m$), we have
$$|k|\geq \frac{D}{\vartheta}\left|\ln\frac{2\epsilon}{C_{12}}\right|-|m| \geq  \frac{D}{2\vartheta}\left|\ln\frac{2\epsilon}{C_{12}}\right|.$$
Hence, if $0<\epsilon\leq \epsilon_2$ for some $\epsilon_2=\epsilon_2(a, V, \alpha, \sigma, C, \vartheta, d)$, then we have
$$\sum\limits_{k \in {\CN}} |G_k(V) \cap (E-\epsilon,E+\epsilon)|
\leq C \sum_{k \in {\CN}} e^{-\vartheta|k|}\leq \epsilon^{\frac{D+2}4}.$$
So we have
\begin{align}
 & |(E-\epsilon,E+\epsilon)\cap \Sigma_{V,\alpha}|\nonumber\\
  \geq  &\ |(E-\epsilon,E+\epsilon)\cap \Sigma_{V,\alpha}\cap [a,b]|\nonumber\\
\geq&\ 2 \epsilon - |(-\infty, a)\cap(E-\epsilon, E+\epsilon)|
 -\, \sum_{k \in {\CN} } |G_k(V) \cap (E-\epsilon,E+\epsilon)|\nonumber\\
\geq&\ 2\epsilon  - \epsilon -  \epsilon^{\frac{D+2}4}\nonumber\\
\geq&\ \frac34 \epsilon,\quad \forall  \  0<\epsilon\leq \epsilon_2.\label{local_homo_1}
\end{align}

\textbf{ Case 3. $(E-\epsilon, E+\epsilon)\cap (b,+\infty)\neq\emptyset$.}
Similarly to the above case, there exists $\epsilon_3=\epsilon_3(b, V, \alpha, \sigma, C, \vartheta, d)>0$ such that
\begin{equation}\label{local_homo_2}
|(E-\epsilon,E+\epsilon)\cap \Sigma_{V,\alpha}|\geq \frac34 \epsilon, \quad \forall  \  0<\epsilon\leq \epsilon_3.
\end{equation}

Let $\epsilon_0:=\min\{\frac{b-a}{2}, \epsilon_1,\epsilon_2, \epsilon_3\}$. By (\ref{local_homo_0}) -- (\ref{local_homo_2}), for any $E\in\Sigma_{V,\alpha}\cap [a,b]$, we have
$$|(E-\epsilon,E+\epsilon)\cap \Sigma_{V,\alpha}|\geq \frac34 \epsilon, \quad \forall  \  0<\epsilon\leq \epsilon_0.$$
As for the case $\epsilon\in (\epsilon_0, {\rm diam}\Sigma_{V,\alpha})$, we have
$$|(E-\epsilon,E+\epsilon)\cap \Sigma_{V,\alpha}|\geq |(E-\epsilon_0,E+\epsilon_0)\cap \Sigma_{V,\alpha}|\geq \frac34 \epsilon_0\geq \frac{3\epsilon_0}{4\, {\rm diam}\Sigma_{V,\alpha}} \cdot \epsilon,$$
which completes the proof.
\end{proof}

\subsection{Applications of the criterion}

Let us consider the spectrum of $H_{V,\alpha,\theta}$ with $\alpha\in\R$ satisfying $\beta(\alpha)=0$, and $V\in C^\omega(\T,\R)$. Split $\Sigma_{V,\alpha}$ into $\Sigma^{\rm sup}_{V,\alpha}\cup\Sigma^{\rm sub}_{V,\alpha}$ as presented in Theorem \ref{global-red}.
As an application of Theorem \ref{thm_homo_spec}, we can get the homogeneity of $\Sigma_{V, \alpha}^{\rm sub}=\bigcup_{i} (\Sigma_{V, \alpha}\cap I_i)$ after showing the $\frac12-$H\"older continuity of the integrated density of states $N=N_{V,\alpha}$ in the global subcritical regime.

\begin{prop}\label{thm_holder}
If $\beta(\alpha)=0$, then IDS is $\frac{1}{2}-$H\"older continuous on $I_i$, $1\leq i \leq m$.
\end{prop}

In the non-perturbative regime considered in Section \ref{subs_duality}, the above result has been shown in \cite{A1} and \cite{AvilaJito}.
In the global subcritical regime, the result was assumed to be known, however,  we could not find a reference in the literature. For completeness, we give a proof in Appendix \ref{Appendix_B}.

\begin{proof}[Proof of Theorem \ref{theo calibration gaps bands}] Let us focus on $E\in \Sigma_{V,\alpha}\cap I_i$ such that the corresponding cocycle is subcritical. Set $I_i:=[a'_i,b'_i]$. By Theorem \ref{global-red},  $a'_i=\underline{E}$ or $E_{m_i}^+$ for some $m_i\in\Z\backslash\{0\}$ and $b'_i=\overline{E}$ or $E_{n_i}^-$ for some $n_i\in\Z\backslash\{0\}$.

Assertion (1) in Theorem \ref{theo calibration gaps bands} was already shown in Corollary \ref{cor_global}.
Since $\alpha\in \R\backslash \Q$ with $\beta(\alpha)=0$,  hypothesis (H2) holds with $C$ and $\vartheta$ given by Corollary \ref{cor_global}. Moreover, (H1) follows from Proposition \ref{thm_holder} with $\sigma=\frac12$, while (H3) holds automatically, since  $\vartheta> 0=\beta(\alpha)$. Hence, by applying Theorem \ref{thm_homo_spec}, we get the assertion (2) of Theorem \ref{theo calibration gaps bands}, and there exists $\mu_i=\mu_i(a'_i,b'_i,V, \alpha)\in (0,1)$, such that
$$| (E-\epsilon,E+\epsilon) \cap \Sigma_{V, \alpha}| > \mu_i\epsilon,\quad \forall \  E\in \Sigma_{V, \alpha}\cap [a'_i,b'_i], \;\  \forall \  0<\epsilon\leq {\rm diam}\Sigma_{V, \alpha}.$$
Taking $\mu_0:=\min_{1\leq i\leq m}\{\mu_i\}$, we can show
assertion (3) of Theorem \ref{theo calibration gaps bands}. Note that the structure of $\Sigma_{V,\alpha}$ is uniquely determined by $V$ and $\alpha$,  thus $\mu_0$  only depends on $V$ and $\alpha$.
\end{proof}

\begin{proof}[Proof of Theorem \ref{theorem_homo_spec}]
For $\alpha \in \mathrm{SDC}$ as defined in \eqref{strongdiophantinecondition}, we have that $\Sigma_{V,\alpha}^{\rm sup}$ is $\mu_1-$homogeneous for some $\mu_1\in(0,1)$ in view of Theorem H in \cite{dgsv}.
In particular, $\alpha \in \mathrm{SDC}$ implies that $\beta(\alpha)=0$.
Combining with the homogeneity of $\Sigma_{V,\alpha}^{\rm sub}$ (Theorem \ref{theo calibration gaps bands} (3)), we can prove directly that $\Sigma_{V,\alpha}=\Sigma_{V,\alpha}^{\rm sub}\cup\Sigma_{V,\alpha}^{\rm sup}$ is $\mu-$homogeneous for  some $0<\mu\leq \min\{\mu_0, \, \mu_1\}$.\end{proof}

As another application of the criterion, we prove the homogeneity of spectrum for the noncritical almost Mathieu operator $H_{\lambda,\alpha,\theta}$.

\begin{proof}[Proof of Theorem \ref{homo-amo}]
By Aubry duality, it is enough for us to consider the case where $0<\lambda<1$.  If $\beta(\alpha)=0$, assertion (1) was shown in Corollary \ref{cor_mathieu_upperbound},  which implies the hypothesis (H2) for $[a,b]=[\underline{E},\overline{E}]$ with $\vartheta= -\frac{\ln \lambda}{4}$, while the hypothesis (H3) holds automatically since $\vartheta>0=\beta(\alpha)$.  Moreover, the $\frac12-$H\"older continuity of the integrated density of states on $[\underline{E},\overline{E}]$ was shown in  Corollary 3.10 of \cite{A1}.
So we can apply Theorem \ref{thm_homo_spec} and get assertions (2) and (3).
\end{proof}

\section{Deift's conjecture -- Proof of Theorem \ref{thm_deift_conjecture}}

To consider Toda lattice equation (\ref{Toda}) or equivalently the Lax pair (\ref{Toda_Lax}),
let us recall some basic notions and results about the almost periodic Jacobi matrix in the work of Sodin-Yuditskii \cite{SY95, SY97}.

Consider a self-adjoint almost periodic Jacobi matrix $J$:
\begin{equation}\label{Jacobi_matrix}
(Ju)_n=a_{n-1}u_{n-1}+b_nu_n+a_n u_{n-1},
\end{equation}
with a compact spectrum $\Sigma=[\inf\Sigma,\sup\Sigma]\setminus\bigcup_{k\in\Z}(E_k^-,E_k^+).$
Assume that $\Sigma$ is homogeneous, and let ${\CJ}_\Sigma$ be the class of reflectionless Jacobi matrices with spectrum $\Sigma$.
Let $\pi_1(\C\backslash\Sigma)$ be the fundamental group of $\C\backslash\Sigma$.
This is a free group admitting a set of generators  $\{c_k\}_{k\in\Z}$, where $c_k$ is a counterclockwise simple loop intersecting $\R$ at $\inf\Sigma-1$ and $\frac12(E_k^{+}+E_k^{-})$.
Then, consider the Abelian compact group
$\pi_*(\C\backslash\Sigma)$ of unimodular characters on $\pi_1(\C\backslash\Sigma)$.
Here a character means a function ${\CK}\colon\pi_1(\C\backslash\Sigma)\rightarrow\T$ satisfying
$$ {\CK}(\gamma_1 \gamma_2)={\CK}(\gamma_1)\, {\CK}(\gamma_2), \quad  \gamma_1, \, \gamma_2 \in \pi_1(\C\backslash\Sigma).$$
An element ${\CK}\in \pi_*(\C\backslash\Sigma)$ is uniquely determined by its action on loops $c_k$, so we can write ${\CK} =({\CK}(c_k))_{k\in\Z}= (e^{2\pi i \tilde{K}_k})_{k\in\Z}$.

\begin{theorem}[Sodin-Yuditskii \cite{SY97}]\label{theorem_SY97}
There is a continuous one-to-one correspondence between almost periodic Jacobi matrices $J\in {\CJ}_{\Sigma}$ and characters ${\CK}\in\pi_*(\C\backslash \Sigma)$.
\end{theorem}

If we identify the Jacobi matrix $J$ given in (\ref{Jacobi_matrix}) with $(a,b)\in\ell^\infty(\Z)\times\ell^\infty(\Z)$, then by Theorem \ref{theorem_SY97},
there exists a continuous map ${\CH}\colon\T^{\Z}\to \ell^\infty(\Z)\times\ell^\infty(\Z)$ such that for any $J\in {\CJ}_{\Sigma}$ given as in (\ref{Jacobi_matrix}), one can find a unique ${\CK}\in\pi_*(\C\backslash \Sigma)$ such that
\begin{equation}\label{one-to-one_map}
(a,b)={\CH}\left(\left({\CK}(c_k)\right)_{k\in\Z}\right)={\CH}\left(\left(e^{-{2\pi{\rm i}\tilde{K}_k}}\right)_{k\in\Z}\right).
\end{equation}

 \smallskip

Now we consider the Lax pair (\ref{Toda_Lax}) in a more general form. Given any $f\in L^\infty(X,\R)$ with $\Sigma\subset X$, we define the infinite-dimensional matrix $f(J)$ in the sense of standard functional calculus and decompose it into $
f^+(J)+f^-(J)$,
the sum of an upper triangular matrix $f^+(J)$ and a lower triangular matrix $f^-(J)$. We also set
$
M_f(J):=f^+(J)-f^-(J).
$
Then, given an almost periodic Jacobi matrix $J_0\in{\CJ}_{\Sigma}$, we  define the Lax pair
\begin{equation}\label{Toda_Lax_proof}
\frac{d}{dt}J(t)=[M_f(J(t)),J(t)], \quad J(0)=J_0.
\end{equation}
\begin{theorem}[Vinnikov-Yuditskii \cite{VY}]\label{theorem_VY}
Assume that $\Sigma$ is homogeneous and the almost periodic Jacobi matrix $J_0\in{\CJ}_\Sigma$ has purely absolutely continuous spectrum.
Given $f \in L^\infty(X,\R)$ with $\Sigma \subset X$, the following holds.
\begin{enumerate}
  \item  There exists a unique solution $J=J(t)$ of (\ref{Toda_Lax_proof}), well-defined for all $t \in \R$. Moreover, for every $t$, $J(t)$ is an almost periodic Jacobi matrix with constant spectrum $\Sigma$.
  \item For $t\in\R$, let ${\CK}^t\in \pi_*(\C\backslash\Sigma)$ be the character corresponding to $J(t)$. There exists a homomorphism $\xi\colon\pi_1(\C\backslash\Sigma)\rightarrow \R$, depending on $f$, such that ${\CK}^t(c_k)={\CK}^0(c_k) \, e^{-2\pi {\rm i} t\xi(c_k)}$.
\end{enumerate}
\end{theorem}

Obviously, with $f(x)=x$ and assuming that all the diagonal elements of $M_f(J)$ vanish, we get the Lax pair (\ref{Toda_Lax}), which is equivalent to the Toda flow (\ref{Toda}).
By the assertion (2) of Theorem \ref{theorem_VY}, combining with (\ref{one-to-one_map}), we get
 $$(a(t),b(t))={\CH}\left(\left(   e^{-2\pi {\rm i} \tilde{K}_k^t}  \right)_{k\in\Z}\right)={\CH}\left(\left(   e^{-2\pi {\rm i} \left[\tilde{K}_k^0 + t \xi(c_k)\right]} \right)_{k\in\Z}\right),$$
which implies the time almost periodicity of solutions of (\ref{Toda}).

\smallskip

Now we are going to prove Theorem \ref{thm_deift_conjecture}.
Let $V\in C^{\omega}(\T,\R)$ be subcritical and $\alpha\in\R\backslash\Q$ with $\beta(\alpha)=0$.
By Kotani's theory \cite{Kot84}, for almost every $\theta\in\T$, for almost every $E$ such that $L_{V,\alpha}(E) = 0$, we have $m^+_{H_{V,\alpha,\theta}}(E) = -\overline{m^-_{H_{V,\alpha,\theta}}}(E)$. It was later improved in Theorem 2.2 of \cite{A1}, where it is shown that the above assertion is true for every $\theta\in\T$.
By (\ref{green_m+-}), we have that $H_{V,\alpha,\theta}$ is reflectionless for every $\theta\in\T$.
Moreover, it follows from Theorem \ref{theo calibration gaps bands} that $\Sigma_{V,\alpha}$ is homogeneous. Thus, by Theorem \ref{theorem_VY}, it is sufficient to verify the purely absolute continuity of spectrum.

\begin{theorem}[Avila \cite{A2}] \label{thm_suncritical_ac}
If $\beta(\alpha)=0$ and $V\in C^{\omega}(\T,\R)$ is subcritical, then the spectrum of the operator $H_{V,\alpha,\theta}$ is purely absolutely continuous.
\end{theorem}

Theorem \ref{thm_suncritical_ac} was proved from the viewpoint of dynamics.
Roughly speaking, in view of  Theorem \ref{global-red}, we can transform the corresponding Schr\"odinger cocycle into the ``non-perturbative regime" (Proposition \ref{prop_ar_1}), for which the purely absolute continuity has been shown in \cite{A1}.

Theorem \ref{thm_suncritical_ac} can also be shown by inverse spectral theory.
Assuming  finite total gap length, homogeneity of the spectrum together with the reflectionless condition, Gesztesy-Yuditskii \cite{GY} have shown that the corresponding spectral measure is purely absolutely continuous. Then, Theorem \ref{thm_suncritical_ac} follows from  Theorem \ref{theo calibration gaps bands}.
\appendix

\section{Proof of Theorem \ref{thm_almost_almost-1}}\label{Appendix_A}

Given $\theta \in \R$ and $\epsilon_0>0$, we denote by $\{n_l\}_l$  the  set of $\epsilon_0-$resonances of $\theta$, i.e.,
$$
\|2 \theta - n_l \alpha\|_\T \leq e^{-\epsilon_0|n_l|},\quad \text{and}\quad \|2 \theta - n_l \alpha\|_\T=\min_{|m|\leq |n_l|} \|2 \theta - m \alpha\|_\T.
$$

Let $\lambda>1$. By \cite{AvilaJito}, the family $\{H_{\lambda,\alpha,\theta}\}_\theta$ is almost localized. Fix $\theta \in \R$, and let $u=(u_j)_{j\in\Z}$ be a generalized solution to $ H_{\lambda,\alpha,\theta}  u= E  u$, with $ u_0=1$ and $| u_{j}|\leq 1$ for all $j \in \Z$.
Given an interval $I=[i_1,i_2] \subset \Z$ of length $N\geq 0$, we denote by $G_I$ the Green's function $(x,y)\mapsto (H_{ \lambda,\alpha,\theta}-E)^{-1}(x,y)$ restricted to $I$ with zero boundary conditions at $i_1-1$ and $i_2+1$. Then for any $j \in I$, we have
\begin{equation}\label{exrpe green}
 u_j=-G_I(i_1,j)  u_{i_1-1}-G_I(j,i_2)  u_{i_2+1}.
\end{equation}
Let us denote by $P_m(\theta)$ the upper-left coefficient of the $m^{th}$ iterate $(m\alpha,\mathcal{A}_{m}(E))$ of the cocycle $(\alpha,S_E^\lambda)$. Then by Cramer's rule, we have
\begin{equation*}\label{cramerrule}
|G_I(i_1,j)|=\left| \frac{P_{i_2-j}(\theta+(j+1)\alpha)}{P_N(\theta+ i_1 \alpha)}\right|,\quad |G_I(j,i_2)|=\left| \frac{P_{j-i_1}(\theta+i_1 \alpha)}{P_N(\theta+ i_1 \alpha)}\right|.
\end{equation*}
Given $\xi > 0$ and $m\in \N$, we say that $y \in \Z$ is $(\xi,m)-$\textit{regular} if there exists an interval $J=[x_1,x_2]\subset \Z$ of length $m$ such that $y \in J$ and
$$
|G_{J}(y,x_i)| < e^{-\xi |y-x_i|},\quad |y-x_i| \geq \frac{1}{7}m,\quad i=1,2.
$$

Recall that $L(\alpha,S_E^\lambda)=\ln \lambda$ for any energy $E\in \Sigma_{\lambda,\alpha}$. By subadditivity, 
for any $\eta>0$, any $E' \in \Sigma_{\lambda,\alpha}$, and for  $m \geq 0$  large enough, we have 
$|\mathcal{A}_{m}(E')|_\T \leq e^{ (\ln \lambda- \eta)m}$. In particular,  $P_m(\theta) \leq e^{ (\ln \lambda- \eta)m}$.

Let $(q_i)_{i \geq 1}$ be the sequence of denominators of best approximants of $\alpha$. We associate with any integer $C_0 |n_l| < |j| < C_0^{-1} |n_{l+1}|$ scales $\ell \geq 0$ and $s\geq 1$ so that
$$
2 s q_\ell \leq\zeta j <   \min(2(s+1)q_\ell,2q_{\ell+1}),
$$
where $\zeta:=\frac{1}{32}$ if $2|n_l|< j < 2^{-1} |n_{l+1}|$, and $\zeta:=\frac{C_0-1}{16C_0}$ otherwise.
We set
\begin{itemize}
\item $I_1:=[-2s q_\ell+1,0]$ and $I_2:=[j-2 s q_\ell +1,j+2 sq_\ell]$ if $j<|n_{l+1}|/3$, $n_{l}\geq 0$.
\item $I_1:=[1,2s q_\ell]$ and $I_2:=[j-2 s q_\ell +1,j+2 sq_\ell]$ if $j<|n_{l+1}|/3$ and $n_{l}<0$.
\item $I_1:=[-2s q_\ell+1,2s q_\ell]$ and $I_2:=[j-2 s q_\ell +1,j]$ if $|n_{l+1}|/3\leq j< |n_{l+1}|/2$.
\item $I_1:=[-2s q_\ell+1,2s q_\ell]$ and $I_2:=[j+1,j+2 sq_\ell]$ if $j\geq |n_{l+1}|/2$.
\end{itemize} In particular, the total number of elements in $I_1 \cup I_2$ is $6 sq_\ell$.
Fix $\delta>0$ arbitrary. If $\epsilon_0>0$ is chosen sufficiently small, 
then in view of $\beta(\alpha)=0$, Lemma 5.8 in \cite{AvilaJito} implies that there exists an integer $j_0=j_0(C_0,\alpha,n,\delta)>0$ such that for $j> j_0$, the set $\{\theta_m:=\theta+m \alpha\}_{m \in I_1 \cup I_2}$ is $\delta-$\textit{uniform}, i.e.,
$$
\max\limits_{z \in [-1,1]}\max\limits_{m \in I_1 \cup I_2} \prod\limits_{m\neq p\in I_1 \cup I_2} \frac{|z-\cos (2 \pi \theta_p)|}{|\cos (2 \pi \theta_m)-\cos(2 \pi \theta_p)|}< e^{(6sq_\ell-1)\delta}.
$$
Following the proof of Lemma 5.4 in \cite{AvilaJito}, we conclude that
for any $\eta>0$, there exists $j_1=j_1(C_0,\alpha,\lambda,\eta)>0$ such that any $j> j_1$ is $(\ln \lambda-\eta,6 s q_\ell-1)-$regular.


\begin{proof}[Proof of Theorem \ref{thm_almost_almost-1} (2)] We consider the case that $\theta$ is $\epsilon_0-$resonant.
We will show that the sequence $( u_j)_j$ decays exponentially in some suitable interval between two consecutive resonances, with a rate close to the Lyapunov exponent $L(\alpha,S_E^{\lambda})=\ln\lambda$.
By the condition $\beta(\alpha)=0$, we know that $|n_l|=o(|n_{l+1}|)$. Let us fix some small $\eta>0$. Given $l>0$ sufficiently large, take $\ell>0$ such that $2q_\ell \leq \zeta(2C_0 |n_l|+1)<2q_{\ell+1}$, and let $2C_0 |n_l|+\eta|n_{l+1}| \leq |j| \leq (2C_0)^{-1} |n_{l+1}|$. We set $b_l:=2 C_0 |n_l|+1$.
Then
for any $y \in [b_l,2j]$, there exists an interval $I(y)=[x_1,x_2]\subset \Z$ 
with $y \in I(y)$ and
$$
\mathrm{dist}(y,\partial I(y))\geq \frac{1}{7}|I(y)|\geq \frac{6 q_{\ell}-1}{7}\geq \frac{q_{\ell}}{2},
$$
where $\partial I(y):=\{x_1,x_2\}$, and such that
$$
|G_{I(y)}(y,x_i)|\leq e^{-(\ln \lambda-\eta)|y-x_i|}\leq e^{-(\ln \lambda-\eta)\frac{q_{\ell}}{2}},\quad i=1,2.
$$
For $z \in \partial I(y)$,  we denote by $z'$ the neighbour of $z$ not belonging to $I(y)$. If $x_2+1<2 j$ or $x_1-1 >b_l$, we can expand $ u_{x_2+1}$ or $u_{x_1-1}$ following \eqref{exrpe green}, with $I=I(x_2+1)$ or $I=I(x_1-1)$. We continue to expand each term until we arrive to $\widetilde z$ such that either $\widetilde z\leq b_l$, or  $\widetilde z> 2j$, or the number 
of $G_I$ terms in the following product becomes $\lfloor \frac{2j}{q_{\ell}}\rfloor$, whichever comes first:
$$
 u_j=\sum\limits_{r,\ z_{i+1} \in \partial I(z_i')} G_{I(j)}(j,z_1)G_{I(z_1')}(z_1',z_2)\dots G_{I(z_r')}(z_r',z_{r+1}) u_{z_{r+1}'}.
$$
In the first two cases, we estimate
\begin{align*}
&|G_{I(j)}(j,z_1)G_{I(z_1')}(z_1',z_2)\dots G_{I(z_r')}(z_r',z_{r+1}) u_{z_{r+1}'}|\\ \leq\ &e^{-(\ln \lambda-\eta)(|j-z_{1}|+\sum_{i=1}^r |z_i'-z_{i+1}|)}\\
\leq\ &e^{-(\ln \lambda-\eta)(|j-z_{r+1}|-(r+1))} \\
\leq\ &\max\big(e^{-(\ln \lambda-\eta)(j-b_l-\frac{2j}{q_{\ell}})},e^{-(\ln \lambda-\eta)(2j-j-\frac{2j}{q_{\ell}})}\big),\\
\leq\ &e^{-(\ln \lambda-\eta)(j+o(j))},
\end{align*}
where we have used that $|b_l| = o(|j|)$, while in the third case, we have
$$
|G_{I(j)}(j,z_1)G_{I(z_1')}(z_1',z_2)\dots G_{I(z_r')}(z_r',z_{r+1}) u_{z_{r+1}'}|\leq e^{-(\ln \lambda-\eta)\frac{q_{\ell}}{2}\lceil \frac{2j}{q_{\ell}}\rceil}.
$$
Fix $\delta>0$ arbitrarily small.
By taking $|j|$ to be sufficiently large, resp. $\eta$ small enough in the previous expression, we conclude that $| u_j| \leq e^{-(\ln \lambda-\delta) |j|}$ for $|j|$ large enough with $2C_0 |n_l|+\eta|n_{l+1}| \leq |j| \leq (2C_0)^{-1} |n_{l+1}|$.\end{proof}

\begin{proof}[Proof of Theorem \ref{thm_almost_almost-1} (1)]
We consider the other case, i.e., when $\theta$ is not $\epsilon_0-$resonant. Denote by $n$ its last $\epsilon_0-$resonance, set $b:=2 C_0 |n|+1$ and let $|j | \geq b$. Let us fix some small $\eta>0$.
Then
for any $y \in [b,2j]$, there exists an interval $I(y)=[x_1,x_2]\subset \Z$ 
with $y \in I(y)$ and
$$
\mathrm{dist}(y,\partial I(y))\geq \frac{1}{7}|I(y)|\geq \frac{6 q_{\ell}-1}{7}\geq \frac{q_{\ell}}{2},
$$
where $\partial I(y):=\{x_1,x_2\}$, and such that
$$
|G_{I(y)}(y,x_i)|\leq e^{-(\ln \lambda-\eta)|y-x_i|}\leq e^{-(\ln \lambda-\eta)\frac{q_{\ell}}{2}},\quad i=1,2.
$$
As previously, we can expand $ u_{x_2+1}$ or $u_{x_1-1}$ following \eqref{exrpe green}, with $I=I(x_2+1)$ or $I=I(x_1-1)$. We continue to expand each term until we arrive to $\widetilde z$ such that either $\widetilde z\leq b$, or  $\widetilde z> 2j$, or the number 
of $G_I$ terms in the following product becomes $\lfloor \frac{2j}{q_{\ell}}\rfloor$, whichever comes first:
$$
 u_j=\sum\limits_{r,\ z_{i+1} \in \partial I(z_i')} G_{I(j)}(j,z_1)G_{I(z_1')}(z_1',z_2)\dots G_{I(z_r')}(z_r',z_{r+1}) u_{z_{r+1}'}.
$$
In the first two cases, we estimate
\begin{align*}
&|G_{I(j)}(j,z_1)G_{I(z_1')}(z_1',z_2)\dots G_{I(z_r')}(z_r',z_{r+1}) u_{z_{r+1}'}|\\ \leq\ &e^{-(\ln \lambda-\eta)(|j-z_{1}|+\sum_{i=1}^r |z_i'-z_{i+1}|)}\\
\leq\ &e^{-(\ln \lambda-\eta)(|j-z_{r+1}|-(r+1))} \\
\leq\ &\max\left(e^{-(\ln \lambda-\eta)(j-b-\frac{2j}{q_{\ell}})},e^{-(\ln \lambda-\eta)(2j-j-\frac{2j}{q_{\ell}})}\right),\\
\leq\ &e^{-(\ln \lambda-\eta)(j+o(j))},
\end{align*}
while in the third case, we have
$$
|G_{I(j)}(j,z_1)G_{I(z_1')}(z_1',z_2)\dots G_{I(z_r')}(z_r',z_{r+1}) u_{z_{r+1}'}|\leq e^{-(\ln \lambda-\eta)\frac{q_{\ell}}{2}\lfloor \frac{2j}{q_{\ell}}\rfloor}.
$$
Fix $\delta>0$ arbitrarily small.
By taking $|j|$ be sufficiently large, resp. $\eta$ small enough in the previous expression, we conclude that  $| u_j| \leq e^{-(\ln \lambda-\delta) |j|}$ for $|j|$ large enough.
\end{proof}

\section{Proof of Proposition \ref{thm_holder}}\label{Appendix_B}

The proof follows Theorem 1.6 of \cite{AvilaJito} (see also Corollary 3.10 of \cite{A1}),  the key points are the quantitative almost reducibility results and Thouless formula.

If $\beta(\alpha)=0,$ then by Proposition \ref{prop_ar_1} (see also Corollary \ref{prop_duality_global}), there exists $0<h_1=h_1(V,\alpha)<1$, such that for any  $E\in\Sigma_{V,\alpha}\cap I_i$, $1\leq i \leq m$,  there exists $\Phi_E\in C^\omega(\T,{\rm PSL}(2,\R))$ with $|\Phi_E|_{h_1}<\Lambda=\Lambda(V,\alpha, c_0 h_1^{3}, h_1)$, $E_*=E_*(E)$ locally constant, and $V_{*}=V_*(E) \in C^\omega_{h_1}(\T,\R)$, $|V_*|_{h_1}<c_0 h_1^{3}$, such that
$$
\Phi_E(\cdot+\alpha)^{-1} S_E^V(\cdot) \Phi_E(\cdot)=S_{E_*}^{V_*}(\cdot),
$$
where $c_0> 0$ is the absolute constant given in Theorem \ref{almostredth}. In particular, the family $\{\widehat H_{V_*,\alpha,\theta}\}_{\theta\in \T}$ is almost localized.


Therefore, by 
Theorem 3.8 of \cite{A1},
 there exist a phase $\theta'=\theta'(E) \in \T$ and positive constants $C=C(\alpha,h_1)$,  $c=c(\alpha,h_1)$, $\epsilon_0=\epsilon_0(h_1)$ such that the following is true.
Let $\{n_j\}_{j}$ be the set of $\epsilon_0-$resonances of  $\theta'$, ordered in such a way that $|n_j| \leq |n_{j+1}|$. For any small $\varepsilon>0$, take $j$ such that $e^{-c{\CN}}\leq\varepsilon\leq e^{-o(n)}$, with $n:=|n_j|+1$ and ${\CN}:=|n_{j+1}|$ (if defined, otherwise ${\CN}:=+\infty$).
By composing $\Phi_E$ with the conjugacy $B$ given by
Theorem 3.8 of \cite{A1},
 and noting that $\Phi_E$ is uniformly bounded, we get $\Psi:=\Phi_E B\in C_c^\omega(\T, {\rm PSL}(2,\C))$
satisfying $|\Psi|_{c}\leq e^{o(n)}$, such that
$$\Psi(\cdot+\alpha)^{-1} S_{E}^{V}(\cdot) \Psi(\cdot)=
\begin{pmatrix}
e^{2 \pi{\rm i}\theta'} & 0 \\
0 & e^{-2 \pi{\rm i}\theta'}
\end{pmatrix}
+
\begin{pmatrix}
q_1(\cdot) & q(\cdot) \\[1mm]
q_3(\cdot) & q_4(\cdot)
\end{pmatrix},$$
with  $|q_1|_{c},\, |q_3|_{c},\, |q_4|_{c}\leq  Ce^{-c{\CN}}$ and $|q|_{c}\leq Ce^{-cn}$.
Let $D:=\begin{pmatrix}
d^{-1} & 0\\
0 & d
\end{pmatrix}$ with $d:= \varepsilon^{\frac14} |\Psi|_c$, and set $W:=\Psi D\in C_c^{\omega}(2\T, {\rm SL}(2,\C))$. It follows from the bounds on $\Psi$ and $\varepsilon$ that $|W|_{c}\leq C'\varepsilon^{-\frac14}$ for some uniform constant $C'>0$. Hence, for
$$U_\varepsilon(\cdot):=W(\cdot+\alpha)^{-1} S_{E+\mathrm{i}\varepsilon}^{V}W(\cdot)=W(\cdot+\alpha)^{-1} \left[S_{E}^{V}(\cdot)+  \begin{pmatrix}
\mathrm{i} \varepsilon & 0\\
0 & 0
\end{pmatrix}\right ]W(\cdot),$$
we get $|U_\varepsilon|_{c}\leq 1+ C''\varepsilon^{\frac12}$ for some uniform constant $C''>0$.
As a result, we obtain the following estimate on the Lyapunov exponent:
\begin{equation}\label{distance_subcritical}
L(\a, S^V_{E+i \varepsilon})= L(\alpha,U_\varepsilon)\leq \ln |U_\varepsilon|_c \leq C'' \varepsilon^{\frac12}.
\end{equation}
The above conclusions are similar to Theorem 4.4 and Corollary 4.6 in \cite{AvilaJito}, and we refer to them for more details.

On the other hand, by Thouless formula,  there exists a constant $c'>0$ such that for any $\varepsilon>0$,
\begin{align*}
L(\a, S^V_{E+i \varepsilon})&= L(\a, S^V_{E+i \varepsilon})-L(\a, S^V_{E})\\ &=\frac12 \int \ln \left(1+\frac{\varepsilon^2}{(E-E')^2}\right)dN_{V,\alpha}(E')\\ &\geq c'(N_{V,\alpha}(E+\varepsilon)-N_{V,\alpha}(E-\varepsilon)).
\end{align*}
Combining the last estimate with \eqref{distance_subcritical}, we deduce that
$$N_{V,\alpha}(E+\varepsilon)-N_{V,\alpha}(E-\varepsilon)\leq C'' c'^{-1}\varepsilon^{\frac12}$$ for
$E\in \Sigma_{V,\alpha}\cap I_i$ and $0<\varepsilon<1$ such that $[E-\varepsilon,E+\varepsilon]\subset I_i$.
Since $N_{V,\alpha}$ is locally constant on the complement of $\Sigma_{V,\alpha}$, we have that $N_{V,\alpha}$ is $\frac12-$H\"older on $I_i$.

\section{Acknowledgements}
We would like to thank A. Avila, M. Goldstein and S. Jitomirskaya for useful discussions. During his PhD, M. Leguil was supported by IMJ-PRG/Universit\'e Paris 6/7 and by ``ANR-15-CE40-0001-03" for the project ``BEKAM". J. You  was partially supported by NSFC grant (11471155) and
973 projects of China (2014CB340701). Z. Zhao was supported by ANR grant ``ANR-15-CE40-0001-03" for the project ``BEKAM". Q. Zhou was partially supported by "Deng Feng Scholar Program B" of Nanjing University, Specially-appointed professor programe of Jiangsu province and NSFC grant (11671192).

\end{document}